\documentclass[12pt,reqno]{amsart}
\usepackage{tikz,dsfont}
\usepackage{amsmath, amssymb, graphicx, mathrsfs, hyperref}
\usepackage{pdfpages}
\usepackage{verbatim}

\usepackage{enumerate}
\newcommand{\Z}{\mathbb{Z}}

\newcommand{\Q}{\mathbb{Q}}

\newcommand{\lcm}{\text{\rm lcm}}

\newtheoremstyle{dotless}{}{}{\itshape}{}{\bfseries}{}{ }{} 
\theoremstyle{}

\newtheorem{theorem}{Theorem}

\newtheorem{lemma}{Lemma}[section]

\newtheorem{prop}{Proposition}

\newtheorem{corollary}{Corollary}

\newtheorem{remark}{Remark}

\title{Classifying linear division sequences}
\author{Andrew Granville}

 \address{D{\'e}partment  de Math{\'e}matiques et Statistique,   Universit{\'e} de Montr{\'e}al, CP 6128 succ Centre-Ville, Montr{\'e}al, QC  H3C 3J7, Canada.}
   \email{andrew.granville@umontreal.ca}

\dedicatory{In memory of Richard K.~Guy, 1916-2020}

\thanks{Many thanks to Yann Bugeaud, Pietro Corvaja, Henri Darmon, Youness Lamzouri, Florian Luca, Alina Ostafe, Attila Peth\H{o}, Carl Pomerance, Hugh C. Williams and Umberto Zannier for useful email conversations, and particularly Carlo Pagano  for help  applying the \v Cebotarev density theorem. \indent
Richard Guy was an avid collector and popularizer of inspiring mathematical problems \cite{UPINT, UPIG, UPIC}. While collaborating on \cite{Guy}, an article memorializing the life and work of Richard Guy, I took the opportunity to re-read several of Guy's papers, and found myself   compelled  to revisit linear division sequences, leading to these results.  \newline  \indent 
The author is partially supported by grants from NSERC (Canada). }

\begin{document} 

\begin{abstract}
We classify  all linear division sequences in the integers, a problem going back to at least the 1930s. 
As a corollary we also classify those linear recurrence sequences in the integers for which $(x_m,x_n)=\pm x_{(m,n)}$. 
We also show that if two linear division sequences have a large common factor infinitely often then they are each divisible by a common linear division sequence on some arithmetic progression. Moreover our proofs also work for polynomials. The key to our proofs are Ritt's irreducibility theorem and the subspace theorem (of Schmidt and Schlickewei), in a direction developed by Bugeaud, Corvaja and Zannier.
\end{abstract}

\maketitle

\section{Introduction}  \label{sec: intro}

\subsection{Definitions and the development of a problem}             
A \emph{linear recurrence sequence}
$(u_n)_{n\geq 0}$ of integers (or rationals, or algebraic integers, or polynomials,...) satisfies a recurrence
  \begin{equation} \label{eq: RecRel}
 u_n=c_1u_{n-1}+\cdots + c_k u_{n-k}
  \end{equation}
 where   $c_1,\dots,c_k$  are integers with $c_k\ne 0$ and   $u_0,\dots,u_{k-1}$ are  integers.  
 The \emph{order} of $(u_n)_{n\geq 0}$   is the smallest possible $k$  in any recurrence  \eqref{eq: RecRel} for the $u_n$.
 
A \emph{division sequence} $(u_n)_{n\geq 0}$ of integers satisfies 
  \[
 u_m \text{ divides } u_n \text{ whenever } m \text{ divides } n.
 \]
We then call $(u_n)_{n\geq 0}$ a \emph{linear division sequence} if it is both a linear recurrence sequence and a division sequence.\footnote{This is sometimes called a \emph{linear divisibility sequence}.}
Examples include the Fibonacci numbers $(F_n)_{n\geq 0}$, the Mersenne numbers $2^n-1$, indeed all Lucas sequences, as well as polynomials of the form $n^d$, and their products.
One can analogously define linear division sequences in the algebraic integers,   in $\mathbb C[t]$ and even for polynomials in many variables. 

For every linear recurrence sequence   \eqref{eq: RecRel},    there is a   \emph{characteristic polynomial}
 \[
 f(x) := x^k-c_1x^{k-1}-\cdots -c_{k-1}x-c_k = \prod_{i=1}^m (x-\alpha_i)^{k_i}
 \]
where the $\alpha_i$ are  distinct, non-zero  (as $c_k\ne 0$) algebraic integers (called \emph{characteristic roots}) and the $k_i$ are positive integers for which $k_1+\cdots+k_m=k$. One can write  
  \begin{equation} \label{eq: Formula}
   u_n = \sum_{i=1}^m g_i(n) \alpha_i^n \text{ for all } n\geq 0  
 \end{equation}
 where the $g_i(t)\in \overline{\mathbb Q}[t]$  have degree $k_i-1$ (see e.g.~\cite{EPSW}).  We call $(u_n)_{n\geq 0}$ \emph{simple} if each $k_i=1$ and so each $g_i(t)$ is a constant.
  If we define $(u_n)_{n\geq 0}$ by \eqref{eq: Formula}
 then it satisfies the linear recurrence \eqref{eq: RecRel}  since knowing the $\alpha_i$'s and the multiplicities $k_i$ allow us to construct the characteristic polynomial and so determine the $c_i$'s.

 In 1936 Marshall Hall \cite{MaHa} partly classified all third-order linear division sequences  opening up the general classification question (though he did not feel that he was in a position to make a general conjecture). He remarked that the problem has its antecedents in the 1875 papers of Lucas (section VII of \cite{Luc}) as well as  work of D.H.~Lehmer \cite{Leh, Leh2},\footnote{The latter reference included  Lehmer's conjecture that the Mahler measure of a non-cyclotomic irreducible polynomial cannot get arbitrarily close to $1$.} Pierce \cite{Pier} and Poulet. The problem is mentioned in the book \cite{EPSW} as being ``very old'' (see also \cite{HCW}).   Further down we will discuss the complicated history of this question. Our main result is a resolution of this problem. To begin we develop some theory and examples of linear division sequences (which we denote LDS for convenience). The main difficulty is that LDS's may be defined differently in different congruence classes modulo some integer $M$: for example, $( 1_{(n,M)=1})_{n\geq 0}$ or any other periodic LDS.
 
 Since any product of linear division sequences (LDS's) is also a LDS, we will factor LDS's into  more basic LDS's, so that every LDS is more-or-less a product of   basic LDS's. There are four types of basic LDS's.

 \subsection*{Periodic LDS's} 
 Any periodic division sequence is a LDS. Indeed  $(u_n)_{n\geq 0}$  is a (purely) periodic LDS of period $M$  if and only if 
 $u_d$ divides $u_D$ whenever $d$ divides $D$, which divides $M$, and
 \begin{equation} \label{eq: PeriodicLDS}
u_n=u_a=\sigma_a u_d    \text{ whenever } n\equiv a \pmod M \text{ and } d=(a,M)
\end{equation}\
where $\sigma_d=1$ and otherwise $\sigma_a=-1$ or $1$.  

If there are characteristic roots $\alpha_i\ne \alpha_j$ such that $\alpha_i/ \alpha_j$ is a root of unity then $(u_n)_{n\geq 0}$ is \emph{degenerate}.\footnote{Some authors have restricted their attention to classifying non-degenerate LDS's but not all, so we investigate both.}  Much more important for our considerations is 
 the \emph{period} $M$ of the linear recurrence sequence $(u_n)_{n\geq 0}$ which is defined to be  the smallest positive integer such that whenever $\alpha_1^{e_1}\cdots \alpha_m^{e_m}$ (each $e_i\in \mathbb Z$) is a root of unity, it is an $M$th root of unity.\footnote{Beware: By this definition, every  linear recurrence sequence has a period, \emph{but} this does not necessarily make it periodic!}
   Thus if $(u_n)_{n\geq 0}$ is degenerate then $M>1$.\footnote{But not vice-versa, for example $8^n-(-4)^n-2^n+(-1)^n$ has period 2 but is non-degenerate.}

There are many ways in which a ``period'' $M$ can be relevant. For example if $(v_n)_{n\geq 0}$  is a LDS then so is
$(u_n)_{n\geq 0}$  where $u_{Mn}=v_n$ and otherwise $u_n=1$. More generally if $(v^{(d)}(n))_{n\geq 0}$  are LDS's for each $d$ dividing $M$
then 
 \begin{equation} \label{eq: BuildingLDS}
u_n := \prod_{d|(M,n)}  v^{(d)}_{n/d}(t) \text{ is also  a LDS}.
\end{equation}
This construction invalidates several   classifications that have been proposed in the literature (e.g.~in \cite{EV}).
We will see how periods impact on all the other types of LDS's:
 
 \subsection*{Power LDS's} We have already noted the examples $( n^d)_{n\geq 0}$ for each integer $d\geq 0$, but there are also   more complicated examples in which $u_n$ is more-or-less a power of $n$ depending on some
period $M\geq 1$: Select integers $e_d$ for each positive divisor $d$ of $M$ for which 
 \[
 0\leq e_d\leq e_D \text{ whenever } d \text{ divides } D, \text{ which divides } M.
 \]
 Then we have the \emph{power LDS}
 \begin{equation} \label{eq: PowerLDS}
    u_n = (n/d)^{e_d} \text{ where } (n,M)=d.
\end{equation}

 \subsection*{Exponential LDS's} An LDS like $( 2^n)_{n\geq 0}$ appears at first sight to be a trivial complication. But  exponents in such a LDS can assume different patterns, and exponential LDS's can be multiplied together. We formulate exponential LDS's as powers of a variable $t$, and then we can substitute an integer for $t$ to get an exponential LDS in $\mathbb Z$. We begin with $u_0=0$ and, for   $1\leq a\leq M$, we define the 
\emph{exponential LDS} in $\mathbb C[t]$  by  
\begin{equation} \label{eq: ExponentialLDS}
    u_{a+kM}(t)  = t^{kh_a} \text{ for all }  k\geq 0
\end{equation}
 where each $h_a\in \mathbb Z_{\geq 0}$ and if $1\leq a,b,r\leq M$ then
\[
h_a\leq rh_b \text{ whenever  } b\equiv r a \pmod M.
\]
We obtain an \emph{exponential LDS} in $\mathbb Z$ by taking
distinct primes $p_1,\dots,p_r$ and exponential LDS's  $(u_{1,n}(t))_{n\geq 0},\dots, (u_{r,n}(t))_{n\geq 0}$ in $\mathbb C[t]$, and then letting
 \begin{equation} \label{eq: ExponentialLDS2}
    u_n  =  u_{1,n}(p_1) \cdots  u_{r,n}(p_r) \text{ for every } n\geq 1.
\end{equation}
Exponential LDS's in $\mathbb Z$ are easily identified since there must exist a prime $p$ such that 
$v_p(u_n)\gg n$ for an infinite sequence of integers $n$ (where $v_p(m)$ denotes the exact power of $p$ dividing $m$), and this does not hold for any other type of LDS.

 \subsection*{Polynomially generated LDS's}  
 This is the most interesting class of LDS's, and includes   Lucas sequences (that is, $\frac{\alpha^n-\beta^n}{\alpha-\beta}$ where   $\alpha\ne \beta$ are both integers, or conjugate quadratic integers) like $F_n$ and $2^n-1$, and many others. The main idea is to construct a linear division sequence in $\mathbb C[t]$, homogenize it to obtain a LDS in $\mathbb C[x,y]^{\text{homogenous}}$, substitute in algebraic integers for $x$ and $y$, and then take the norm (or an appropriate subproduct of the norm) to obtain a LDS in $\mathbb Z$. For example $(\frac{t^n-1}{t-1})_{n\geq 0}$ is a 
 LDS in $\mathbb C[t]$, and it homogenizes to the  LDS $(\frac{x^n-y^n}{x-y})_{n\geq 0}$  in $\mathbb C[x,y]^{\text{homogenous}}$. We can substitute in $x=2,y=1$ to obtain the Mersenne numbers, or $x,y=\frac{1\pm \sqrt{5}}2$  for the Fibonacci numbers\footnote{A priori one might have expected to take the norm here for $\mathbb Q( \sqrt{5})/\mathbb Q$ but we need not since the value is already an integer, and so we restrict our product to a subproduct of the norm.}, or conjugate quadratics $\alpha$ and $\beta$ to obtain a Lucas sequence.
 
The simplest LDS's in $\mathbb C[t]$ take the form $(f(t^n)/f(t))_{n\geq 0}$ for an appropriate choice of $f(t)\in \mathbb C[t]$. We just saw that Lucas sequences come from the example
   $f(t)=t-1$, but one can also take $f(t)=t^2-1$ or $t^3-1$, or even $\lcm[t^2-1,t^3-1]=(t+1)(t^3-1)$ (the least common multiple as polynomials), and in general 
 \begin{equation} \label{eq: f(t)-defn}
f(t)=   \text{lcm} [(t^k-1)^{a_k}: k\geq 1]  
\end{equation}
where each $a_k\geq 0$ is an integer, with only finitely many of the $a_i$ non-zero. (This is a reformulation of a result of Koshkin \cite{Kosh}.).

A \emph{multiplicatively-dependent} linear recurrence is one in which all   pairs of  characteristic roots $\alpha_i\ne \alpha_j$ are multiplicatively dependent (that is there exist non-zero integers $r,s$ for which $\alpha_i^r=\alpha_j^s$).\footnote{In which case there exists a $\gamma$ such that every $\alpha_i=\zeta_i\gamma^{a_i}$ for some $a_i\in \mathbb Z$ and root of unity $\zeta_i$ (see Lemma \ref{lem: premultdep}).}  The LDS's $(f(t^n)/f(t))_{n\geq 0}$ given by \eqref{eq: f(t)-defn} are all multiplicatively-dependent LDS's in $\mathbb C[t]$, and there are more, for example one can take
$f_n(t)=(\zeta t)^n-1$ where $\zeta$ is a  root of unity for each $n\geq 0$, or even
\begin{equation} \label{eq: f(t)inM-defn}
f_n(t)=\underset{\substack{\ m\geq 1 \\ \,  0\leq j\leq M-1 }}\lcm  (  (\zeta_{M}^{j} t^m)^n-1)^{h_{m,j}}
\end{equation}
where $\zeta_M$ is a primitive $M$th root of unity, each $h_{m,j}\geq 0$ is an integer, with only finitely many non-zero.  
LDS's $(f_n(t))_{n\geq 0}$ in $\mathbb C[t]$ of period $M$ are classified in Theorem \ref{thm: genF(t)},
 normalized so that   $f_d(t)=1$ whenever $d$ divides $M$.

Homogenizing the $(f_n(t))_{n\geq 0}$ we obtain a \emph{multiplicatively-dependent LDS} $(F_n(x,y))_{n\geq 0}$ in  $\mathbb C[x,y]^{\text{homogenous}}$;
here we adjust our definition of  multiplicatively-dependent since now it is the ratios $\alpha_i/\alpha_j$ of the characteristic roots that are multiplicatively dependent (not the $\alpha_i$'s themselves). 

Next for any given $\alpha,\beta\in \overline{\mathbb Q^*}$ where $\alpha/\beta$ is not a root of unity, let $L$ be the splitting field extension for $\alpha$ and $\beta$ over
$\mathbb Q$, and let  $H$ be a minimally sized set of elements of $\text{Gal}(L/\mathbb Q)$ for which
\[
\prod_{\sigma\in H} F_n (\alpha,\beta)^\sigma \in \mathbb Z \text{ for all } n, \text{ and } \prod_{\sigma\in H} (\alpha,\beta)^\sigma \text{ is an ideal of } \mathbb Z
\]
(where $\theta^\sigma$ means applying the Galois element $\sigma$ to the algebraic number $\theta$).\footnote{$H$ is typically a set of representatives for 
$
\text{Gal}(L/\mathbb Q)\, /\, \text{Stab}_{\text{Gal}(L/\mathbb Q)} ( \tfrac{F_n(\alpha,\beta)}{F_1(\alpha,\beta)}: n\geq 1  ).
$
but not always, as we will see with Lehmer sequences below.}
 Then select $q$    to be the positive integer for which
\[
(q) = \prod_{\sigma\in H} (\alpha,\beta)^\sigma, \text{   let } v_n = \frac 1{ q^{\deg F_n}} \cdot  \prod_{\sigma\in H} F_n (\alpha,\beta)^\sigma ,
\]
and $(u_n)_{n\geq 0}$ is a \emph{polynomially generated LDS} in $\mathbb Z$ where $u_n=v_n/v_{(n,M)}$.

  We call $(u_n)_{n\geq 0}$  a  \emph{Lucas sequence of order} $r$ if
 $F_n(x,y)=x^n-y^n$ and $|H|=r$.

\begin{remark} For convenience, in the definitions of power LDS, 
 exponential  LDS  and polynomially generated LDS, we ``normalized'' the values at  each divisor $d$ of $M$, 
 forcing $u_d=1$, but left $u_d$ to be more flexible for a periodic LDS.   
\end{remark}
 
 \subsection{Classifying    LDS's and strong LDS's in $\mathbb Z$}
We begin by classifying non-degenerate linear division sequences in the integers.

 \begin{corollary} \label{thm: main6}
 If $(u_n)_{n\geq 0}$ is a non-degenerate linear division sequence of period $M$ in the integers with $u_1=1$  then $u_n=n^e \lambda^{n-1}v_n$ where $e\geq 0$ and $\lambda\ne 0$ are integers, and
 $(v_n)_{n\geq 0}$ is the product of  polynomially generated LDS's, constructed   from LDS's $(f_n(t)/f_1(t))_{n\geq 0}$ in $\mathbb C[t]$ with  $f_n(t)$ as in \eqref{eq: f(t)inM-defn}.
 \end{corollary}
 
 The $\zeta_M$'s in the definition of the $v_n$'s in Corollary \ref{thm: main6} seem to be necessary: for example
 $8^n-(-4)^n-2^n+(-1)^n$, and see the  discussion after the proof of Corollary \ref{thm: main6}.
  
  Corollary \ref{thm: main6}  follows from our main theorem:
 
 \begin{theorem} \label{thm: main}
 If $(u_n)_{n\geq 0}$ is a linear division sequence in the integers with period $M$ then $u_n$ is the product of a periodic sequence of rationals (of period $M$), a power LDS, an exponential  LDS  and a finite number of polynomially generated LDS's.
 \end{theorem}
 
 Theorem \ref{thm: main} would be  ``if and only if'', if we could be more specific about the periodic sequence of rationals (so as to guarantee each $(u_n)_{n\geq 0}$ is a division sequence in $\mathbb Z$) but the necessary conditions appear to be complicated.  Certainly the product is a LDS  in $\mathbb Z$ if  the periodic sequence is any periodic LDS, but there are can be several other cases besides.  If we don't allow an exponential LDS factor (which  is arguably the more natural case) then we can provide an ``if and only if'' statement:
   
  \begin{corollary} \label{thm: main2}
  $(u_n)_{n\geq 0}$ is a linear division sequence in the integers  with  \newline
 \centerline{ $\limsup_{n\to \infty:\ u_n\ne 0} v_p(u_n)/n =0$ for every prime $p$ } \newline
 if and only if
    $u_n$ is the product of a periodic LDS, a power LDS,   and a finite number of polynomially generated LDS's.
 \end{corollary}
 
 In section \ref{sec: FindFactors} we will explain how to determine the ``LDS factors'' from Theorem \ref{thm: main} 
 and Corollary \ref{thm: main2} of a given LDS in $\mathbb Z$.
 
 A \emph{strong linear division sequence in the integers}  is a linear recurrence sequence $(u_n)_{n\geq 0}$ for which
 \[
 (u_m,u_n)=|u_{(m,n)}|  \text{ for all integers } m,n\geq 1.\footnote{Some authors allow only $(u_m,u_n)=u_{(m,n)}$ so that the $u_r$'s are all $\geq 0$. However this definition then excludes desirable examples like $u_n=\frac 1{2i}((2+i)^n-(2-i)^n)$.}
 \]
 Well-known examples include $F_n, 2^n-1, n$ and powers of these. In our next result we will classify all strong LDS's, showing that there are little more than the obvious generalizations of these examples.

\begin{theorem} \label{thm: main3}
 $(u_n)_{n\geq 0}$ is a strong linear division sequence in the integers of period $M$ with $u_1=1$ if and only if $\kappa_n$  is a strong periodic LDS of period $M$ with $\kappa_n=\pm \kappa_d$ where $d=(n,M)$ for $1\leq n\leq M$, and for some integer $r\geq 1$ one of the following holds: 
\begin{itemize} 
\item $u_n=\kappa_d$;   
\item $u_n = \kappa_d (n/d)^r$  and   each $\kappa_{p^e}$ is a power of $p$,   and $\kappa_d = \prod_{p^e\| d} \kappa_{p^e}$ for each $d|M$; 
\item  $u_n=\kappa_d   (w_n/w_d)^r$, $w_n$  a Lucas  sequence or a Lehmer sequence (with $r$ even).\footnote{Lehmer sequences will be described in section \ref{sec: conjugates}.}
\end{itemize}
Moreover if $D$ is the smallest integer for which a given prime power $p^e$ divides $u_D$, then
 $p^e$ divides $u_n$ if and only if $D$ divides $n$.  
\end{theorem}

 Our proofs of these theorems do not require that the linear recurrence $(u_n)_{n\geq 0}$ is   a   division sequence in the hypotheses, but rather only  that $u_n$ divides $u_{jn}$ (or $(u_{in},u_{jn})=u_{(i,j)n}$)  for all $i,j$ up to some (explicit) bound that depends only on $k$,   for   many large $n$ in each arithmetic progression mod $M$. We can interpret this as saying   that if   $\gcd(u_n,\dots,u_{rn})=u_n$ often enough then there is an   ``algebraic reason'' for this (that is, that $(u_n)_{n\geq 0}$ is a LDS).
 One might guess that even if these gcds are only pretty large then there must also be an
 ``algebraic reason'' for it, as we now exhibit:

\begin{theorem} \label{thm: main4}
Suppose that $(u_n)_{n\geq 0}$ is a linear recurrence sequence in the integers with period $M$ of order $k$ with $u_0=0$. Fix $\epsilon>0$. Suppose that $\gcd(u_n,\dots,u_{(k-1)n})>e^{\epsilon n}$ for 
$\geq \epsilon \frac xM$ integers $n\leq x$ with $n\equiv a \pmod M$ once $x$ is sufficiently large.
Then there exists a 
LDS, $(v_n)_{n\geq 0}$, and a linear recurrence sequence  $(w_n)_{n\geq 0}$, both in $\mathbb Z$, such that 
\[
 u_n=v_nw_n \text{ for all sufficiently large } n\equiv a \pmod M.
 \]
 \end{theorem}
 
 We can  also prove such a theorem when taking the gcd of two different LDS's:
 
  \begin{theorem} \label{thm: main5}
 If $(u_n)_{n\geq 0}$ and $(v_n)_{n\geq 0}$ are linear division sequences in $\mathbb Z$ for which
$\gcd(u_m,v_n)>e^{\epsilon (m+n)}$ for infinitely many pairs $m,n$, then there exists a 
 LDS $(w_n)_{n\geq 0}$ in $\mathbb Z$ and $q\in \mathbb Z_{>0}$ such that $w_n$ divides $(u_{qn},v_{qn})$ for all $n\geq 1$.
  \end{theorem}

  In both of these theorems the LDS factor can be taken to be  either an exponential or a polynomially generated LDS, and so grows exponentially fast.
   The proofs of Theorems \ref{thm: main4} and \ref{thm: main5}  are a relatively simple consequence a (deep)   application by Corvaja and Zannier \cite{CZ} of the Schmidt-Schlickiwei subspace theorem.
  In section \ref{sec: Future} we conclude the paper by discussing various developments, particularly the paper of Levin \cite{Lev},   exploring  further questions about large factors of linear recurrence sequences.

 Hall \cite{MaHa} and many subsequent authors focused on  determining all of the LDS's in $\mathbb Z$ of low order with $u_0=0$. In section \ref{sec: Loworder} we use our results to establish a method that can determine all LDS's in $\mathbb Z$  with $u_0=0$ with order up to any given bound

   We  give two different though inter-related proofs of Theorem \ref{thm: main}, each depending on  the Ritt-Gourin-MacColl theory of multinomial factorizations \cite{Rit, Gou, MacC}, and the Hadamard quotient theorem \cite{vdP, Rum}. We will see that each approach has its benefits for different applications.  Corollaries \ref{thm: main6} and \ref{thm: main2} and Theorem \ref{thm: main3} are essentially corollaries involving mostly complicated elementary number theory, but also the \v Cebotarev density theorem for the proof of  Theorem  \ref{thm: main3}

  An analogous result to Theorem \ref{thm: main} can be proved in $\mathbb C[t_1,\dots,t_m]$ with obvious modifications (since the application of the Ritt-Gourin-MacColl theory will work analogously). Analogous results to Theorems 
  \ref{thm: main4} and  \ref{thm: main5}  can be proved in $\mathbb C[t_1,\dots,t_m]$ by replacing Corvaja-Zannier \cite{CZ} by the results of Ostafe \cite{Osta}.

 \subsection*{Simplifying notation}  Every fractional ideal $I$ of $K$ may be written as $N/D$ where $N$ and $D$ are coprime integral ideals, by the unique factorization theorem for ideals. Even if $I$ is principal, one cannot necessarily determine $N$ and $D$ that are both coprime and principal (that is, the analogy of ``reduced fractions'' is complicated in number fields).  We let $I^*:=N$ be the numerator. Polynomially generated LDS's were defined as the generator of the ideal which is the product over a set of conjugates of  
 \[
 \frac{ (F_n(\alpha,\beta))}{(\alpha,\beta)^{\deg F_n}} = (f_n(\gamma))^* \text{ where } F_n(x,y)=y^{\deg F_n} f_n(x/y)
 \text{ and } \gamma=\alpha/\beta,
 \]
 which gives a different perspective on polynomially generated LDS's.
 We   say that $f_n(\gamma)$ divides $u_n$ if $(f_n(\gamma))^*$ divides $u_n$. We also define 
 $N(I):=\text{Norm}_{K/\mathbb Q}(I^*)$.
 
 For each subset $I\subset \{ 1,\dots, k\}$ let $g_I(t):=\sum_{i\in I} g_i(t)$ where the $g_i(t)$ are as defined in \eqref{eq: Formula}. Let  $ \mathcal I^*$ be the set of partitions
$ \mathcal I =I_1\cup \dots \cup I_h$ of $ \{ 1,\dots,k\}$ (with $\alpha_j\in I_j$ for $1\leq j\leq h$)
for which  each $g_{I_j}(t)=0$ and 
the elements of $ \{ \alpha_i/\alpha_j:\ i\in I_j, 1\leq j\leq h\}$ are pairwise multiplicatively dependent.\footnote{$u$ and $v$ are multiplicatively dependent if there exist non-zero integers $r$ and $s$ for which $u^r=v^s$.}
 
 We will see that for every   $\gamma$ which gives rise to a polynomially generated LDS for which  $f_n(\gamma)\ne 1$ divides $u_n$, there exists a partition  $ \mathcal I\in \mathcal I^*$ such that 
 $( \alpha_i/\alpha_j)^n=(\gamma^n)^{e_i}$ for some $e_i\in \mathbb Z$ for all $i\in I_j, 1\leq j\leq h$.


\section{History and first remarks on LDS's}

The first example of a linear division sequence in $\mathbb Z$ is the Fibonacci numbers, defined by the linear recurrence $F_{n}=F_{n-1}+F_{n-2}$ for all $n\geq 2$
with $F_0=0$ and $F_1=1$. One   also has the Mersenne numbers 
$u_n=2^n-1$, or $u_n=\frac{\alpha^n-\beta^n}{\alpha-\beta}$ for two distinct integers $\alpha,\beta$, or even for two conjugate quadratic algebraic integers $\alpha,\beta$ (that is, any \emph{Lucas sequence}).  These are all polynomially generated LDS's.
 There are other   examples like 
$u_n=a^n, u_n=n^d$ and appropriate periodic sequences,   as well as any product of already known linear division sequences.   

In 1875 \'Edouard Lucas \cite{Luc} in the first volume of the first mathematics journal to be printed in North America, wrote extensively (in French) about the local and global properties of linear recurrence sequences, mentioning the divisibility of Fibonacci numbers by one another. Such divisibility properties arose in several subsequent works, and in
1936 Marshall Hall \cite{MaHa} asked specifically about LDS's and partly classified  third-order ones (that is, $k=3$ in \eqref{eq: RecRel}). Since then it was until very recently implicitly supposed that all LDS's are a product of the simplest LDS's mentioned in the previous paragraph. Some support of this supposition was given by  B\'ezivin,   Peth\H{o} and  van der Poorten \cite{vP} who showed that any LDS \emph{divides} the product of certain LDS's of order 2, times  $cn^d$ for some integers $c\ne 0$ and $d\geq 0$, but this is not  precise enough to answer many obvious questions (like classifying all  LDS's of order 3).

To discuss LDS's we need to make, without loss of generality, some simplifying assumptions. First, we may assume that 
\[
u_1=1
\]
 since if  $(u_n)_{n\geq 0}$ is a LDS then so is $(u_n/u_1)_{n\geq 0}$
(and $u_1\ne 0$ else $u_n=0$ for all $n$).

Either $(u_n)_{n\geq 0}$ is a periodic LDS, or 
\[
u_0=0.
\]
For if  $u_0\ne 0$ then each  $u_n$ divides $u_0$, a finite set of possibilities, so there are only finitely many possibilities for 
$(u_{n-1},\dots, u_{n-k})$.  Therefore $u_{n-i}=u_{m-i}$ for $1\leq i\leq k$ for some $n>m\geq 0$, so that
$u_{n+i}=u_{m+i}$ for all $i\geq 0$ by induction using \eqref{eq: RecRel}; that is, $(u_n)_{n\geq 0}$ is periodic.  
We will prove our classification of periodic LDS's (as well as power LDS's and  exponential LDS's) at the end of this section.

We now highlight  why it has   been challenging to even guess at what the classification should be,    through several examples.   Several authors seem to have assumed that linear division sequences will be the product of 
Lucas sequences and power LDS's. If so then all periodic LDS's would be the   product
of LDS's of order two but this is easily contradicted by examples like $u_n=7$ if $3|n$ and $u_n=1$ otherwise, given by
$u_n=u_{n-3}$ for all $n\geq 3$ so that
$u_n=3+2\omega^n+2\overline{\omega}^n$ where $\omega$ is a primitive cube root of unity. 

So perhaps LDS's might all be products of Lucas sequences, periodic LDS's and power LDS's?
The Lucas sequences might not all be defined over $\mathbb Z$ so would be a product of conjugate Lucas sequences.
However this  does not  give all the possibilities as shown by an example found by  
 Peter Bala \cite{Bal}, namely $(L_nF_{3n})_{n\geq 0}$ where $L_n$ is the  sequence that  begins $2,1,3,4,\dots$ with $L_{n}=L_{n-1}+L_{n-2}$ for all $n\geq 2$, so that
\[
L_nF_{3n}=\lcm[ F_{2n}, F_{3n}] = U_n(\alpha,\beta) \text{ with } \alpha,\beta=\frac{-1\pm \sqrt{5}}2
\]
and 
 \[
 U_n(x,y)=\frac{f(x^n,y^n)}{f(x,y)} \text{ with } f(x,y)=(x+y)(x^3-y^3),
 \]
   a linear division sequence in $\mathbb Z[x,y]$. At first sight it is not obvious that  $f(x^m,y^m)$ divides $f(x^n,y^n)$ whenever $m$ divides $n$, but becomes obvious when we note that
\[
f(x,y)=  \lcm[ x^2-y^2, x^3-y^3] \text{ in } \mathbb Z[x,y].
\]
This is the main idea that was missing in previous attempts to classify LDS's, that one can take lcms of LDS's in $\mathbb C[t]$ to obtain another LDS in $\mathbb C[t]$, and then homogenize and substitute algebraic integers to obtain a LDS in $\mathbb Z$. In the next section we will classify LDS's in $\mathbb C[t]$, which will allow us to find all polynomially generated LDS's (as given in the introduction).
When taking the product over conjugate algebraic integers, there are several issues to take care of:

\subsection{Conjugates and gcd's, not norms} \label{sec: conjugates}
Substituting algebraic integers $x=\alpha, y=\beta$ into the LDS $u_n(x,y)=\frac{x^n-y^n}{x-y}$ in $\mathbb Z[x,y]$ and taking  
$\text{Norm}_{K/\mathbb Q}u_n(\alpha,\beta)$ where $K=\mathbb Q(\alpha,\beta)$ we  obtain a LDS  in $\mathbb Z$ but not necessarily the most ``basic'' LDS that could come out of this construction. For example with  $\alpha,\beta=\frac{-1\pm \sqrt{5}}2$,   the norm of  $u_n(\alpha,\beta)$ is $F_n^2$, not the simpler $F_n$.  

Let $L$ be the splitting field extension for $\alpha$ and $\beta$ over $\mathbb Q$ and $G:=\text{Gal}(L/\mathbb Q)$.
Let  $H$ be a minimally sized set of elements of $\text{Gal}(L/\mathbb Q)$ for which
$\prod_{\sigma\in H} u_n (\alpha,\beta)^\sigma \in \mathbb Z \text{ for all } n, \text{ and } \prod_{\sigma\in H} (\alpha,\beta)^\sigma \text{ is an ideal of } \mathbb Z$, and then we divide the first quantity be the second to obtain a LDS in $\Z$. 

The following family of examples, due to Lehmer \cite{Leh, HCW}, is surprising so worth mentioning:
 For      squarefree  integers $r, s=r-4d$ let  $\alpha,\beta=\frac{\sqrt{r}\pm \sqrt{s}}2$, which are two roots of an irreducible quartic polynomial in $\mathbb Z[x]$   (the other roots are $-\alpha$ and $-\beta$).  Now $u_n(\alpha,\beta)\in \mathbb Z$ if $n$ is odd, $\in \sqrt{r}\cdot \mathbb Z$ if $n$ is even, so we let $H=\{ \i, \rho\} $ where $\rho(\sqrt{r})= -\sqrt{r}$, and therefore every $u_n(\alpha,\beta)\cdot u_n(\alpha,\beta)^\rho\in \mathbb Z$. But $u_n(\alpha,\beta)^\rho =(-1)^{n-1} u_n(\alpha,\beta)$ and so $u_n(\alpha,\beta)^2$ is a Lucas sequence of order two. This has characteristic roots
 $\alpha^2,\beta^2 = \frac{r-2d\pm \sqrt{rs}}2$, the roots of an irreducible quadratic polynomial in $\mathbb Z[x]$, and  $\alpha\beta=d\in \mathbb Z$.  We call $u_n(\alpha,\beta)^2$ a \emph{Lehmer sequence}, which is evidently a (non-decomposable) LDS of order 3 (and a Lucas sequence of order 2). We will take another perspective on Lehmer sequences at the end of this section.

\medskip
   When $\alpha,\beta\in \mathbb Z$ we may assume they are coprime by dividing out by their gcd. However in a number field their gcd, the ideal $(\alpha,\beta)$, might be non-principal, in which case $u_n(\alpha,\beta)$ would be unavoidably divisible by $(\alpha,\beta)^{n-1}$. This explains why we took $q\in \mathbb Z_{>0}$ for which 
   $(q)=\prod_{\sigma\in H}  (\alpha,\beta)^\sigma$ and take the LDS
   $q^{-\deg u_n} \prod_{\sigma\in H} u_n(\alpha,\beta)^\sigma$ in $\mathbb Z$.
   
Any Lucas sequence of order 2 takes the form
\[
q^{-(n-1)} \frac{\tau_1^n-\nu_1^n}{\tau_1-\nu_1} \cdot \frac{\tau_2^n-\nu_2^n}{\tau_2-\nu_2} =q^{-(n-1)} \frac{ \alpha^n+\delta^n-\beta^n-\gamma^n} { \alpha+\delta-\beta-\gamma} ,
\]
where $(q)=(\tau_1,\nu_1)(\tau_2,\nu_2)$ with $\alpha=\tau_1\tau_2, \beta=\tau_1\nu_2, \gamma=\nu_1\tau_2$ and $\delta=\nu_1\nu_2$, so that $\alpha\delta=\beta\gamma$.   Guy and Williams \cite{GW}  considered examples where  $\alpha\delta=\beta\gamma=m\in \mathbb Z$ so that 
\[
(x-\alpha)(x-\delta)=x^2-\rho x+m,
(x-\beta)(x-\gamma)=x^2-\overline{\rho} x+m \text{ with } (x-\rho)(x-\overline{\rho})\in \mathbb Z[x].
\]
In this case one obtains an alternative (though equivalent) factorization
  \begin{equation} \label{eq: HW}
  u_n:=\frac{ \alpha^n+\delta^n-\beta^n-\gamma^n} { \alpha+\delta-\beta-\gamma}  =  \alpha^{-(n-1)}  \bigg( \frac{\alpha^n -\beta^n}{\alpha -\beta}   \bigg)\bigg( \frac{\alpha^n -\gamma^n}{\alpha -\gamma} \bigg) .
 \end{equation}
One can develop these examples starting from the last two displayed equations. If $\alpha,\beta,\gamma$ and $\delta$ are all irreducible then $u_n$ cannot be expressed without dividing through by some quantity like $\alpha^{n-1}$. For example, we have the famous example
\[
2\cdot 3 = (1+\sqrt{-5}) (1-\sqrt{-5})
\]
so that $u_n=\frac 13(2^n+3^n-(1+\sqrt{-5})^n- (1-\sqrt{-5})^n)$ is a LDS starting $0,1,7,21$ with
$u_{n+4}=7u_{n+3}-22u_{n+2}+42u_{n+1}-36u_n$ for $n\geq 4$. In terms of a Lucas sequence of order two we have
\[
u_n=3^{-(n-1)}  \bigg( \frac{3^n -(1+\sqrt{-5})^n}{2-\sqrt{-5}}   \bigg)\bigg( \frac{3^n -(1-\sqrt{-5})^n}{2+\sqrt{-5}}    \bigg) .
\]

If $\alpha,\beta=\frac{\sqrt{r}\pm \sqrt{s}}2$ where $(r,s)=1$ then
$\alpha\sqrt{r},\beta\sqrt{r}=\frac{r\pm \sqrt{D}}2$ where $D=rs$, so that $u_n(\alpha\sqrt{r},\beta\sqrt{r})= \sqrt{r}^{n-1} u_n(\alpha,\beta)$ is a Lucas sequence in $\Z$ (where $u_n(\alpha,\beta)$ is a Lehmer sequence). Although the $u_n(\alpha\sqrt{r},\beta\sqrt{r}), n\geq 1$ are stable under the action of $G$, the gcd $(\alpha\sqrt{r},\beta\sqrt{r})=\sqrt{r}$ is not, so $H$ can be taken to be  the identity together with any other $\sigma\in G$, obtaining the 
 Lehmer sequence $u_n(\alpha,\beta)^2 =u_n(\alpha\sqrt{r},\beta\sqrt{r})^2/r^{n-1}$.

\subsection{gcd's and lcm's of LDS's in $\mathbb Z$}
If $(u_n)_{n\geq 0}$ and $(v_n)_{n\geq 0}$ are division sequences in $\mathbb Z$ then so are 
$(\gcd(u_n,v_n))_{n\geq 0}$ and $(\lcm[u_n,v_n])_{n\geq 0}$. They can be linear recurrence sequences, like
$\gcd(F_{2n}, F_{3n})=F_n$ and $\lcm[F_{2n}, F_{3n}]=L_nF_{3n}$ but  not necessarily so, as we will show  with the example   $\gcd(2^n-1,F_n)$ in Theorem \ref{thm: NotALinRec}.  We now show that the gcd and lcm questions are   equivalent:
\medskip

\noindent \textbf{The Hadamard quotient theorem}   (van der Poorten and Rumely, \cite{Rum, vdP})
\emph{ If $(a_n)_{n\geq 0}$ and  $(b_n)_{n\geq 0}$ are linear recurrence sequences in the integers and
$a_n=b_nc_n$  for all $n\geq 0$ where each $c_n\in \mathbb Z$ then $(c_n)_{n\geq 0}$ also satisfies a linear recurrence.}\footnote{Corvaja and Zannier \cite{CZ2} proved the much stronger result that if $b_n$ divides $a_n$ for infinitely many $n$ then there exists a polynomial $f(x)$ and a linear recurrence $(c_n)_{n\geq 0}$ in $\mathbb Z$ for which $f(n)a_n=b_nc_n$  for all $n$ in a suitable arithmetic progression. The example $a_n=2^n-2, b_n=n+2^n+(-2)^n$ shows that the $f$ can be necessary when taking $n$ to be odd. This also follows from the recent work of Levin \cite{Lev} which we discuss   in section \ref{sec: Future}.}
\medskip

 Taking $a_n=u_nv_n$ with $b_n= (u_n,v_n)$ or $[u_n,v_n]$ we deduce that 
 $(u_n,v_n)$ satisfies a linear recurrence if and only if $[u_n,v_n]$ does.

\subsection{Classifying periodic LDS's}
Suppose that $(u_n)_{n\geq 0}$ is a LDS in $\mathbb Z$ of period $M$. Evidently $u_d$ divides $u_D$ whenever $d$ divides $D$ divides $M$.  If $(a,M)=d$ then there exist integers $A$ and $B$ such that $a=Ad$ and $aB\equiv d \pmod M$, so that $u_d$ divides $u_{Ad}=u_a$, which divides $u_{aB}=u_d$ (by periodicity), and so $|u_a|=|u_d|$.
We therefore obtain the classification given in \eqref{eq:  PeriodicLDS}.

\subsection{Classifying power LDS's}
Here  $u_n=n^{e_a}$ where the $e_a$ have period $M$.
Let $p$ be a large prime $\equiv 1 \pmod M$. If $d$ divides $a$  and $(a,p)=1$ then
 $u_{dp}$  divides $u_{ap}$, and so 
$e_d=v_p(u_{dp})\leq v_p(u_{ap})=e_a$ (where we define $v_p(p^vr/s)=v$ when $p\nmid rs$ and $v\in \mathbb Z$).
Moreover $e_d\geq 0$ as $u_{dp}$ is an integer. Taking $a=D|M$ we deduce that 
$0\leq e_d\leq e_D$ whenever $d|D|M$. Now suppose $(a,M)=d$ so as in the last paragraph
$u_{dp}$ divides $u_{ap}$ which divides $u_{aBp}$, so that 
$e_d\leq e_a\leq e_d$ as $aBp\equiv d\cdot 1\equiv d \pmod M$, so that $e_a=e_d$.
 Therefore we have 
\[
u_n = n^{e_d}   \text{ where } d=(n,M), 
\]
which is almost the classification in \eqref{eq: PowerLDS}; there we adjust the constant outside by dividing through by $d^{e_d}$. This is a LDS since $n^{e_d}$ divides $(np)^{e_D}$ where $D=(dp,M)$ which divides $dp$, and so
 $(n/d)^{e_d}=(np/dp)^{e_d}$ divides $(np/D)^{e_d}$ which divides $(np/D)^{e_D}$ as $e_d\leq e_D$.

\subsection{Classifying exponential LDS's}
Here  $u_n=m_a^{k}$ where $n=a+kM$ for $1\leq a\leq M$ and $k\geq 0$.
We will study divisibility for each prime $p$, so if $v_p(m_a)=h_{p,a}$ then
$u_{n,p}=p^{kh_{p,a}}$ for $n=a+kM$ (so that $u_n=\prod_p u_{n,p}$) must be a division sequence.

Now if $b\equiv ra \pmod M$ with $1\leq a,b,r\leq M$ then $u_{a+kM,p}$ divides $u_{ra+rkM,p}=u_{b+\ell M}$ where $\ell=rk+ [\frac{ra}m]$, and so
\[
 kh_{p,a} = v_p(  u_{a+kM}) \leq  v_p( u_{b+\ell M} ) =\ell h_{p,b}  
\]
and so, dividing through by $k$, we get $h_{p,a} \leq rh_{p,b} +O(1/k)$.  Letting $ k\to \infty$
we obtain  $h_{p,a} \leq rh_{p,b}$ as these are all integers. We therefore deduce that $u_n$ is given by \eqref{eq: ExponentialLDS2}.

 \section{The period of a linear recurrence sequence} \label{sec: Period}
In \eqref{eq: Formula} we write $g_i(t)=\sum_{j=0}^{k_i-1} g_{i,j} t^j $ where the  $g_{i,j}$ are algebraic numbers.  Given the $\alpha_i$ one can determine the $g_{i,j}$ from $u_0,\dots,u_{k-1}$, and vice-versa.  
 
 \subsection{A multiplicative basis for the roots of the characteristic polynomial}
 In an example like $1^n-2^n-3^n+6^n$ we see that the $n$th powers, $1^n, 2^n, 3^n$ and $6^n$, are all multiplicatively dependent on $2^n$ and $3^n$. In general we can  find a multiplicative basis, though we need to be careful about ``torsion'' (for example, for $1+5^n+(-5)^n$ we have the ``basis'' $5^n$ and $(-1)^n$, the latter being ``torsion'').  One can write each
   \[
  \alpha_i = \zeta_M^{e_{i,0}} \gamma_1^{e_{i,1}}\cdots \gamma_r^{e_{i,r}}
 \]
where the $\gamma_i$ are multiplicatively independent with each $e_{i,j} \in \mathbb Z$, and the torsion multipliers are all powers of $\zeta_M$, a primitive $M$th root of unity, for some $M\geq 1$.  To determine the $\gamma_i$'s and the $e_{i,j}$'s we can take logarithms of all of the multiplicative dependencies between the $\alpha_i$'s and use linear algebra to find a basis.  This  argument  implies that each $\gamma_i$ can be written as an $M$th root of unity  times some product of $\alpha_j$'s.
 
 Therefore  
we can rewrite  \eqref{eq: Formula} as $u_n=U_a(n,\gamma_1^n,\dots,\gamma_r^n)$ for  every $n\equiv a \pmod M$ where
  \begin{equation} \label{eq: Polyn}
U_a(x,y_1,\dots,y_r)= \sum_{i=1}^m g_i(x)  \zeta_M^{e_{i,0}a} y_1^{e_{i,1}}\cdots y_r^{e_{i,r}},
 \end{equation}
and then $u_n$ is determined by this fixed polynomial for all integers $n$ in the congruence class $a\pmod M$. This explains why the period, $M$, is a key element in our considerations.  We  simplify by  combining the terms in \eqref{eq: Polyn} for which the $(e_{i,1},\dots,e_{i,r})$ exponent vectors are identical (though we don't change the notation, so this same sum now has a modified meaning).

The exponents $e_{i,j}$ are integers but some might be negative. Therefore we can write
\[
U_a(x,y_1,\dots,y_r)=x^{e_0} y_1^{e_1}\cdots y_r^{e_r} \, P_a(x,y_1,\dots,y_r)
\]
where     $e_j=\min_i e_{i,j}$ for all $j\geq 1$ with $x^{e_0}$ the largest power of $x$ that divides each $g_i(x)$, so that $P_a$ is a polynomial, with a constant term in each variable.

The set of linear recurrences over $\mathbb Z$ evidently form a commutative ring, and include any periodic sequence of integers since 
\[
1_{n\equiv a \pmod M} = \frac 1M \sum_{\zeta: \zeta^M=1}  \zeta^{-a} \cdot \zeta^n .
\]
Each individual $(U_a(n,\gamma_1^n,\dots,\gamma_r^n))_{n\geq 0}$ ($1\leq a \leq M$)  is a linear recurrence sequence, as they each take the form
 \eqref{eq: Formula}. Moreover we can recover $(u_n)_{n\geq 0}$ from them all by taking
 \[
u_n = \sum_{a=1}^{M}  \,  1_{n\equiv a \pmod M}  U_a(n,\gamma_1^n,\dots,\gamma_r^n).
 \]
Thus we can study the structure of each $U_a(x,y_1,\dots,y_r)$ separately and still recover the whole.

\subsection{Products of linear recurrence sequences}
 If $(u_n)_{n\geq 0}$ and  $(v_n)_{n\geq 0}$ are linear recurrence sequences then write $u_n$ as in \eqref{eq: Formula} and similarly
$v_n= \sum_{j=1}^\ell h_j(n) \beta_j^n$.  Now let the $\gamma_i$ be a multiplicative basis for the non-torsion part of the $\alpha_i$'s and $\beta_j$'s, so we can write $u_n=U_a(n,\gamma_1^n,\dots,\gamma_r^n)$ and $v_n=V_a(n,\gamma_1^n,\dots,\gamma_r^n)$ for all $n\equiv a \pmod M$ for a suitable value of $M$. We then see that if $w_n=u_nv_n$ for all $n\geq 0$ then $(w_n)_{n\geq 0}$ is also a linear recurrence sequence (as we have a representation as in \eqref{eq: Formula} where the roots of its characteristic polynomial are a subset of the $\alpha_i\beta_j$-values).  Moreover we can define
$w_n=W_a(n,\gamma_1^n,\dots,\gamma_r^n)$ where
 \begin{equation} \label{eq: Id}
W_a(x,y_1,\dots,y_r)=U_a(x,y_1,\dots,y_r)V_a(x,y_1,\dots,y_r).
 \end{equation}

The Hadamard quotient theorem states that if $(w_n)_{n\geq 0}$ and  $(u_n)_{n\geq 0}$ are linear recurrence sequences for which
$w_n=u_nv_n$  for all $n\geq 0$ where each $v_n\in \mathbb Z$ then $(v_n)_{n\geq 0}$ is also a linear recurrence sequence. We   therefore deduce that  \eqref{eq: Id} holds.

\section{Polynomial LDS's}   \label{sec: NeedPolyLDS}

We are interested in simple LDS's in $\mathbb C[t]$ (or equivalently in $\mathbb C[x,y]^{\text{homogenous}}$)   in which all of the roots of the characteristic polynomial are multiplicatively dependent on $t$. These are the product of an exponential LDS (which takes account of powers of $t$ in each term) and
a LDS $(u_n(t))_{n\geq 0}$  with $u_0(t)=0$ and $u_n(0)\ne 0$ for all $n\geq 1$ of  period $M$, which can be written as
 \begin{equation} \label{eq: GenPolyform}
u_n(t) =\kappa_a(t) f_a(t^n) \text{ whenever } n\equiv a \pmod M \text{ with } n\geq 1
 \end{equation}
where each $f_a(t)\in \mathbb C[t]$ with $f_a(0)\ne 0$ and $\kappa_a(t)\in \mathbb C(t)$ . In this section we will determine all possibilities, which we will use in determining all polynomially generated LDS's in $\mathbb Z$.

\subsection{Polynomial LDS's of period 1} These must be multiples of
$u_n(t) = f(t^n)/f(t)$ for all $n\geq 1$ with $u_0(t)=0$, where $f(t^m)$  divides $f(t^n)$ in $\mathbb C[t]$ whenever $m$ divides $n$.
We have already seen the examples  $f(t)=t-1$ and    $f(t)= (t+1)(t^3-1)= \lcm[ t^2-1, t^3-1]$ but there are many more:

 \begin{prop} \label{prop: 1 variable} Let $f(x)\in \mathbb C[x]$ with $f(0)\ne 0$. Then
 $f(x^m)$  divides $f(x^n)$ in $\mathbb C[x]$ whenever $m$ divides $n$ if and only if $f(x)$ takes the form
 \[
   \underset{\, m\geq 1}{\lcm}\, (x^{m}-1)^{h_m}
 \]
 for some  $h_m\in \mathbb Z_{\geq 0}$, only finitely many  of which are non-zero.
  \end{prop}   

Here  the lcm is in $\mathbb C[x]$. Any such sequence $(f(x^n))_{n\geq 0}$ is a linear division sequence.

 \begin{proof} Since $f(x)$ divides $f(x^k)$ we see that if $f(\alpha)=0$ then $f(\alpha^k)=0$ for all $k\geq 1$ and so we must have
 $\alpha^j=\alpha^i$ for some $j>i$, else $f(x)$ would have infinitely many distinct roots. Therefore   $\alpha$ is a root of unity, as $\alpha\ne 0$ by the hypothesis.  If $\alpha$ is a primitive $m$th root of unity then $f(\alpha^k)=0$ for every $k,1\leq k\leq m$, and so all the $m$th roots of unity are roots of $f(x)$; that is, $x^m-1$ divides $f(x)$. Taking derivatives one finds  that the multiplicity is the same for each primitive $m$th root of unity, and at least as large otherwise.  Therefore there exists an integer $a_m$ such that $\phi_m(x)^{a_m}\| f(x)$ (where 
  $\phi_m(x)$ is  the $m$th cyclotomic polynomial) and  $a_n\geq a_m$ whenever $n$ divides $m$, so that 
 $f(x)=c  \prod_{m\leq M} \phi_m(x)^{a_m}$. This is a  result of Sergiy  Koshkin \cite{Kosh}   and our result follows
by taking each $b_n=a_n$, and dividing out by $(x-1)^{a_1}$.
  \end{proof}

\subsection{Polynomial LDS's of period M}  

We begin by showing that we can divide through by $t^d-1$ when $(n,M)=d$ in the simplest example;

  \begin{lemma} \label{lem: changePLDS}
  For any   $M\in \mathbb Z_{>1}$ the sequence $u_n(t)=\frac{t^n-1}{t^{(n,M)}-1}$ is a LDS in $\mathbb Z[t]$.  
    \end{lemma}

  \begin{proof}  Write $t^n-1=\prod_{m|n} \phi_m(t)$  where $\phi_m(t)$ is the $m$th cyclotomic polynomial., which has no repeated factors, and let $u_n=(t^n-1)/(t^{(n,M)}-1)$.  We need to show that $u_n$ divides $u_{np}$ for any prime $p$ and integer $n$.
  Let $d=(n,M)$. If $p\nmid M/d$ then $(np,M)=d$ and so $(t^d-1)u_n=t^n-1$ which divides $t^{np}-1=(t^d-1)u_{np}$ and the result follows.
  If $p|M/d$ write $n=Nd$ so that $(M/d,N)=1$ and therefore $p\nmid N$.  
  Any factor $\phi_m(t)$ of either $t^{dp}-1$ or $t^{dN}-1$ must divide $t^{Ndp}-1$. If $\phi_m(t)$ divides both $t^{dp}-1$ and $t^{dN}-1$ then
  $m$ divides $(dN,dp)=d(N,p)=d$ and so $\phi_m(t)$ also divides  $t^{d}-1$. Either way $(t^{dp}-1)(t^{dN}-1)$ divides $(t^{d}-1)(t^{Ndp}-1)$ and so
  $u_n$ divides $u_{np}$.
    \end{proof}

     \begin{lemma} \label{lem: variant1} Let $f(x)\in \mathbb C[x]$ be monic with $f(0)\ne 0$. Then
 $f(x)$  divides $f(x^k)$ in $\mathbb C[x]$ for all integers $k\equiv 1 \pmod M$ 
 if and only if  $f(x)$ takes the form 
  \begin{equation} \label{eq: U(t)-form}
 \underset{\substack{\ m\geq 1 \\ \,  0\leq j\leq M-1 }}\lcm   (\zeta_{M}^j x^m-1)^{h_{m,j}}
 \end{equation}
 for some  $h_{m,j}\in \mathbb Z_{\geq 0}$, only finitely many  of which are non-zero.
 \end{lemma}    
 
 Since $\zeta_{M}^j x^m-1$ divides $\zeta_{M}^{kj} x^{km}-1$, it is convenient to include the terms
 $(\zeta_{M}^j x^m-1)^{h_{km,J}}$ where $J\equiv kj \pmod M$ into the lcm  if $h_{km,J}\ne 0$ (and we will do so henceforth).  Once we do this we have $h_{m,j}\geq h_{km,kj}$ for all integers $k\geq 1$.

\begin{proof} We proceed much like in the proof of Proposition \ref{prop: 1 variable}.
 Suppose $f(\alpha)=0$ with $\alpha\ne 0$ then  $f(\alpha^k)=0$ for all  $k\equiv 1 \pmod M$ so   $\alpha$ is a  root of unity, say a
 primitive $r$th root of unity. Now $\alpha^{1+iM}$ is a root of $f(x)$ for all integers $i$, and the $1+iM \pmod r$ are distinct 
 for $0\leq i< r/m$ where   $m=(r,M)$. Therefore
 \[
 \prod_{i=0}^{r/m-1} (x-\alpha^{1+iM}) = x^{r/m} - \alpha^{r/m} \text{ is a factor of } f(x),
 \]
 and $\alpha^{r/m}$ is a primitive $m$th root of unity, and so an $M$th root of unity.
 Handling powers of factors as in the proof of Proposition \ref{prop: 1 variable}
  the result follows in one direction; and in the other since $(\zeta_M^{h_i})^k=\zeta_M^{h_i}$.
 \end{proof}  
 
 We now  classify simple LDS's in $\mathbb C[t]$:

 \begin{theorem} \label{thm: genF(t)}  
 $(u_n(t))_{n\geq 0}$ in $\mathbb C[t]$     is a simple LDS in $\mathbb C[t]$
of period $M$ for which all of the roots of the characteristic polynomial are multiplicatively dependent on $t$  if and only if it is the product of a periodic LDS, an exponential LDS (in powers of $t$) and some LDS $(v_n(t))_{n\geq 0}$, each in $\mathbb C[t]$ with period dividing $M$, where
  \begin{equation}  \label{eq: vnform}
v_n(t) =   \underset{\substack{\ m\geq 1 \\ \,  0\leq j\leq M/d-1 }}\lcm  
\bigg( \frac{(\zeta_{M}^{j} t^m)^n-1}  {(\zeta_{M}^{j} t^m)^d-1}    \bigg)^{h_{d,m,j}}
 \end{equation}
for all $n$ with $(n,M)=d$, where the $h_{d,m,j}\in \mathbb Z_{\geq 0}$ with only finitely many non-zero,
satisfying the inequalities
\[
h_{D,m,j}\geq h_{d,km,kj \text{ mod } {M/d}} \text{ whenever } d \text{ divides } D, \text{ and } k\geq 1.
\]
 \end{theorem}    
 
 Here  $v_n(t)$ is the lcm of simpler LDS's given by Lemma \ref{lem: changePLDS}: That is, sequences
 $( \frac{y^n-1}  {y^d-1} )^{h_d}$ where $y=\zeta_{M}^{j} t^m$ and $h_d=h_{d,m,j}$ with $h_D\geq h_d$ whenever $d|D|M$.
 
 \begin{proof}
 Let $d=(a,M)$.
 We apply Lemma \ref{lem: variant1} with $f(x)=f_a(x^a)$ and $M$ replaced by $M/d$  for $f_a(t)$ as in 
 \eqref{eq: GenPolyform}, and writing $\zeta_{M/d}=\zeta_M^a$ we obtain
  that 
  \[
u_n(t) =\kappa_a(t)    \underset{\substack{\ m\geq 1 \\ \,  0\leq j\leq M/d-1 }}\lcm  
\bigg( \frac{(\zeta_{M}^{j} t^m)^n-1}  {(\zeta_{M}^{j} t^m)^d-1}    \bigg)^{h_{a,m,j}}
 \]
  whenever  $n\equiv a \pmod M$ as $\zeta_{M/d}^j=\zeta_M^{aj}=\zeta_M^{nj}$, with the definition of $\kappa_a(t)$ adjusted from  \eqref{eq: GenPolyform} (by multiplying through by the denominator here).
  Written like this, $u_d(t)=\kappa_d(t)$ is a polynomial.  
  Moreover $\kappa_d(t)$ divides $\kappa_D(t)$ whenever $d|D|M$.
  
  Let $H=M\cdot \prod (m: \text{ some } h_{*,m,*}\ne 0\}$.
  Select integers $a,b$ with $(a,M)=(b,M)=d$.
  Take any $B\equiv b \pmod M$ and 
  select   integers $A_1\equiv A_2 \equiv a \pmod M$ divisible by $B$
  with $A_1/B, A_2/B$ and $H$ pairwise coprime.
 Therefore
 \[
 ((\zeta_{M}^{j} t^m)^{A_1}-1,(\zeta_{M}^{j'} t^{m'})^{A_2}-1)=
  ((\zeta_{M}^{j} t^m)^{B}-1,(\zeta_{M}^{j'} t^{m'})^{B}-1),
 \]
 and so  $u_B(t)$ divides $(u_{A_1}(t), u_{A_2}(t))$ which divides 
  \[
 \kappa_a(t)    \underset{\substack{\ m\geq 1 \\ \,  0\leq j\leq M/d-1 }}\lcm  
\bigg( \frac{(\zeta_{M}^{j} t^m)^B-1}  {(\zeta_{M}^{j} t^m)^d-1}    \bigg)^{h_{a,m,j}}.
 \]
Therefore each $h_{b,m,j}\leq h_{a,m,j}$ and, by the analogous argument, 
$h_{a,m,j}\leq h_{b,m,j}$, so that each $h_{a,m,j}= h_{d,m,j}$ selecting $b=d$.
Moreover if $B=d$ then 
$u_d(t)=\kappa_d(t)$  divides $(u_{A_1}(t), u_{A_2}(t))=\kappa_a(t)$.

Now take $a=d$ above and select $B_1\equiv B_2\equiv b \pmod D$ with 
$(B_1,B_2)=d$ and $B_1/d, B_2/d$ and $H$ pairwise coprime. Therefore $\kappa_b(t)$ divides $u_b(t)$, which divides 
$(u_{B_1}(t), u_{B_2}(t))$ which, by the last displayed equation, divides $\kappa_d(t)$ since we have
$((\zeta_{M}^{j} t^m)^{B_1}-1,(\zeta_{M}^{j'} t^{m'})^{B_2}-1)=
  ((\zeta_{M}^{j} t^m)^{d}-1,(\zeta_{M}^{j'} t^{m'})^{d}-1)$. Now taking $b=a$ we have 
 $\kappa_d(t)$  divides $\kappa_a(t)$ and vice-versa, so that $\kappa_a(t) = \kappa_d(t)$ times a constant.
 Therefore   the $(\kappa_n(t))_{n\geq 0}$ form a periodic LDS in $\mathbb C[t]$, so we can write 
 \[
 u_n(t) =\kappa_a(t)  v_n(t)
 \]
 when $n\equiv a \pmod M$ with $(n,M)=d$ and $v_n(t)$ as in \eqref{eq: vnform}.

 If $pd$ divides $M$ with $p$ prime then $v_n(t)$ divides $v_{pn}(t)$ whenever $n\equiv d \pmod M$ so, by comparing terms,
 \[
 h_{pd,m,j} \geq \max \{  h_{d,m,j} ,  h_{d,pm,pj} \} = h_{d,m,j}.
 \]
The result follows.
   \end{proof}

We will deduce the full classification of such LDS's in $\mathbb C[t]$ in Corollary \ref{cor: genF(t)}.
\smallskip

One might hope that the LDS's   in \eqref{eq: vnform} might all be of the form \eqref{eq: BuildingLDS}, and therefore be constructed as a product of simpler LDS's. However this is not true as we see from the example with 
$u_n(t)=\frac{t^{3n}-1}{t^n-1}$ if $5\nmid n$, and $u_n(t)=\lcm[ \frac{t^{2n}-1}{t^n-1},\frac{t^{3n}-1}{t^n-1}]=\frac{(t^n+1)(t^{3n}-1)}{t^n-1}$ if $5|n$, since $t^n+1$ is not an LDS. One could however rewrite \eqref{eq: vnform} as the lcm of slightly simpler LDS's.

\section{Vandermonde type determinants}    \label{sec: Vandermonde}

We now introduce a way to study the structure of a linear division sequences.

\begin{prop} \label{pr: coolio} 
If $(x_n)_{n\geq 0}$ is a  recurrence sequence of algebraic integers or polynomials satisfying  \eqref{eq: RecRel} then
\[
\gcd(x_0,x_n,\dots,x_{(k-1)n})  \text{ divides } \gcd(x_0,\dots,x_{k-1})  \prod_{j=1}^m   (n\alpha_j^{n-1}) ^{\binom{k_j}2}\cdot      \prod_{1\leq i<j\leq m} \bigg( \frac{\alpha_i^n-\alpha_j^n}{\alpha_i-\alpha_j} \bigg)^{k_ik_j}.
 \]
\end{prop}

We deduce that if $(u_n)_{n\geq 0}$ is a linear division sequence with $u_1=1$ then 
\[
u_n \text{ divides }   \prod_{j=1}^m   (n\alpha_j^{n-1}) ^{\binom{k_j}2}\cdot      \prod_{1\leq i<j\leq m} \bigg( \frac{\alpha_i^n-\alpha_j^n}{\alpha_i-\alpha_j} \bigg)^{k_ik_j}
 \]
 by Proposition \ref{pr: coolio}, which is the main theorem of B\'ezivin et al  \cite{vP}. (Our proof is a modification of  Barbero \cite{Bar}, which is substantially easier than the proof in \cite{vP}.)

\subsection{Beyond Vandermonde}
We begin by reproving the main results from \cite{FH}, arguably a little more easily.
Let $v(x)$ denote the column vector $(1,x,\ldots,x^{k-1})^T$ and so, differentiating each entry,
\[
v^{(j)}(x)= \frac 1{j!} \bigg( \frac{d}{dx} \bigg)^{j} v(x) =\bigg(0,0,\ldots, \binom{j}j , \binom{j+1}j x\dots , \binom{k-1}j x^{k-1-j} \bigg)^T .
\]
 
\begin{lemma} \label{lem: Vand1}
Let $x_1,\ldots,x_m$ be independent variables and $k_1+\cdots+k_m=k$. Then the determinant of the $k$-by-$k$ matrix $M$ given by appending the column vectors,
\[
M= \big(  v^{(0)}(x_1),\ldots, v^{(k_1-1)}(x_1),    \ldots,  v^{(0)}(x_m),\ldots, v^{(k_m-1)}(x_m) \big)
\]
equals 
\[
\pm   \prod_{1\leq i<j\leq m} (x_i-x_j)^{k_ik_j} .
\]
\end{lemma}

\begin{proof}  For independent variables $x_{i,j}$, we construct   the Vandermonde matrix 
\[
V:= (v(x_{1,0}),\ldots,v(x_{1,k_1-1}),\ldots,v(x_{m,0}),\ldots,v(x_{m,k_m-1})) ,
\]
so that $M= DV  |_{\text{Each } x_{i,j}=x_i}$ where $D$ is the differential operator
 \[
\frac 1{1!} \frac{d}{dx_{1,1}} \cdots  \frac 1{(k_1-1)!}  \bigg( \frac{d}{dx_{1,k_1-1}}\bigg)^{k_1-1}  \cdots 
\frac 1{1!}  \frac{d}{dx_{m,1}}  \cdots   \frac 1{(k_m-1)!}  \bigg( \frac{d}{dx_{m,k_m-1}}\bigg)^{k_m-1}.
 \]
 Since taking the determinant is a linear function in the entries of each column vector,
 \[
 \det M = D \det V\bigg|_{\text{Each } x_{i,j}=x_i} = D \prod_{\substack{1\leq i\leq m \\ 0\leq j\leq k_i-1 \\ (i,j)\ne (i',j') }} (x_{i,j}-x_{i',j'})\bigg|_{\text{Each } x_{i,j}=x_i}
 \]
 We work first with the $x_{1,j}$-variables, the others are analogous. Letting $x_{1,0}=x_1$ our polynomial takes the form
 \[
\pm \prod_{1\leq \ell \leq k_1-1}  (x_{1,\ell}-x_1)\cdot  \prod_{1\leq \ell< j\leq k_1-1} (x_{1,\ell}-x_{1,j}) \cdot \prod_{0\leq \ell \leq k_1-1} \prod_{\substack{2\leq i\leq m \\ 0\leq j\leq k_i-1  }} (x_{1,\ell}-x_{i,j}) \cdot P(x_{i,j}: i\geq 2),
 \]
 where $P(x_{i,j}: i\geq 2)$ is a polynomial in the $x_{i,j}$ with $i\geq 2$.

For $\ell=1$ we have $(x_{1,1}-x_1)$ together with a product of terms $(x_{1,1}-x_{*,*})$. We differentiate wrt $x_{1,1}$ using the product rule and then let $x_{1,1}=x_1$.  Unless we have differentiated out the $(x_{1,1}-x_1)$-term then we get $0$ when we take $x_{1,1}=x_1$. Hence our determinant becomes
 \[
\pm \prod_{2\leq \ell \leq k_1-1}  (x_{1,\ell}-x_1)^2\cdot  \prod_{2\leq \ell< j\leq k_1-1} (x_{1,\ell}-x_{1,j}) \cdot \prod_{0\leq \ell \leq k_1-1} \prod_{\substack{2\leq i\leq m \\ 0\leq j\leq k_i-1  }} (x_{1,\ell}-x_{i,j}) \cdot P(x_{i,j}: i\geq 2)
 \]
 since  $x_{1,\ell}-x_{1,1}=x_{1,\ell}-x_{1}$ for each $\ell\geq 2$.  For $\ell=2$ we differentiate twice wrt $x_{1,2}$ and then take $x_{1,2}=x_1$.
 We see using the product rule that unless this is applied each time to $(x_{1,2}-x_1)^2$ we get $0$; therefore our determinant, now divided by $2!$, becomes
 \[
\pm  \prod_{3\leq \ell \leq k_1-1}  (x_{1,\ell}-x_1)^3\cdot  \prod_{3\leq \ell< j\leq k_1-1} (x_{1,\ell}-x_{1,j}) \cdot \prod_{0\leq \ell \leq k_1-1} \prod_{\substack{2\leq i\leq m \\ 0\leq j\leq k_i-1  }} (x_{1,\ell}-x_{i,j}) \cdot P(x_{i,j}: i\geq 2)
 \]
  since  $x_{1,\ell}-x_{1,2}=x_{1,\ell}-x_{1}$ for each $\ell\geq 3$.   We now proceed by induction on $\ell,  1\leq \ell\leq k_1-1$   applying  $ ( \frac{d}{dx_{1,\ell}} )^{\ell}$ and then letting $x_{1,\ell}=x_1$. When we have finished differentiating with respect to our $x_{1,\ell}$ variables, we end up with 
  \[
\pm      \prod_{\substack{2\leq i\leq m \\ 0\leq j\leq k_i-1  }} (x_1-x_{i,j})^{k_1} \cdot P(x_{i,j}: i\geq 2)
  \]
 Repeating this argument with each variable we end up with the claimed result.
 \end{proof}

 Let 
 \[
 v_{(j)}(x) := \frac 1{j!} \bigg( x\frac{d}{dx} \bigg)^{j} v(x)=\frac 1{j!} (0^j,1^jx,2^j x^2\ldots,(k-1)^jx^{k-1})^T.
 \]
   
 \begin{corollary} \label{cor: Vand1} Let $x_1,\ldots,x_m$ be independent variables and $k_1+\cdots+k_m=k$. Let 
 $M_{k_1,\dots,k_m}(x_1,\dots,x_m)$ be the  $k$-by-$k$ matrix $M$ given by appending the column vectors.
\[
M_{k_1,\dots,k_m}(x_1,\dots,x_m)= \bigg(  v_{(0)}(x_1),\ldots, v_{(k_1-1)}(x_1),    \ldots,  v_{(0)}(x_m),\ldots, v_{(k_m-1)}(x_m) \bigg)
\]
Then 
\[
\det M_{k_1,\dots,k_m}(x_1,\dots,x_m) = \pm    \ x_j^{\binom{k_j}2}\cdot      \prod_{1\leq i<j\leq m} (x_i-x_j)^{k_ik_j} .
\] 
\end{corollary}

\begin{proof} One can show by induction that $v_{(k)}(x) = \sum_{j=0}^k a_{k,j} x^jv^{(j)}(x) $ for some constants $a_{k,j}$. For $k=0$ we have
$a_{0,0}=1$. Otherwise
\[
\sum_{j=0}^k a_{k,j} x^jv^{(j)}(x)  = v_{(k)}(x) = x\frac{d}{dx}v_{(k-1)}(x)  = \sum_{j=0}^{k-1} a_{k-1,j} x\frac{d}{dx} (x^jv^{(j)}(x) )
\]
\[
= \sum_{j=0}^{k-1} a_{k-1,j} (   j x^jv^{(j)}(x) ) + x^{j+1}v^{(j+1)}(x) ) = \sum_{j=0}^k (ja_{k-1,j}+ a_{k-1,j-1}) x^jv^{(j)}(x)
\]
where $a_{k-1,k}=0$ and so $a_{k,j} =ja_{k-1,j}+ a_{k-1,j-1}$ with each $a_{k,k}=1$ by induction. Hence the matrix columns
$(v_{(0)}(x),\ldots, v_{(k-1)}(x))$ can be replaced, in calculating the determinant, by $(v(x), xv^{(1)}(x),\ldots, x^{k-1}v^{(k-1)}(x))$, which equals
$x^{\binom k2}$ times the determinant which now has matrix columns $(v(x), v^{(1)}(x),\ldots,  v^{(k-1)}(x))$. The result follows from Lemma 1.
 \end{proof}

 \begin{proof} [Proof of Proposition \ref{pr: coolio}] 
 Let  $(u_n^{(\ell)})_{n\geq 0}$ for $\ell=0,1,\dots, k-1$ be the recurrence sequences satisfying
 \eqref{eq: RecRel}  with 
 \[
 u_r^{(\ell)}=\delta_{\ell,r} \text{ (the Dirac delta function) for } 0\leq r\leq k-1. 
 \]
 We are interested in the matrix 
 \[
U_n:= 
\begin{pmatrix}
u_0^{(0)}& u_0^{(1)} & \dots &u_0^{(k-1)}\\
u_{n}^{(0)} & u_{n}^{(1)}& \dots &   u_{n}^{(k-1)} \\
 \dots & \dots & \dots & \dots \\
u_{(k-1)n}^{(0)}   & u_{(k-1)n}^{(1)} & \dots & u_{(k-1)n}^{(k-1)} \\
\end{pmatrix} \, .
\]
 If for  convenience we write 
 \[
 u_r^{(\ell)} = \sum_{i=1}^m \sum_{j=0}^{k_i-1} e_{i,j}^{(\ell)} \frac{r^j}{j!} \alpha_i^r .
 \]
 then  $U_n=M_{k_1,\dots,k_m}(\alpha_1^n,\dots,\alpha_m^n)\cdot D_n\cdot E$ where 
\[
E=
 \begin{pmatrix}
e_{1,0}^{(0)} & e_{1,0}^{(1)} & \dots &  e_{1,0}^{(k-1)}\\
e_{1,1}^{(0)}  & e_{1,1}^{(1)} & \dots & e_{1,1}^{(k-1)} \\
 \dots & \dots & \dots & \dots \\
e_{m,k_m-1}^{(0)}  & e_{m,k_m-1}^{(1)} & \dots & e_{m,k_m-1}^{(k-1)}  \\
\end{pmatrix} \text{ and } 
D_n= \begin{pmatrix}
 n^0&   0& \dots &   0\\
0 &  n^1& \dots &  0\\
 \dots & \dots & \dots & \dots \\
0&  0& \dots &  n^{k_m-1} \\
\end{pmatrix}.
\]
Moreover $U_1$ is the identity matrix so that $\det U_1=1$ and therefore
\[
\det U_n= \frac{\det M_{k_1,\dots,k_m}(\alpha_1^n,\dots,\alpha_m^n) } {\det M_{k_1,\dots,k_m}(\alpha_1,\dots,\alpha_m) }
=  \prod_{j=1}^m   (n\alpha_j^{n-1}) ^{\binom{k_j}2}\cdot      \prod_{1\leq i<j\leq m} \bigg( \frac{\alpha_i^n-\alpha_j^n}{\alpha_i-\alpha_j} \bigg)^{k_ik_j} .
\]
by Corollary \ref{cor: Vand1}  (which is Theorem 1 of \cite{Bar}).

 Any recurrence sequence $(x_n)_{n\geq 0}$   satisfying  \eqref{eq: RecRel}  can be written as
 \[
 x_n=\sum_{\ell=0}^{k-1} x_\ell u_n^{(\ell)} \text{ for all } n\geq 0.
 \]
 For a given integer $n$ let $G_n:=\gcd(x_0,x_n,\dots,x_{(k-1)n})$.
 Let $X_{n,\ell}$ denote the matrix $U_n$ with the $\ell$th column replaced by $(x_0,x_{n},\dots,x_{(k-1)n})^T$, so that 
 $G_n$ divides $\det X_{n,\ell}$. We can compute this determinant by subtracting $x_i$  times column $i$ from column $\ell$ for each $i\ne \ell$, which will leave $x_\ell$ times the $\ell$th column of $U_n$. Therefore $\det  X_{n,\ell} = x_\ell \det U_n$ and so $G_n$ divides $x_\ell \det U_n$ for $0\leq j\leq k-1$. Combining these we obtain
 \[
 \gcd(x_0,x_n,\dots,x_{(k-1)n})  \text{ divides } \gcd(x_0,\dots,x_{k-1}) \det U_n
 \]
and the result follows.
\end{proof}

\subsection{Zeros of $(u_n)_{n\geq 0}$}
We deduce an important Corollary in the algebra of linear recurrence sequences.

\begin{corollary} \label{cor: Un=0}  
Let $(u_n)_{n\geq 0}$ be a linear recurrence sequence with period $M$.  If $u_{A+j\ell}=0$ for $j=0,\dots,k-1$
where $A\equiv a \pmod M$ with $1\leq a\leq M$ and $M$ divides $\ell>0$ then $U_a(x,y_1,\dots,y_r)=0$.
\end{corollary}

\begin{proof}  In the notation of section \ref{sec: Period}, let $\delta_i=\gamma_1^{e_{i,1}}\cdots \gamma_r^{e_{i,r}}$
and  $h_i(x)=g_i(A+\ell x)  \zeta_M^{e_{i,0}a} \delta_i^A$ for each $i$, so that $\deg h_i=\deg g_i$. Then the values of $\Delta_i:=\delta_i^\ell =\alpha_i^\ell $ are all distinct, and define
\[
x_j:= \sum_{i=1}^m h_i(x)\Delta_i^j \text{ for all } j\geq 0
\]
so that $x_j=u_{A+j\ell}$ by  \eqref{eq: Polyn}. By hypothesis we have $x_j=0$ for $j=0,\dots,k-1$, which yields a system of $k$ linear equations in the $k$ coefficients of the $h_i$-polynomials.  The corresponding matrix has non-zero determinant by Corollary \ref{cor: Vand1}, and so the $h_i$-coefficients are all $0$, implying the $g_i$-coefficients are all $0$, and therefore $U_a(x,y_1,\dots,y_r)=0$.
\end{proof} 

\begin{corollary} \label{cor: Un=02}  
Let $(u_n)_{n\geq 0}$ be a linear division sequence with   $u_m=0$ for some integer $m\geq 1$.
Then there exist  divisors $d_1,\dots,d_\ell$  of $M$ such that $u_n=0$ if and only if $n$ is divisible by some $d_i$, and 
 $U_a(x,y_1,\dots,y_r)=0$ whenever some $d_i$ divides $a$.
\end{corollary}
  
 \begin{proof}  
Let $d=(m,M)$  so there exists an integer $r$ which is divisible by $m$, with $r\equiv d \pmod M$.
Any integer $n\equiv r \pmod  {[m,M]}$ is divisible by $m$ so that  $u_m=0$ divides $u_n$   and  therefore
 $u_n=0$. Therefore  $U_d(x,y_1,\dots,y_r)=0$ and so $u_d=0$  by Corollary \ref{cor: Un=0}. But then $u_n=0$ whenever $d$ divides $n$, and if $d$ divides $a, 1\leq a\leq M$ then   $u_n=0$ for all $n\equiv a \pmod M$ and so 
 $U_a(x,y_1,\dots,y_r)=0$  by Corollary \ref{cor: Un=0}.
 \end{proof}

 \section{Factorization Theorems}

 Suppose we are given representations  for $u_n,v_n$ and $w_n$ as in \eqref{eq: Polyn}. If
 $w_n=u_nv_n$ for all $n$ in some arithmetic progression of length $k$, with each $n\equiv a \pmod M$ then 
 \eqref{eq: Id} holds, by  Corollary \ref{cor: Un=0}.

In particular if $u_n$ is a linear recurrence sequence for which $u_n$ divides $u_{q n}$ for all $n\geq 0$ for a given prime $q$, then we can write 
$u_{qn}=u_nv_n$ where $(v_n)_{n\geq 0}$ is also a linear recurrence sequence, by the  Hadamard quotient theorem, and therefore
\eqref{eq: Id}  holds with $W_a(x,y_1,\dots,y_r)=U_{qa}(qx,y_1^q,\dots,y_r^q)$ (with $U_{qa}=U_b$ where $b$ is the least positive residue of $qa \pmod M$), so that  
 \begin{equation} \label{eq: Id3}
U_{qa}(qx,y_1^q,\dots,y_r^q)=U_a(x,y_1,\dots,y_r)V_a(x,y_1,\dots,y_r).
 \end{equation}
The polynomial $U_a(x,y_1,\dots,y_r)$ therefore divides the polynomial $U_{qa}(qx,y_1^q,\dots,y_r^q)$. One might guess this can only happen if these polynomials have a particular special shape.
 
 \subsection{Factorizations} \label{sec: FactorBinom}
Our goal is to understand the  factorization of $U_a(x,y_1,\dots,y_r)$ and then $U_a(x,y_1^m,\dots,y_r^m)$ for all $m\geq 1$, given \eqref{eq: Id3}.
We partly develop a theory that culminated in the 1935 paper by MacColl \cite{MacC} (following the papers of 
 Ritt \cite{Rit} in 1927 and Gourin \cite{Gou} in 1930).\footnote{They studied ``exponential polynomials'' by developing a version of Lemma \ref{lem: RGM} for
 $f(x,y_1^{m_1},\dots,y_r^{m_r})$ but here we only need the exponents to be equal.}

To begin with we can divide $U_a(x,y_1,\dots,y_r)$ by the monomial $x^{e_0} y_1^{e_1}\cdots y_r^{e_r}$ (as above), leaving a polynomial  with a constant term in each variable.
Next we have \emph{binomial factors} which are irreducible polynomials of the form
\[
B(x,y_1,\dots,y_r):= x^{a_0} y_1^{a_1}\cdots y_r^{a_r} -c x^{b_0} y_1^{b_1}\cdots y_r^{b_r} 
\]
where  $\min \{ a_i,b_i\}=0$ for each $i$.  If $a_0=b_0=0$ then we can always factor
\[
B(x,y_1^m,\dots,y_r^m) = \prod_{j=0}^{m-1} (y_1^{a_1}\cdots y_r^{a_r} - \zeta_m^j c^{1/m} y_1^{b_1}\cdots y_r^{b_r} )
\]
as a product of binomial factors. If say $a_0>b_0=0$ then we get a factorization if and only if $m$ divides $a_0$, again into binomial factors (and an analogous result if $b_0>a_0=0$). Finally we have  \emph{multinomial factors} which are irreducible polynomials containing at least three monomials. One can show \cite{MacC} that for any irreducible multinomial $f(x,y_1,\dots,y_r)$, each $f(x,y_1^m,\dots,y_r^m)$ factors into irreducible multinomial factors, no binomial factors.

 If $f(x,y_1^m,\dots,y_r^m)=g(x,y_1,\dots,y_r)h(x,y_1,\dots,y_r)$ then we can also write \newline
 $f(x,y_1^{mk},\dots,y_r^{mk})=g(x,y_1^k,\dots,y_r^k)h(x,y_1^k,\dots,y_r^k)$ for all $k\geq 2$, which we call an \emph{old factorization}. A factorization that is not old is a \emph{new factorization}. The key to our results is given by the following:

\begin{lemma} [Ritt-Gourin-MacColl] \label{lem: RGM} If $f(x,y_1,\dots,y_r)\in \mathbb C[x,y_1,\dots,y_r]$ is an irreducible multinomial then
there are only finitely many integers $m$ for which $f(x,y_1^m,\dots,y_r^m)$ has a new factorization.  
\end{lemma}

\begin{corollary}\label{cor: RGM}
Suppose that each irreducible factor of $f(x,y_1,\dots,y_r)\in \mathbb C[x,y_1,\dots,y_r]$ is a  multinomial factor.
There exists an integer $N=N_f$ such that if $g(x,y_1,\dots,y_r)$ is an irreducible factor of $f(x,y_1^N,\dots,y_r^N)$ then
  $g(x,y_1^k,\dots,y_r^k)$ is irreducible for every integer $k\geq 1$.
\end{corollary}

\begin{proof}
Let $\mathcal M(f)$ be the set of integers $m>1$ for which there is a new factorization of $f(x,y_1^m,\dots,y_r^m)$, so that $\mathcal M(f)$ is finite 
by Lemma \ref{lem: RGM}.  Let $N=\lcm[m: m\in \mathcal M(f)]$.  The result follows from Lemma \ref{lem: RGM}.
\end{proof}

\begin{prop}\label{prop: Key result} Let $q$ be a prime with $q\nmid M$.
Suppose that  $(u_n)_{n\geq 0}$ is a linear recurrence sequence in the integers with period $M$, for which $u_n$ divides $u_{qn}$ for every integer $n\geq 1$.
Then the polynomials $U_a(x,y_1,\dots,y_r)$, as defined above,  are each either $0$ or a monomial times a product of binomials none of which contain the variable $x$.
\end{prop}

\begin{proof} Let $k$ be the order of $q$ mod $M$ so that   $q^ka\equiv a \pmod M$ for any $a \pmod M$. Therefore iterating \eqref{eq: Id3}  $k$ times, we obtain
\[
U_a(x,y_1,\dots,y_r) \text{ divides } U_{a}(Qx,y_1^Q,\dots,y_r^Q)
\]
where $Q=q^k$.  Let $f(x,y_1,\dots,y_r)$ be the product of the multinomial irreducible factors of 
$U_a(x,y_1,\dots,y_r)$ 
so that  $f(x,y_1,\dots,y_r)$ divides $f(Qx,y_1^Q,\dots,y_r^Q)$. By  Corollary \ref{cor: RGM}  the polynomials
$f(x,y_1^N,\dots,y_r^N)$ and $f(X,y_1^{QN},\dots,y_r^{QN})$ have the same number of irreducible factors. However since the first divides the second (with $X=Qx$), these irreducible factors must be the same, and so these polynomials must be equal, and have the same degree. However this is impossible unless the degree in each $y_i$ is $0$; in other words,  $f(x,y_1,\dots,y_r)$ is a polynomial in $x$ only, which is not a multinomial. Therefore
$U_a$ is a  monomial times a product of binomials.

Every binomial involving $x$ takes the form $x^{a_0} y_1^{a_1}\cdots y_s^{a_s} - c  y_{s+1}^{a_{s+1}}\cdots y_r^{a_r}$. Let $N$ be the lcm of the $a_0$ from all such factors. Substituting in $y_i^N$ for $y_i$ and factoring leaves only binomial factors of the form
$x  y_1^{b_1}\cdots y_s^{b_s} - C  y_{s+1}^{b_{s+1}}\cdots y_r^{b_r}$, which are all irreducible.  
Let $h(x,y_1,\dots,y_r)$ be the product of these binomials. Then $h(x,y_1,\dots,y_r)$ divides $h(Qx,y_1^Q,\dots,y_r^Q)$ but again both polynomials have the same number of irreducible factors so must be equal, that is of the same degree, again implying that  the degree in each $y_i$ is $0$. Therefore each $h$ is a  polynomial in $x$ only such that $h(x)$ divides $h(Qx)$. But if $\rho$ is the root of largest absolute value of $h(x)$ then $h(Q\rho)=0$ and so $\rho=0$, that is $h(x)$ must be a constant times a power of $x$, so must equal $1$ as this is not a binomial.
\end{proof}

We now present a different way to reach the same conclusion, which will help in the proof of Theorem \ref{thm: main4}.

\medskip

\noindent \textbf{Proposition 3 revisited} \emph{Suppose that  $(u_n)_{n\geq 0}$ is a linear recurrence sequence in the integers with period $M$, for which $u_n$ divides $u_{jn}$ for $1\leq j<k$.
Then the polynomials $U_a(x,y_1,\dots,y_r)$, as defined above,  are each either $0$ or a monomial times a product of binomials none of which contain the variable $x$.}
\medskip

   \begin{proof} Proposition \ref{pr: coolio} implies the result of B\'ezivin,   Peth\H{o} and  van der Poorten \cite{vP}:
 \[
 u_n \text{ divides } \prod_{j=1}^m   (n\alpha_j^{n-1}) ^{\binom{k_j}2}\cdot      \prod_{1\leq i<j\leq m} \bigg( \frac{\alpha_i^n-\alpha_j^n}{\alpha_i-\alpha_j} \bigg)^{k_ik_j}.
 \]
The right-hand side is also a linear recurrence sequence and so, in the discussion and notation of section \ref{sec: Period},  
\[
U_a(x,y_1,\dots,y_r) \text{ divides }  x^{e_0} y_1^{e_1}\cdots y_r^{e_r}
 \prod_{1\leq i<j\leq m}  ( \zeta_M^{e_{i,0}a} y_1^{e_{i,1}}\cdots y_r^{e_{i,r}} - \zeta_M^{e_{j,0}a} y_1^{e_{j,1}}\cdots y_r^{e_{j,r}})^{k_ik_j}.
\]
The right-hand side is a product of monomials and binomials  and only factors into 
 monomials and binomials as we saw there, and therefore $U_a(x,y_1,\dots,y_r)$ factors into 
 monomials and binomials.
\end{proof}

In the original Proposition \ref{prop: Key result} the binomials in $U_a$ are shown to be of the form $\gamma^n-1$ where $\gamma$ can be expressed as a product of $\alpha_i$'s to integer powers so that each $\gamma\in K$. In the modified 
Proposition \ref{prop: Key result} the binomials in $U_a$ come from   binomials that divide some $\alpha_i^n-\alpha_j^n$; that is,   there exist some $i\ne j$ and integer $r$ for which $\alpha_i=\gamma^r\alpha_j$.

 \subsection{The shape of a linear division sequence}

Let $K=\mathbb Q(\alpha_1,\dots,\alpha_m)$, let $A$ be the ring of integers of $K$.

\begin{corollary}\label{cor: Key1}
If  $(u_n)_{n\geq 0}$ is a linear division sequence in the integers with associated period $M\geq 1$ then  
\begin{equation} \label{eq: un-value1}
u_n=  \kappa_a n^{\epsilon_a} \lambda_a^{n-1} \prod_{j=1}^\ell v_{j,n}(\nu_j,\delta_j)\text{ whenever } n\equiv a \pmod M \text{ for } 1\leq a\leq M.
\end{equation}
where the $\kappa_a\in K$ and the $\epsilon_a\in \mathbb Z^+$,
\begin{itemize}
\item The $ \lambda_a\in K$ are products of the $\alpha_i$'s to integer powers;
\item The $\nu_j, \delta_j\in A$  are  products of the $\alpha_i$'s to positive integer powers;
\item  The $\nu_j/\delta_j$ are pairwise multiplicatively independent;
\item The $(v_{j,n}(x,y))_{n\geq 0}$ are multiplicatively-dependent linear division sequences in $A[x,y]^{\text{homogenous}}$,   monic in $x$, with  $v_{j,1}(t)=1$.
\end{itemize}
\end{corollary}

It could be that  $\kappa_a=0$ for some $a$. If so then  there exist  divisors $d_1,\dots,d_\ell$  of $M$ such that 
$\kappa_a=0$  if and only if $a$ is divisible by some $d_i$,  by Corollary \ref{cor: Un=02}.

 
 \begin{proof}
 Proposition \ref{prop: Key result} implies that we can write 
 \[
 U_a(x,y_1,\dots,y_r) = \kappa_a'  x^{e_{a,0}} y_1^{e_{a,1}}\cdots y_r^{e_{a,r}}  \prod_{i=1}^I B_i(y_1,\dots,y_r)
 \]
 where $ B_i(y_1,\dots,y_r):= z_i-c_i$ with $c_i\in \mathbb C^*$ and each $z_i=y_1^{z_{i,1}}\cdots y_r^{z_{i,r}}$ for some $z_{i,j}\in \mathbb Z$.
 Now if $B_i(z_i)$ divides $B_j(z_j^k)$ then $z_i$ and $z_j$ must be multiplicatively dependent, and so 
since the $y_j$ are multiplicatively independent,  
the exponent vectors $(z_{i,1},\dots ,z_{i,r})$ and $(z_{j,1},\dots ,z_{j,r})$ must be linearly dependent.  Since these are integer vectors they must each be an integer multiple of some integer vector. Therefore there exist $w_1,\dots,w_\ell$ which are pairwise multiplicatively independent such that each $z_i$ equals  $w_j^{b_j}$ for some $j$ and for some integer $b_j$ (and so each $w_j=y_1^{w_{i,1}}\cdots y_r^{w_{i,r}}$ where each $w_{i,j}\in \mathbb Z$). We can then write each
\[
U_a(x,y_1,\dots,y_r) = \kappa_a' x^{e_{a,0}} y_1^{e_{a,1}}\cdots y_r^{e_{a,r}} \prod_{j=1}^\ell P_{a,j}(w_j)
\]
where $P_{a,j}$ is a monic polynomial (as it is a product of factors of the form $w_j^{b_j}-c_i$).

Let $\eta_j=w_j(\gamma_1,\dots,\gamma_r)= \gamma_1^{w_{i,1}}\cdots \gamma_r^{w_{i,r}}$ so that for  $\epsilon_a=e_{a,0}$
\[
u_n=U_a(n,\gamma_1^n,\dots,\gamma_r^n) = 
\kappa_a' n^{\epsilon_a} (\gamma_1^{e_{a,1}}\cdots \gamma_r^{e_{a,r}})^n \prod_{j=1}^\ell P_{a,j}(\eta_j^n).
\]

 $P_{a,i}(w_i)$ divides $x^{he_{a,0}} y_1^{he_{a,1}}\cdots y_r^{he_{a,r}} \prod_{j=1}^\ell P_{ha,j}(w_j^h)$ for any $h\geq 1$ and each $a\in \{ 1,\dots,M\}$. 
Writing these   polynomials as products of binomial factors, we see that each $w_i-c$ can only divide $P_{ha,i}(w_i^h)$ and so $P_{a,i}(w_i)$ divides $P_{ha,i}(w_i^h)$. Hence  if we define
$v_{j,n}(t):=P_{a,j}(t^n)/P_{1,j}(t)$ whenever $n\equiv a \pmod M$ then    $(v_{j,n}(t))_{n\geq 0}$ is a multiplicatively dependent linear division sequence in $\mathbb C[t]$
with $v_{j,1}(t)=1$.
Therefore
\[
u_n= 
\kappa_a'' n^{\epsilon_a}(\gamma_1^{e_{a,1}}\cdots \gamma_r^{e_{a,r}})^n \prod_{j=1}^\ell v_{j,n}(\eta_j) \text{ whenever } n\equiv a \pmod M \text{ for } 1\leq a\leq M,
\]
with $\kappa_a''=\kappa_a' \prod_{j=1}^\ell P_{1,j}(\eta_j)$.

The expression for each $\eta_j$ is a product of $\gamma_i$'s to integer powers, which equals a product of $\alpha_i$'s to integer powers.
Therefore we can write $\eta_j=\nu_j/\delta_j$ where $\nu_j$ and $\delta_j$  are each the product of $\alpha_i$'s to non-negative integer powers, and with no 
$\alpha_i$ in common.  If $P_{a,j}$ has degree $d$ ($=d_{a,j}$) then homogenize so that $P_{a,j}(x,y)=y^dP_{a,j}(x/y)$ and therefore
$ P_{a,j}(\eta_j^n) = \delta_j^{-dn} P_{a,j}(\nu_j^n,\delta_j^n)$. Writing $v_{j,n}(x,y)=y^{dn}v_{j,n}(x/y),\ \kappa_a=\kappa_a''\lambda_a$ and 
 $\lambda_a:=\gamma_1^{e_{a,1}}\cdots \gamma_r^{e_{a,r}} \prod_j \delta_j^{-d_{a,j}}$ we obtain the claimed result.
\end{proof}
 
This same proof, with suitable modifications, works more easily for LDS's in $\mathbb C[t]$.
To ``complete'' Theorem \ref{thm: genF(t)} the only part of Corollary \ref{cor: Key1} that we need is that 
any LDS is a power LDS times a simple LDS, and so we deduce the  full classification of LDS's in $\mathbb C[t]$:
 
 \begin{corollary} \label{cor: genF(t)}  
 $(u_n(t))_{n\geq 0}$   is a LDS in $\mathbb C[t]$  for which all of the roots of the characteristic polynomial are multiplicatively dependent on $t$  
 if and only if it is the product of  a power LDS and a periodic LDS,   an exponential LDS  (in powers of $t$) and some LDS $(v_n(t))_{n\geq 0}$ as in \eqref{eq: vnform}, each in $\mathbb C[t]$.
 \end{corollary}

\section{Proof of Theorem \ref{thm: main}}
  
The left-hand side of  \eqref{eq: un-value1}, $u_n$, is an integer whereas the 
 right-hand side is a product of numbers in $K$, so we need to better understand this equality, and  by studying each of the terms  on the right-have side of 
 \eqref{eq: un-value1}, we will  deduce Theorem  \ref{thm: main}.

 \subsection{Proof of Theorem \ref{thm: main}: I. Polynomially generated LDS's}

\begin{proof} [Proof that  $\prod_{j=1}^\ell v_{j,n}(\nu_j,\delta_j)$ is a product of polynomially generated LDS's in $\mathbb Z$] 
Let  $G:=\text{Gal}(K/\mathbb Q)$.  The ideals $(\nu_j,\delta_j)$ are not necessarily   principal   so we cannot just divide them out from 
 $\nu_j$ and $\delta_j$ like we do in the integer case. However if $I$ is a minimal representative of the ideal class of 
 $(\nu_j,\delta_j)$ in the  class group of $K/\mathbb Q$, and $I^c$ is a (minimal) representative of the inverse
  so that $I\cdot I^c=(g)$ and $(\nu_j,\delta_j)\cdot I^c=(h)$  for some $g,h\in A$, then
 $(\frac gh\cdot \nu_j,\frac gh\cdot \delta_j) =I$. We therefore adjust the $\nu_j,\delta_j$ in  \eqref{eq: un-value1} accordingly, which will also lead to us adjusting each $\lambda_a$. We will need that if $I$ is the representative of the ideal class of  $(\nu_j,\delta_j)$ then
 $I^\sigma$ is the representative of the ideal class of  $(\nu_j^\sigma,\delta_j^\sigma)$ for each $\sigma\in G$.

 Let $L=\mathbb Q(\nu_i,\delta_i)$ and then $H$ be a set of representatives of 
 \[
 \text{Aut}(L)/\text{Stab}(\{ v_{i,n}(\nu_i,\delta_i): n\ge 1\}) .
 \]
 For each $\sigma\in H$ we see that 
 \[
\sigma(v_{i,n}(\nu_i,\delta_i)) \text{ divides }    \sigma(u_n)= u_n=  \kappa_a n^{\epsilon_a} \lambda_a^n \prod_{j=1}^\ell v_{j,n}(\nu_j,\delta_j),
 \]
 Now $\sigma(v_{i,n}(\nu_i,\delta_i))$ is the homogenous polynomial $\sigma(v_{i,n}(x,y))$ evaluated at $x=\nu_i^\sigma,y=\delta_i^\sigma$ and so
 $\nu_i^\sigma/\delta_i^\sigma$ must equal some $\nu_j/\delta_j$, and by our choices of gcd in the previous paragraph, we see that 
 $(\nu_i,\delta_i)^\sigma=(\nu_i^\sigma,   \delta_i^\sigma)$. Therefore
 $\sigma(v_{n}(\nu_i,\delta_i))$ divides  $v_{j,n}(\nu_j,\delta_j) = \sigma(v_{i,n}) (\nu_i^\sigma,   \delta_i^\sigma)$.  Therefore the product on the right-hand side of
 \eqref{eq: un-value1} can be partitioned into subproducts of the form
 \[
 \prod_{\sigma\in H}   \sigma(v_{i,n}) (\nu^\sigma,   \delta^\sigma)  .
 \]
 (We remark that each $\sigma(v_{i,1})(x,y)=1$ by Corollary \ref{cor: Key1}.)
 We chose $H$ to be a minimal set of automorphisms for which this product is stable under the action of $G$; moreover, it is a product of elements of $A$ and stable under Galois so it is an integer. 
 
Now  each $(\nu^\sigma,   \delta^\sigma)=I^\sigma$ so we can certainly divide through by $q^{\deg v_{i,n}}$ where the positive integer $q$ is obtained from
the integer ideal $(q)=\prod_{\sigma\in H} I^\sigma$, adjusting the value of $\lambda_a$ accordingly. We therefore obtain  a polynomially generated LDS in $\mathbb Z$.
\end{proof}

\subsection{Sorting out some details}

We can now rewrite Corollary \ref{cor: Key1} as
\begin{equation} \label{eq: un-value2a}
u_n=  \kappa_a n^{\epsilon_a} \lambda_a^{n-1}  \prod_{j=1}^J u_{j,n} \text{ whenever } n\equiv a \pmod M \text{ for } 1\leq a\leq M,
\end{equation}
where each $(u_{j,n})_{n\geq 0}$ is a polynomially generated LDS in $\mathbb Z$, the $\epsilon_a\in \mathbb Z^+$ and
 the $\kappa_a, \lambda_a\in K$ (though not with the same values as in Corollary \ref{cor: Key1}).
 
Since each $u_n, u_{j,n}\in \mathbb Z$ we deduce that each $ \kappa_a  \lambda_a^{n-1}\in  \mathbb Q$.
Therefore $\tau_a:=\lambda_a^M=( \kappa_a  \lambda_a^{a+M-1})/ (\kappa_a  \lambda_a^{a-1})\in  \mathbb Q$, and so taking
$\kappa_a'=\kappa_a \lambda_a^{a-1}\in  \mathbb Q$ we deduce that 
 \begin{equation} \label{eq: un-value2a}
u_n=  \kappa_a' n^{\epsilon_a} \tau_a^{k}  \prod_{j=1}^J u_{j,n} \text{ writing } n=a+kM \text{ for } 1\leq a\leq M, k\geq 0,
\end{equation}
 with the $\kappa_a', \tau_a \in  \mathbb Q$.
 
 Next we need the following lemma which we will prove in  section \ref{sec: P-divisibilty}:
 
 \begin{lemma} \label{PolyGenDivByp} 
  Let   $(u_{n})_{n\geq 0}$ be a polynomially generated LDS in $\mathbb Z$.
 \begin{enumerate} [(a)]
\item  For any prime $p$ there exist constants $a_p,b_p$ such that 
$v_p(u_n) \leq a_p+b_pv_p(n)$. In particular $v_p(u_n) \ll_p \log 2n$.
\item There exists a constant $C=C((u_{n})_{n\geq 0})$ such that if $p$ divides $n$ then 
$p^{v_p(u_n)}\ll (pe^{n/p})^{C}$.
\item There exists a finite set of integers $n_1,\dots,n_k>1$ such that if $p|u_n$ then some $n_i$ divides $n$.
Moreover if $p|u_{np}$ then $p$ divides both $u_n$ and $u_{np}/u_n$.
\end{enumerate}
\end{lemma}

This implies that each $\tau_a\in \mathbb Z$ for if prime $p$ divides the denominator of $\tau_a$ then $v_p( \tau_a^{k} )\leq -k$ whereas for $n=a+kM$ we have 
$v_p(\kappa_a' n^{\epsilon_a} \prod_{j=1}^J u_{j,n})   \ll   \log 2k$ by Lemma \ref{PolyGenDivByp}(a) and so $v_p(u_n)<0$ if $n$ is large and $\equiv a \pmod M$, contradicting that $u_n\in \mathbb Z$. 

Now for each prime $p$ which divides some $\tau_a$, let $h_a=v_p(\tau_a)$. Therefore $v_p(u_{a+kM}) = kh_a+O(\log k)$.
If $b\equiv ra \pmod M$ with $1\leq a,b,r\leq M$ then $u_{a+kM}$ divides $u_{ra+rkM}=u_{b+\ell M}$ where $\ell=rk+ [\frac{ra}m]$, and so
\[
 kh_a+O(\log k)= v_p(  u_{a+kM}) \leq  v_p( u_{b+\ell M} ) =\ell h_b +O(\log \ell)
\]
and so $kh_a\leq rkh_b +O(\log k)$ which implies that $h_a\leq rh_b$ letting $k\to \infty$. This is therefore a specialization of an exponential LDS in $\mathbb C[t]$, and so $(\tau_a^k)_{n=a+kM\geq 0}$ is an exponential LDS in $\mathbb Z$.
  
To appreciate the $\epsilon_a$-values, let $p$ be a prime  that does not divide the numerator or denominator of any non-zero $\kappa_a'$ or $\tau_a$.  Fix $a$ and $r$ with $b\equiv ra \pmod M$ and assume also $p\nmid  ar\prod_{j=1}^J u_{j,a} u_{j,ra}$ and so $p$ does not then divide $\prod_{j=1}^J u_{j,ap}u_{j,rap}$ by Proposition \ref{prop: P-div}(c).

  Now $u_{ap}$ divides $u_{arp}$ and so
  $\epsilon_a=v_p(u_{ap})\leq v_p(u_{arp}) =\epsilon_b $, looking at the exponent of $p$.  Now if $(a,M)=d$ then $\epsilon_d\leq \epsilon_a$ writing $a=Ad$, and as $d\equiv ra \pmod M$ for some $r$, we also have $\epsilon_a\leq \epsilon_d$. Therefore $\epsilon_a=\epsilon_d$ where $d=(a,M)$.  We also note that if 
  $ D=dr|M$ then $\epsilon_d\leq \epsilon_D$.  Letting  $\kappa_a=d^{\epsilon_d}\kappa_a'$, so $\kappa_a' n^{\epsilon_a}=\kappa_a (n/d)^{\epsilon_d}$,  the $\kappa_a$ form a periodic sequence of rationals, and the $(n/d)^{\epsilon_d}$ are a power LDS.
  
    This completes the proof of Theorem \ref{thm: main}. \hfill \qed


\subsection{Proof of Corollaries \ref{thm: main6} and \ref{thm: main2}}

\begin{proof}  [Proof of Corollary \ref{thm: main6}]
Suppose that $(u_n)_{n\geq 0}$ has period $M$. By Theorem \ref{thm: main}
\[
u_n=\kappa_1 n^e \lambda^{n-1} \prod_j u_{j,n} \text{ for all } n\equiv 1 \pmod M
\]
where $e=e_1$ and $\lambda^M\in \mathbb Z$, and each $u_{i,n}=(w_{i,n}(\gamma_i))^*$ with
$w_{i,n}(t)= f_n(t)/f_1(t)$ for some $f_n(t)$ as in \eqref{eq: f(t)inM-defn}. Moreover each of the factors  in \eqref{eq: f(t)inM-defn} appears in the 
construction of $w_{i,n}(t)$ in every arithmetic progression mod $M$.  Multiplying the expression in 
\eqref{eq: f(t)inM-defn}  out gives an expression for $u_n$ of the form
\[
u_n=  n^e \sum_{i=1}^m g_i  \zeta_{M}^{e_in} \delta_i^n
\]
with $e_i\in \mathbb Z$ for every $n\equiv 1 \pmod M$.  Now $(u_n)_{n\geq 0}$ is   non-degenerate so each term in this sum appears in the expression for $u_n$ for all $n$, never getting cancelled by another term. Therefore every $u_n$ can be given by the same expression as  for $n\equiv 1 \pmod M$, and the result follows, since $u_1=1$.
 \end{proof}

 We would like to know when  expressions like \eqref{eq: f(t)inM-defn}, multiplied out, are non-degenerate, as well as their products when $t$ is replaced by $\gamma_i$. In particular we would like to simplify the result in Corollary \ref{thm: main6} if possible. However, for any integer $M$ and $a(x)\in \mathbb Z[x]$, when we multiply out
 \[
  \prod_{\substack{1\leq j\leq M \\ (j,M)=1}} ( (\zeta_M^j a(\zeta_M^j))^n-1) ( a(\zeta_M^j)^{2n} -1)
\]      
there are typically no cancelations and so it seems unlikely that one can easily obtain a simpler answer than we have given.  For example, for $M=3$ and $\alpha\in \mathbb Z[\omega]$,
 \[
    ( (\omega \alpha)^n-1) ((\overline{\omega\alpha})^n-1) ( \alpha^{2n} -1)( \overline{ \alpha}^{2n} -1)
\]  
is non-degenerate and has 16 distinct terms, as long as  $\overline{ \alpha}/ \alpha$ is not a root of unity.\footnote{If $\overline{ \alpha}/ \alpha$ is a root of unity then $\alpha=r\zeta $ or $r\zeta \sqrt{-3} $ where $r\in \mathbb Z$ and $\zeta\in \{\pm 1,\pm \omega, \pm \overline{\omega})$.}
      
 \begin{proof}  [Proof of Corollary \ref{thm: main2}]
 Theorem \ref{thm: main} shows what LDS's in $\mathbb Z$ look like.
 Moreover we can only have $\limsup_{n\to \infty:\ u_n\ne 0} v_p(u_n)/n =0$ for every prime $p$ if each $\tau_a=1$, by 
 Lemma \ref{PolyGenDivByp}(a). Therefore     for $n\equiv a \pmod M$ with $(a,M)=d$ we have
  \begin{equation} \label{eq: un-value2c}
u_n=  \kappa_a (n/d)^{\epsilon_d}   \prod_{j=1}^J u_{j,n}  
\end{equation}
 the product of a periodic sequence of rationals $\kappa_a$, a power LDS,   and a finite number of polynomially generated LDS's. Now  $u_d=\kappa_d$ whenever $d$ divides $M$. Hence if $d$ divides $M$ then $\kappa_d\in \mathbb Z$ and 
 $\kappa_d$ divides $\kappa_D$ if $d|D|M$.

 If prime $p$ divides the denominator of some $\kappa_a$ then $p$ divides $(n/d)^{\epsilon_d}   \prod_{j=1}^J u_{j,n}$ whenever $n\equiv a \pmod M$, which happens only if each such $n$ is divisible by at least one of $n_1, \dots,n_k>1$  (this list should include $p$ for divisibility of $n/d$) by the last part of Lemma \ref{PolyGenDivByp}(c).   However one can construct an integer $n\equiv a \pmod M$ that is not divisible by any element of the list except divisors of $d=(a,M)$.  Therefore $\kappa_a\in \mathbb Z$   if $(a,M)=1$.   For other $d$ note that  $(u_{rd}/u_d)_{r\geq 0}$ is a LDS of period $M/d$
 and so   the constant  $\kappa'_{a/d}$ ($=\kappa_a/\kappa_d$) $ \in \mathbb Z$ and then $\kappa_a\in \mathbb Z$ and is divisible by $\kappa_d$.

Now suppose that $d=(a,M)\equiv ar \pmod M$ for some integer $r \pmod {M/d}$.  Therefore if $R\equiv r \pmod {M/d}$ then $u_a$ divides $u_{Ra}$ which implies that 
\[
\kappa_a \text{ divides } \kappa_d R^{\epsilon_d}  v_R, \text{ where } v_R:=  \prod_{j=1}^J \frac{u_{j,Ra}}{u_{j,a}} 
\]
 is a polynomially generated LDS. Suppose that prime $p$ divides $\kappa_a$. There exists an integer $L$ (the product of the $n_i$'s in Lemma \ref{PolyGenDivByp})(c)
 such that if $(R,L)=1$ then 
 $p\nmid v_R$ by  Lemma \ref{PolyGenDivByp}(c). Therefore we select an integer $R$ for which $(R,pL)=1$ and $R\equiv r \pmod {M/d}$, which is possible as $(r,M/d)=1$, and we deduce that $p\nmid  R^{\epsilon_d}  v_R$. Therefore $\kappa_a$ divides $\kappa_d$. We conclude that $(\kappa_a)_{a\geq 0}$ is a periodic LDS.
\end{proof}

\subsection{Determining the ``factors'' (of Theorem \ref{thm: main}) of a given LDS} \label{sec: FindFactors}

 Given a LDS $(u_n)_{n\geq 0}$ in the form \eqref{eq: Formula} we first determine its period and then create an explicit formula for $u_n$ in each arithmetic progression mod $M$ which takes account of any cancelation of terms in \eqref{eq: Formula} for that progression. We can then  easily determine the periodic LDS (from the $u_d$-values for $d|M$), 
  the power LDS (from the largest power of $n$ in the explicit formula dividing each term when $(n,d)=M$), and the  exponential  LDS (from the gcd of the uncanceled $\alpha_i$'s). 
  We are left with a LDS $(v_n)_{n\geq 0}$ which is the product of   polynomially generated LDS's and so the $g_i$'s in the formula  \eqref{eq: Formula} for $v_n$ are all constants.
 
  Each of these polynomially generated LDS's comes from
  taking  the  lcm of   factors of the form $(\eta^n-1)^*$.
 To complete our understanding we want to determine these   $\eta$-values (and the powers to which they appear) from the formula  \eqref{eq: Formula} assuming $(u_n)_{n\geq 0}$ is a polynomially generated LDS.
 (We replace $v_n$ by $u_n$ for notational convenience.)

 If $\mathcal I\in \mathcal I^*$ then there exists an algebraic number $\eta_{\mathcal I}$ such that
 \[
  \alpha_i=\zeta_M^{a_i} \alpha_j \eta_I^{r_i} \text{ for each }  i\in I_j, i\ne j,  1\leq j\leq h
 \]
 where the $a_i, r_i\in \mathbb Z$ with $\gcd_i r_i=1$ and some $a_i=0$ (by Lemma \ref{lem: premultdep}). Select $d_{\mathcal I}$ to be the smallest positive integer for which each $\zeta_M^{a_id_{\mathcal I}}=1$ so that if $d_{\mathcal I}$ divides $n$ then
 \[
 \gcd( (\alpha_i/\alpha_j)^n-1 :  i\in I_j, i\ne j,  1\leq j\leq h)^* = (\eta_{\mathcal I}^n-1)^*
 \]
by Corollary \ref{Cor: multfep}.
In this case we have
\[
u_n =  \sum_{j=1}^h \alpha_j^n   f_{I,j}(\eta_{\mathcal I}^n) \text{ where } f_{I,j}(t):= \sum_{i\in I_j} g_i t^{r_i} \in K[t]
\]
with $f_{I,j}(1)=g_{I_j}=0$. Let $e_{\mathcal I}\geq 1$ be the largest integer for which 
\[
(t-1)^{e_{\mathcal I}} \text{ divides }  f_{I,j}(t) \text{ for } 1\leq j\leq h,
\]
so that if $d_{\mathcal I}$ divides $n$ then $(\eta_{\mathcal I}^n-1)^{e_{\mathcal I}}$ divides $u_n$.
 
Partition the set $\mathcal I^* $ into subsets $S_1\cup\cdots \cup S_\ell$ where
the $\eta_{\mathcal I}, \mathcal I \in S_j$ are multiplicatively dependent (and independent of the $\eta_{{\mathcal I}'}$ for all  ${\mathcal I}'\not \in S_j$). Therefore there exists an  algebraic integer $\eta_j$ such that 
$\eta_{\mathcal I}=\zeta_M^{\tau_{\mathcal I}} \eta_j^{g_{\mathcal I}}$ 
where $\tau_{\mathcal I}, g_{\mathcal I}\in \mathbb Z$ for all $\mathcal I \in S_j$, with $\gcd_{\mathcal I \in S_j} g_{\mathcal I}=1$. Finally let
\[
h_{j,n}(t) = \underset{\substack{\  \mathcal I \in S_j \\ \,  d_{\mathcal I} \text{ divides } d }}\lcm \bigg(   
\frac{(\zeta_M^{\tau_{\mathcal I}} t^{g_{\mathcal I}})^n-1}{(\zeta_M^{\tau_{\mathcal I}} t^{g_{\mathcal I}})^{d}-1}    \bigg)^{e_{\mathcal I}} \text{ where } d=(n,M).
\]

Now $(\zeta_M^{\tau_{\mathcal I}} t^{g_{\mathcal I}})^n-1=(\eta_{\mathcal I})^n-1$, and so each
$(\frac{(\zeta_M^{\tau_{\mathcal I}} t^{g_{\mathcal I}})^n-1}{(\zeta_M^{\tau_{\mathcal I}} t^{g_{\mathcal I}})^{d}-1}    )^{e_{\mathcal I}}$ divides $u_n$, and therefore the polynomial-lcm, given by $h_{j,n}(\eta_j)$ also divides $u_n$.  Our goal is to prove that 
 \begin{equation} \label{eq: un-structure}
u_n= \pm \prod_{j=1}^\ell h_{j,n}(\eta_j)^*,
\end{equation}
and we have proved that the right-hand side divides the left since the $\eta_j$ are multiplicatively independent
they give rise to distinct polynomially generated LDS's.. Since $u_n$ is a polynomially generated LDS, we now  complete the proof of \eqref{eq: un-structure} by showing that we have already accounted for all factors of $u_n$ of the form $\eta^n-1$.

So suppose $\eta^n-1$ divides $u_n$ for all $n$ divisible by $d$ (where $d$ divides $M$). By the Hadamard quotient theorem applied to $(u_{dm})_{m\geq 0}$ there exists another linear recurrence sequence $(v_n)_{n\geq 0}$ for which 
$u_n=(1-\eta^n)v_n$ whenever $n\equiv a \pmod M$. We write  $(v_n)_{n\geq 0}$ in the form \eqref{eq: Formula}, 
say $v_n=\sum_{i} h_i(n) \beta_i^n$ partitioning these terms into sets $I_1^-\cup\dots \cup I_\ell^-$ 
so that the $\beta_i/\beta_j$ is multiplicatively dependent with $\eta$ for all $i,j\in I_j^-$, where we 
 select $\beta_j\in I_j^-$ so that  if $\beta_i/\beta_j=\zeta_i \eta^{r_i}$ then $r_i\geq 0$ for all $i\in I_j^-$. Therefore
\[
v_n = \sum_{j=1}^\ell \beta_j^n \cdot \sum_{i\in I_j^-} h_i(n)(\beta_i/\beta_j)^n \text{ and } 
u_n =  \sum_{j=1}^\ell \beta_j^n \cdot \sum_{i\in I_j^-} h_i(n)(\beta_i/\beta_j)^n(1-\eta^n),
\]
the only cancelation that can take place when multiplying through by $\eta^n-1$ is among the terms in each $I_j^-$
and so $\beta_j$   cannot be canceled as it contains $\eta$ to the lowest degree. We therefore write $\alpha_j=\beta_j$ and $g_j(n)=h_j(n)$ for $1\leq j\leq \ell$. When we multiply out the $j$th polynomial here we have
\[
\sum_{i\in I_j^-} h_i(n)(\beta_i/\beta_j)^n(1-\eta^n) = \sum_{i\in I_j^-} h_i(n)(\zeta_i \eta^{r_i})^n 
- \sum_{i\in I_j^-} h_i(n)( \zeta_i \eta^{r_i+1})^n;
\]
which yields the formula \eqref{eq: Formula}. Therefore
  the $\alpha$-values from \eqref{eq: Formula} that give the set $I_j$ are a subset of 
$\{ \zeta_i \eta^{r_i}\beta_j: i\in I_j^-\} \cup \{ \zeta_i \eta^{r_i+1}\beta_j: i\in I_j^-\}$
where we remove $\zeta_{i'} \eta^{r_{i'}}\beta_j =  \zeta_i \eta^{r_i+1}\beta_j$
if and only if $h_i(n)=h_{i'}(n)$. Note that the ratios $\alpha_I/\alpha_J$ are all ratios of the form 
$\beta_i/\beta_j \eta^{-1, 0 \text{ or } 1}$ and so are multiplicatively dependent on $\eta$ and therefore each other.
Moreover $g_{I_j}=  h_{  I_j^-} - h_{  I_j^-} =0$, and so we have proved that 
$I_1\cup\dots \cup I_\ell \in \mathcal I^*$, which means that $\eta^n-1$ was already accounted for in our construction of the right-hand side of \eqref{eq: un-structure}.

\subsection{The $P$-divisibility of polynomially generated LDS's} \label{sec: P-divisibilty}

Throughout this subsection we are given a prime ideal $P$ dividing a rational prime $p$, and an algebraic number $\gamma$, not $0$ or a root of unity nor divisible by $P$





\begin{prop} \label{prop: P-div}  
  For each integer $e\geq 1$ let $g_e\geq 1$ be the order of $\gamma \pmod {P^e}$ so that if $v_P(\gamma^j-1)\geq e$ then $g_e$ divides $j$.
\begin{enumerate}[(a)]
\item If $v_P(\gamma^j-1)\geq e>\frac{ v_P(p)}{p-1}$  then
$v_P(\gamma^j-1)=v_P( \gamma^{g_e}-1)+v_P(j/g_e)$.
\item If $P$ divides $\gamma^j-1$ then $P$ divides $\frac{ \gamma^{jr}-1}{ \gamma^j-1}$ if and only if $p$ divides $r$.
\item If $P$ divides $\gamma^{ap}-1$ then $P$ divides $\gamma^{a}-1$ and so $p\nmid g_1$.
\end{enumerate}
\end{prop}

\begin{proof}  Let $\gamma^j=1+\pi$ where $\pi\in P^e\setminus P^{e+1}$.  
Therefore $\gamma^{jr}=(1+\pi)^r\equiv 1+r\pi \pmod {P^{2e}}$, and so 
$v_P(\gamma^{jr}-1)>e$ if and only if $p|r$.   If, say $r=p$ then 
\[
\gamma^{jp}-1= \sum_{j=1}^{p-1} \frac 1j  \binom {p-1}{j-1} \cdot p \pi^j + \pi^p  \equiv  p \pi \pmod {P^{e+1+v_P(p)}}
\]
provided $pe=v_P( \pi^p)>v_P(p\pi)=e+v_P(p)$, that is $e>\frac{ v_P(p)}{p-1}$. We deduce that if $p\nmid r$ then 
$v_P( \gamma^{jrp^f}-1 )=e+fv_P(p)=e+v_P(p^f)=e+v_P(rp^f)$. 

If $P$ divides $\gamma^{ap}-1$ then  $P\nmid \gamma$ and so there exists $f\geq 1$ such that $\gamma^{p^f-1}\equiv 1 \pmod P$. Hence
  $\gamma^a \equiv  \gamma^{a p^f} \equiv  (\gamma^{ap})^{ p^{f-1}}\equiv 1 \pmod P$ so that $P|\gamma^a-1$.\end{proof}

\begin{corollary} \label{cor: P-div}  Let $w_n=\frac{ \gamma^{n}-1}{ \gamma-1}$ where $\gamma\in \overline{\mathbb Q}^*$ but not a root of unity.
There exists  $n_P>1$ such that if $w_n\equiv 0 \pmod P$ then $n_P$ divides $n$. Moreover
$v_P(w_n)\ll_{P,w} 1+ \cdot v_P(n)$.
\end{corollary}

\begin{proof} 
This follows from Proposition \ref{prop: P-div} with $n_P=g_1$ unless $g_1=1$, in which case $n_P=p$.
\end{proof}

\begin{proof} [Proof of Lemma \ref{PolyGenDivByp}] Let $P$ be a prime ideal dividing $p$.
By the formula in \eqref{eq: f(t)inM-defn} and the description of  polynomially generated LDS's, if $p^e$ divides $u_n$ then $P^e$ divides 
the numerator of a fixed product of terms of the form
\[
    \frac{\zeta_M^{nb}\gamma^{nk}-1}{\zeta_M^{b}\gamma^k-1}  
\] 
for some algebraic number $\gamma$ that is not a root of unity and not divisible by $P$ (by bounding the lcm's in the definition by a product).  Therefore Lemma \ref{PolyGenDivByp}(a) and the first part of (c) follow from applying
Corollary \ref{cor: P-div} to each such term.   The rest of Lemma \ref{PolyGenDivByp}(c) follows from Proposition \ref{prop: P-div}(b) and (c).

We may assume that $p>1+ \deg (K/\mathbb Q)$ else the result follows from (a).
Then $\frac{ v_P(p)}{p-1}<1$ as   $v_P(p)\leq \deg (K/\mathbb Q)$.
Therefore if  $P^v$ divides $\gamma^n-1$ where $p|n$ then $P$ divides $\gamma^{n/p}-1$ by 
Proposition \ref{prop: P-div}(c), and then $v=v_P(\gamma^{n/p}-1)+v_P(p)$ by 
Proposition \ref{prop: P-div}(a). Multiplying together the conjugates yields
Lemma \ref{PolyGenDivByp}(b).  
\end{proof}

\section{Strong linear division sequences}
A strong linear division sequence is a linear recurrence sequence  for which $(u_n,u_m)=|u_{(n,m)}| \text{ for all integers } m,n\geq 1$.
Therefore if $m$ divides $n$ then  $u_m|u_n$ so $(u_n)_{n\geq 0}$ is a LDS, and so we can start  understanding strong LDS's by using   the classification of LDS's in Theorem \ref{thm: main}.

Now if $u_n=a_nb_n$ where $a_n,b_n$ are LDS's then
\[
u_{(n,m)} = a_{(n,m)}b_{(n,m)} \text{ divides } (a_m,a_n)\cdot (b_m,b_n) \text{ which divides }
(u_m,u_n)=u_{(n,m)}
\]
which implies that $a_n$ and $b_n$ are also strong LDS's. By factoring like this we can rule out many possibilities for strong LDS's, by ruling out possible factors:

\subsection*{No strong exponential LDS's} We have $u_{a+kM}=\kappa_a \tau_a^k$.  Taking $k$ large we have $d=(a,M)=(a+kM, a+(k+1)M)$ so that 
$u_d=(\kappa_a \tau_a^k, \kappa_a \tau_a^{k+1}) = \kappa_a \tau_a^k$ and so $\tau_a=1$.

This implies that   strong LDS's are part of the classification in Corollary \ref{thm: main2}, and so must be a product of a periodic LDS, a power LDS and a finite number of polynomially generated LDS's.

\subsection*{Periodic LDS's}
Suppose that  $u_n$ is a strong periodic LDS. For any divisors $d,D$ of $M$ we have
$(u_d,u_D)=u_{(d,D)}$. One can easily write out precise criteria for when this happens.

\subsection*{The only strong power LDS's are a fixed power of $u_n=n/d$} Here  $u_n=   (n/d)^{e_d}$ where $(n,M)=d$. Suppose that $d\ne D$ both divide $M$ with $e_D<e_d$.
We can select $a$ and $b$ to be primes for which $e_D<a\leq b< ae_d/e_D$. 
Let $g=(d,D)$ and $p$ be a prime not dividing $M$. Then
\[
p^{ae_g}= u_{p^ag}= (u_{p^ad}, u_{p^bD}) = ( p^{ae_d}, p^{be_D}) = p^{be_D},
\]
so that $ae_g=be_D$, which is impossible as $(a,b)=1$ and $a\nmid e_D$ as $a>e_D$. We deduce that the $e_d$ are all equal and so
the claim holds.

\subsection*{No strong product of two independent polynomially generated LDS}
Let $u_n=(\alpha^n-1)(\beta^n-1)$ where  $\alpha$ and $\beta$ are  non-zero multiplicatively independent algebraic numbers in $K$.  We will show in Proposition \ref{prop; cebapple} of Appendix \ref{App A} that there exist primes $q$ and $r$ and prime ideals $P$ of $K$ with $v_P(\alpha)=v_P(\beta)=0$ such that 
$\alpha$ is a $q$th power mod $P$, but not $\beta$, and $\beta$ is an $r$th power mod $P$, but not $\alpha$, by 
applying the \v Cebotarev density theorem    to the factorization of 
\[
(x^q-\alpha)(x^q-\beta)(x^r-\alpha)(x^r-\beta) \pmod P.
\]
Therefore  $q$ divides $b:=\text{ord}_P(\beta)$, but not $a:=\text{ord}_P(\alpha)$, so $b$ does not divide $a$; and, similarly,  $r$ divides $a$, but not $b$, so $a$  does not divide $b$.  Now  $P$ divides $u_a$ and $u_b$, so if
$(u_n)$ is a strong LDS then $P$ divides $u_g$ where $g=(a,b)<a,b$.
 However this implies that $P$ divides $\alpha^g-1$ or $\beta^g-1$ which is false as $g<a=\text{ord}_P(\alpha)$ and $g<b=\text{ord}_P(\beta)$.

\subsection*{No strong product or  lcm of two  polynomial   LDS's}
Let $u_n=[t^{an}-1, t^{bn}-1] $ or $(t^{an}-1)(t^{bn}-1)$ with $a<b$; we can take $(a,b)=1$ if necessary by replacing $t^g$ by $t$ where $g=(a,b)$.
Now $t^{ab}-1$ divides $(u_a,u_b)=u_1$ which divides $(t^{a}-1)(t^{b}-1)$. This is only possible if $b=ab$ that is, if $a=1$.\footnote{In our earlier  LDS example $u_n=[F_{2n}, F_{3n}]$ we have $F_6=8$ divides $(u_2,u_3)$ but   not $u_1=2$.}
In the lcm case $u_n= t^{bn}-1$ so we only need consider the product $(t^{n}-1)(t^{bn}-1)$.

If we had to consider a pair $\zeta_M^{in} t^{an}-1,\zeta_M^{jn} t^{bn}-1$ then we let $n=MN$ and $T=t^M$, and reduce to to the pair $T^{aN}-1,T^{bN}-1$ as above.

\subsection*{No strong product of polynomially generated LDS's}
We reduced to the case of a strong LDS of the form $u_n=(\gamma^{n}-1) (\gamma^{bn}-1)$ in the previous paragraph. Suppose  that  prime ideal $P$ divides $\gamma^{br}-1$  but not  $ b $, and so $P\nmid \gamma^{r}-1$
by Proposition \ref{prop: P-div} (else $p|b$). Then $P(\gamma^{br}-1)$ divides $u_{br}$. Moreover 
$p $ divides $ \frac {\gamma^{pbr}-1}{\gamma^{br}-1}$ by Proposition \ref{prop: P-div}, and so
$P(\gamma^{br}-1)$ divides $\gamma^{pbr}-1$ which divides $u_{pr}$. Therefore
$P(\gamma^{br}-1)$ divides $(u_{br},u_{pr})=u_r$, so that $P$ divides $\gamma^{r}-1$ which is false.

\subsection*{Simple polynomially generated LDS's}
We have  proved that any  polynomially generated LDS that divides a strong LDS must be the power of one 
polynomially generated LDS,  of the form $(\frac{\gamma^n-1}{\gamma-1})^*$. Let $\gamma=\alpha/\beta$ with $N((\alpha,\beta))$ as small as possible,
so that the power of a polynomially generated LDS takes the form $L_n(\alpha,\beta)^m$ where $L_n(x,y)=\frac{x^n-y^n}{x-y}$, perhaps divided by some $q^{n-1}$ and lies in $\Z$.
By Lemma \ref{lem: Lehmer} we know that either $L_n(\alpha,\beta)$ is a Lucas sequence in the integers, or a Lehmer sequence (as in section \ref{sec: conjugates}) with $m$ even.

 
\subsection*{Summary so far}
By  Corollary  \ref{thm: main2} we therefore deduce that if $u_n$ is a strong LDS of period $M$ in the integers it takes the form
\[
u_n = \kappa_d (n/d)^r \bigg(  \frac{\alpha^n-\beta^n}{\alpha^d-\beta^d} \bigg)^s \text{ where } d=(n,M),
\]
$r$ and $s$ are integers $\geq 0$,   $ \kappa_n$ is a strong periodic LDS of period $M$ in $\Z$ and the final term is a power of a  (suitably normalized)
Lucas sequence   or a  Lehmer sequence

\subsection*{At least one of $r$ and $s$ equals $0$} Suppose not so that $p|u_p$ if $p\nmid M$. 
If $\alpha,\beta\in \Z$ let  $d=\alpha\beta$ and $e=p-1$; if not, suppose that  $d=(\alpha-\beta)^2$ for Lucas sequences, and $d=rs$ for Lehmer sequences
and let $e=p-(\frac dp)$.  Now select $p> Md\cdot \frac{\alpha^{M}-\beta^{M}}{\alpha-\beta}$ so that 
 $p$ divides $\frac{\alpha^{e}-\beta^{e}}{\alpha^D-\beta^D}$ where $D=(e,M)$, which divides $u_{e}$. Hence
$p$ divides $(u_p,u_{e})=(u_p,u_{p\pm 1})=u_1=1$, a contradiction.

\subsection*{Divisibility criteria} 
Let  $D$ be the smallest integer for which a given prime power $p^e$ divides $u_D$.
If  $p^e$ divides $u_n$ then $p^e$ divides $(u_n,u_D)=\pm u_{(n,D)}$; by minimality $D\leq (n,D)$ and so $D$ divides $n$.
The last statement in the result follows.

 If $u_n=\kappa_d (n/d)^r $ with $r\geq 1$ and $p^e|\kappa_M$ then $D$ divides $M$ and $D$ divides $n=p^f$ when $f$ is sufficiently large, so that 
 $D$ must be a power of $p$. Hence each $\kappa_{p^i}$ is a power of $p$, and so  $\kappa_D = \prod_{p^m\| D} \kappa_{p^m}$, that is,
 $\kappa$ multiplicative.
  
 This completes the proof of Theorem \ref{thm: main3}. \hfill \qed

\subsection{Failing to generalize the Lehmer sequence} \label{sec: Lehmerfail}
  Let $L_n(x,y)=\frac{x^n-y^n}{x-y}$, and note that if $\alpha/\beta$ is not a root of unity then  $L_n(\alpha,\beta)\ne 0$ for all $n\geq 1$.
  We are interested in when $L_n(\alpha,\beta)^m\in \Q$ for all $n\geq 1$. We already know about the Lucas sequences for $m=1$ (where $\alpha,\beta$ are integers, or conjugate quadratics) and the Lehmer sequences of section \ref{sec: conjugates} for $m=2$. We now show there are essentially no other examples:

\begin{lemma}  \label{lem: Lehmer} Suppose that $\alpha$ and $\beta$ are non-zero algebraic integers and $\alpha/\beta$ is not a root of unity, for which  $L_n(\alpha,\beta)^m\in \Z$ for all $n\geq 0$  for some integer $m\geq 2$. There exists an algebraic integer $\gamma$ for which 
$\alpha=\gamma\alpha', \beta=\gamma\beta'$ with $\gamma^m\in \Q$ and each $L_n(\alpha',\beta')\in \Q$, or  $m$ is even and there exists coprime integers $r,s$ such that $\alpha'=\sqrt{r}+\sqrt{s}, \beta'=\sqrt{r}-\sqrt{s}$ so  that
$L_n(\alpha',\beta')^2\in \Q$.
\end{lemma}

In other words, $L_n^m$ is an exponential LDS times a power of a Lucas sequence or a power of a Lehmer sequence.

\begin{proof}  Now $L_n^m$ satisfies a linear recurrence with characteristic roots
\[
\alpha^i \beta^{m-i} \text{ for } 0\leq i\leq m,
\]
each occurring once. Therefore 
\[
\prod_{i=0}^m (x-\alpha^i \beta^{m-i})\in \Z[x] .
\]
We determine some of the coefficients explicitly: The coefficient of $-x^m$ is 
\[
\sum_i \alpha^i \beta^{m-i}= L_{m+1} \in \Z;
\]
the coefficient of $(-1)^{m+1}x^0$ is 
\[
\prod_i \alpha^i \beta^{m-i}= ( \alpha\beta)^{m(m+1)/2} \in \Z;
\]
and the  coefficient of $(-1)^{m}x$ is 
\[
\prod_i \alpha^i \beta^{m-i} \cdot \sum_j\frac 1{\alpha^j \beta^{m-j}} = ( \alpha\beta)^{m(m-1)/2}L_{m+1} \in \Z.
\]
since $(\alpha\beta)^m/(\alpha^j \beta^{m-j})=\alpha^{m-j} \beta^j$. We deduce that 
\[
\frac{L_{m+1}  \cdot ( \alpha\beta)^{m(m+1)/2} }{( \alpha\beta)^{m(m-1)/2}L_{m+1} } = ( \alpha\beta)^m\in \Q,
\]
and so is in $\Z$ as it is an algebraic integer.




Let $K$ be the splitting field extension of $\Q(\alpha,\beta)$ and $G:=\text{Gal}(K/\Q)$.
From the assumptions  and the above discussion we have 
\[
(\alpha+\beta)^m=L_2^m, (\alpha\beta)^m \text{ and } (\alpha^2+\alpha\beta+\beta^2)^m=L_3^m \in \Z
\]
so that if $\sigma \in G$ then
\[
\sigma (L_2)=\zeta_1L_2, \sigma(\alpha\beta)=\zeta_2(\alpha\beta), \text{ and } 
\sigma(L_3)=\zeta_3L_3
\]
where $\zeta_1,\zeta_2,\zeta_3$ are all $m$th roots of unity.

We will prove that $\zeta_2=\zeta_1^2$ so that
$(x-\sigma(\alpha))(x-\sigma(\beta))=(x-\zeta_1 \alpha)(x-\zeta_1 \beta)$.
Therefore $\sigma(\alpha/\beta)=\alpha/\beta$ or $\beta/\alpha$. If this holds for every $\sigma \in G$ then
$\alpha/\beta\in \mathbb Q$ or $ \mathbb Q(\sqrt{d})$ for some integer $d$.  If $\alpha/\beta\in \mathbb Q$ 
then we obtain the first case in the result. If $\alpha/\beta\in \mathbb  Q(\sqrt{d})$ then by Hilbert's Theorem 90  we can write
\[
\frac \alpha \beta = \frac{ u+v\sqrt{d} } { u-v\sqrt{d} } 
\]
for coprime $u,v \in \mathbb Z$ and $d$ a squarefree integer.   Now if $g=(u^2,d)$ then $g$ is squarefree so we can write $r=u^2/g$ and $s=v^2d/g$ and then 
\[
\frac \alpha \beta = \frac{   \sqrt{r}+\sqrt{s}}{\sqrt{r}-\sqrt{s}}
\]
with $(r,s)=1$. Writing   $\alpha=\gamma(\sqrt{r}+\sqrt{s})$ and  $\beta=\gamma(\sqrt{r}-\sqrt{s})$  then 
$L_n=\gamma^{n-1}L_n(\sqrt{r}+\sqrt{s},\sqrt{r}-\sqrt{s})$. 
Therefore $L_2= \gamma \sqrt{r}$ so $\gamma^m\in \Q$ if $m$ is even, and $\gamma^m\in \sqrt{r} \Q$ if $m$ is odd. The second case  follows.

We now assume that there is some $\sigma \in G$ such that  $\zeta_2\ne \zeta_1^2$ (and eventually establish a contradiction).
Now $L_3=L_2^2-\alpha\beta$, and so applying $\sigma$ to both sides
\[
\zeta_3(L_2^2-\alpha\beta)=\zeta_3L_3=\sigma(L_3)=\zeta_1^2L_2^2-\zeta_2(\alpha\beta).
\]
Now  $\zeta_2\ne \zeta_3$ (else  $L_2=0$ as $\zeta_2\ne \zeta_1^2$) and $\zeta_1^2\ne \zeta_3$ (else $\alpha\beta=0$), 
and so
\[
\rho:= \frac{ L_2^2}{  \alpha\beta} =    \frac{1-\zeta_3^{-1}\zeta_2}{1-\zeta_3^{-1}\zeta_1^2}.
\]
Since $\arg(1-e^{i\theta})\in (-\frac \pi 2,\frac \pi 2]$ for all $\theta$ and $\rho\ne 1$ we see that $\rho\not\in \mathbb R$.


Now $L_4=L_2^3-2(\alpha\beta)L_2$ and suppose $ \sigma(L_4)=\zeta_4L_4$. Then
 \[
 \zeta_4(L_2^3-2(\alpha\beta)L_2)=\sigma(L_4)=\zeta_1^3L_2^3-2\zeta_1\zeta_2(\alpha\beta)L_2
 \]
 Now  $ \zeta_1^3\ne \zeta_1\zeta_2$ as $\zeta_2\ne \zeta_1^2$, but then
 $\zeta_4\ne \zeta_1^3$ (else $\alpha\beta=0$) and $\zeta_4\ne\zeta_1\zeta_2$ (else $L_2=0$).
 Dividing by $L_2$, $(1-\zeta_4^{-1}\zeta_1^3) L_2^2 = 2 \alpha\beta( 1-\zeta_4^{-1}\zeta_1\zeta_2)$ and inserting the above expression for
 $L_2^2 / \alpha\beta$ yields
 \[
2  (1-\zeta_4^{-1}\zeta_1\zeta_2)(1-\zeta_3^{-1}\zeta_1^2)= (1-\zeta_4^{-1}\zeta_1^3)(1-\zeta_3^{-1}\zeta_2)
\]
which can be rewritten as
 \begin{equation} \label{eq:Zetas}
(1+\zeta_2\zeta_3^{-1}) (1+\zeta_1^3\zeta_4^{-1})   =
2 (\zeta_1^2\zeta_3^{-1}  +\zeta_1\zeta_2\zeta_4^{-1} ) ,
\end{equation}
Letting $\zeta_2\zeta_3^{-1}=e^{2ia}, \zeta_1^3\zeta_4^{-1}=e^{2ib}, \zeta_1^{-1}\zeta_2\zeta_3\zeta_4^{-1} =e^{2ic}$,
this becomes
\[
\cos (a+b) +     \cos (a-b) = 2    \cos   a  \cdot  \cos  b   = 2 \cos c.
\]
Conway and   Jones \cite[Theorem 7]{CJ} of J. H. Conway and A. J. Jones showed that some subsum of these cosines must vanish:
\begin{itemize}

\item If $\cos (a+b) +     \cos (a-b)=0$ then $\cos   a=0$  or  $\cos  b =0$, and $\cos  c =0$, which means both sides of \eqref{eq:Zetas} equal $0$. Therefore
$\zeta_3= -\zeta_2  $ with $\zeta_1\zeta_4  = \zeta_2^2 $, or   $ \zeta_4=- \zeta_1^3$ with $\zeta_1^4 =\zeta_2\zeta_3$.

\item If one cosine on the left-hand side  is $0$ then the equation becomes $0\pm 1=2\cdot (\pm \frac 12)$ by \cite[Theorem 7]{CJ}. Therefore \eqref{eq:Zetas} equals $(1+i)(1-i)  = 2(-\omega-\omega^2)$ or $(1+i)(1+i)   = 2(-i\omega-i\omega^2)$.
This yields $\zeta_2 = -i \omega^2 \zeta_1^2, \zeta_3  = -\omega^2 \zeta_1^2$, 
or  $\zeta_2=-\omega^2 \zeta_1^2 , \zeta_3=i \omega^2 \zeta_1^2$. 

\item If $\cos(a+b)=\cos(a-b)=\cos c$  then $\{ 2a,2b\} = \{ 0,\pm 2c\} $ though $2a\ne 0$ as $\zeta_2\ne\zeta_3$.
Therefore $ \zeta_4=\zeta_1^3$ and $\zeta_2= -\zeta_1^{2}$ or  $\zeta_3= - \zeta_1^2$.
 
\end{itemize}

These cases yield:
\[
  \rho:=\frac{ L_2^2}{  \alpha\beta} =    \frac{2\theta}{1+\theta},   1+\theta,  i\omega-\omega,   \frac{1-i}{1+i \omega},  \frac{1+\phi}{1-\phi},  \text{ or }  \frac{1+\theta}{2}
\]  
respectively, where $\theta=\zeta_2/ \zeta_1^2$ and $\phi=\zeta_1^2/\zeta_3$.  Now $\rho$ and $L_3/(\alpha\beta)=\rho-1\in \Q^{1/m}$ and so
\[
1+\zeta_k= 1+\theta,   1+\theta,       1+i \omega^2, 1+i \omega,   1-\phi,  \text{ or }   1+\theta \in \Q^{1/m} 
\]
respectively for some integer $m$ (as $2,  (i-1)^2,\omega^3  \in \Q$), and some $k$, where
 $\zeta_k$ is a primitive $k$th root of unity. Therefore
 $\zeta_k$  is a root of $(1+x)^m-a$ for some $a\in \mathbb Q$, so $\phi_k(x)$ divides $(1+x)^m-a$.
Therefore if $\phi_k(\zeta)=0$ then
 $(a^2)^{1/m} = |1+\zeta|^2= 2 + \zeta+\overline{\zeta}$ and so 
 $ \zeta+\overline{\zeta}$ is fixed, meaning it can take at most two possible values, so $\phi(k)\leq 2$ and $k=1,2,3,4$ or $6$.
  We can therefore rule out the third and fourth cases   and assume $\phi$ and $\theta$ are $k$th roots of unity for $k=3,4$ or $6$,
  so that $\rho\in \mathbb Q(i)$ or $\Q(\omega)$.

Now $\ell_n^m\in \mathbb Q$ where $\ell_{2n}:=L_{2n}/ L_2(\alpha\beta)^{n-1}$ and  $\ell_{2n+1}:=L_{2n+1}/(\alpha\beta)^n$, with
$\ell_3=  \rho-1, \ell_4= \rho-2, \ell_6/\ell_3= \rho-3$ and $\ell_8/\ell_4= \rho^2-4\rho+2$.  Hence
$\rho, \rho-1, \rho-2, \rho-3\in \Q^{1/m}$.   If $ a+bi\in \Q^{1/m}$ then $\frac{a+bi}{a-bi}= \frac{(a+bi)^2}{ a^2+b^2}$ is a   root of unity, and so
  $a=0, b=0$ and $|a|=|b|$. Similarly if $ a+b\omega\in \Q^{1/m}$ then $a=0, b=0, |a|=|b|, a=2b$ or $b=2a$.
  Therefore $\rho=2+\omega$ or $1-\omega$ since $\rho\not\in \Q$.  This implies that  
  $\rho^2-4\rho+2=\omega^2-2=-3-\omega$ or $\omega-2$ which do not belong to any $\Q^{1/m}$.
  \end{proof}

\section{Consequences of Vandermonde-type results}  \label{sec: 1st consequences}

In section \ref{sec: Vandermonde} we applied linear algebra directly to the sum in \eqref{eq: Formula}. Now we group together those $\alpha_i^n$ that are congruent modulo a given power of a prime ideal, $P^e$: 
Let  $\alpha_1^n,\dots ,\alpha_h^n$ be a maximal set of distinct residues of the $\alpha_i^n \pmod {P^e}$ (re-labelling the $\alpha_i$ if necessary) and then let $S_n(P^e)$  be the partition $I_1\cup \dots \cup I_h$ of $\{ 1,\dots,k\}$ given by
 \[
 I_j:=\{ i\in  \{ 1,\dots,k\}:\ \alpha_i^n\equiv \alpha_j^n \pmod {P^e}\} \text{ for } 1\leq j\leq h=h(P^e).
 \]
We see that  $S_n(P)\leq S_ n(P^2)\leq \dots$.\footnote{Here we defined 
a partial ordering on the partitions $S:= I_1\cup \dots \cup I_h$ of $\{ 1,\dots,k\}$ by refining the $I_j$ into smaller subsets. That is,  $S\leq T$ if $T=\cup\, I_{j,s}$
where each $I_j=I_{j,1}\cup \cdots \cup I_{j,\ell_j}$. }  Moreover $P^e$ divides $ \mathcal G_n(S_ n(P^e))$ where
 \[
   \mathcal G_n(S) := \gcd_{\substack{ i\in I_j \\ 1\leq j\leq h}} ( (\alpha_i/\alpha_j)^n-1)^*.
 \]
 Let $\mathcal I^\dag$ be the set of those partitions $S= I_1\cup \dots \cup I_h$ of $\{ 1,\dots,k\}$ for which $g_{I_j}(t)=0$ for every  $I_j\in S$.
 
Given  $g(t)=\sum_j g_jt^j\in K[t]$, let $R_g=\prod_{ j: g_j\ne 0} g_j$ (with $R_0=1$) and 
 \[
 R = k! \cdot c_k \cdot  \prod _{I\subset \{ 1,\dots, k\}  }  R_{g_I}.
 \]
 
  \begin{lemma} \label{lem: 1stPartition} Let $(u_n)_{n\geq 0}$ be a linear recurrence sequence with $u_0=0, u_1=1$.
If  $P$ is a prime ideal with  $P\nmid Rn$ and  $P$ divides $(u_n,\dots,u_{(k-1)n})$ then $S_n(P)\in \mathcal I^\dag$.
\end{lemma}


\begin{proof}  
Let  $S_n(P)=I_1\cup \dots \cup I_h$.  
If $r\geq 0$ then, as $j\in I_j$,
\[
u_{rn}=\sum_{i=1}^m g_i(rn)  \alpha_i^{rn} = \sum_{j=1}^h \sum_{i\in I_j} g_i(rn)  \alpha_i^{rn}
\equiv  \sum_{j=1}^h \sum_{i\in I_j} g_i(rn)     \alpha_j^{rn} =  \sum_{j=1}^h  g_{I_j}(rn)     \alpha_j^{rn}  \pmod P.
\] 
Let $\kappa_j=\deg g_{I_j}+1$ and $\kappa(P)=\kappa=\sum_{j=1}^h \kappa_j \leq k$.
Taking $r=0,1,.\dots,\kappa-1$ we have a $\kappa$-by-$\kappa$ system of linear equations. We know the determinant of the corresponding matrix by  Corollary \ref{cor: Vand1}; its prime factors    divide
either $j!$ for some $j\le \kappa\leq k$, or some $\alpha_j$ and so $c_k$, or some difference $\alpha_i^{n} -\alpha_j^{n}$
with $1\leq i<j\leq h$, and so $P$ does not divide the determinant as $P\nmid R$.

Therefore $P$ divides all of the coefficients of all of the $g_{I_j}(nx)$, that is $P|R_{g_{I_j}}n$ if some $g_{I_j}\ne 0$, which contradicts $P\nmid Rn$. Hence each $g_{I_j}= 0$.
 \end{proof}

We need to revisit this argument but with prime powers, which is more complicated:

 \begin{prop}  \label{prop: Primepowers}
Let $(u_n)_{n\geq 0}$ be a linear recurrence sequence with $u_0=0, u_1=1$.
If  $P$ is a prime ideal with  $P\nmid Rn$ and   $P^v$ divides $(u_n,\dots,u_{(k-1)n})$ with $v \geq 1$ and $n>k$ then
there exists an integer $e\geq v/k!^2$ such that $S_n(P^e)\in \mathcal I^\dag$. 
\end{prop}

\begin{proof} We proceed as in the proof of  Lemma \ref{lem: 1stPartition} but this time with $S_n(P^v)$. If some $g_{I}(t)\ne 0$
then its non-zero coefficients are not divisible by $P$ (as $P\nmid R$) and so $P^v$ divides the determinant. Ruling out certain terms as before we deduce that  
\[
P^v \text{ divides } \prod_{1\leq i< j\leq h(P^v)} (\alpha_i^n-\alpha_j^n)^{\kappa_i\kappa_j} .
\]
Therefore, since  $2\sum_{i<j} \kappa_i\kappa_j\leq  (\sum_i \kappa_i)^2=\kappa^2$,
 there exists   $1\leq i<j\leq m$ with 
 \[
 P^{v_1} \text{ divides }  \alpha_i^n-\alpha_j^n \text{ for some } v_1\geq 2v/\kappa(P^v)^2.
 \]
 and so $h(P^{v_1})\leq h(P^v)-1$ and  $\kappa(P^{v_1})\leq \kappa(P^v)-1$.  If some $g_{I}(t)\ne 0$ for the partition $S_n(P^{v_1})$ then we repeat this same argument with $v_1$ in place of $v$, etc. Suppose this process terminates after $m$ steps (in that  $g_{I}(t)=0$ for every part $I$ of the partition $S_n(P^{v_m})$) then $m\leq k-1$ since  $\kappa(P^{v_m})\leq \kappa(P^v)-m\leq k-m$, and so
 \[
 e=v_m \geq  v \cdot \prod_{i=0}^{m-1} (2/\kappa(P^{v_i})^2 \geq v \cdot \prod_{i=0}^{m-1} (2/(k-i))^2 \geq v/k!^2. \qedhere
 \]
 \end{proof}
 
 Let $\rho_k=1/(2 k!^2\# \{ \text{Partitions of }  \{ 1,\dots,k\} \}$.

 \begin{corollary}  \label{cor: Whattodo}
Let $(u_n)_{n\geq 0}$ be a linear recurrence sequence with $u_0=0, u_1=1$, which is not divisible by a non-trivial exponential LDS. Fix $\kappa>0$ and $\epsilon>0$ significantly smaller, and $y$ large. Suppose that $n$ is chosen so that $ \sum_{p|n,   p>y} \frac 1p<\epsilon$ and
$(u_n,\dots,u_{(k-1)n})>e^{\kappa n}$. Then there exists a partition $S \in \mathcal I^\dag$ for which
\[
\mathcal G_n(S)  \geq e^{\rho_k \kappa n}
\]
 \end{corollary}

\begin{proof} Let  $P$ be a prime ideal and   $P^{v_P}\| (u_n,\dots,u_{(k-1)n})$,
so that $P^{v_P}$ divides $n^{O(1)} c_k^{O(n)}$ times a product of terms
 $\alpha_i^n-\alpha_j^n$  by Proposition \ref{pr: coolio} . Lemma \ref{PolyGenDivByp} (a) then implies that 
 \[
 \prod_{P|R, \ P\nmid c_k n} P^{v_P} \ll_R n^{O_R(1)}.
 \]
 If $P$ divides $c_k$ (so that $P$ divides some $\alpha_i$) but not $n$ then we can apply the same argument to
 $x_n=\sum_{i:\ P\nmid \alpha_i} g_i(n)\alpha_i^n$; this sum is non-empty since $u_n$ is not divisible by an exponential LDS and therefore the non-canceled out $\alpha_i$ cannot have a common prime ideal factor.
 Therefore $P^{v_P(x_n)} \ll_P  n^{O_P(1)}$ and since
 $x_n\equiv u_n \pmod {P^{n-O(1)}}$ this implies that $P^{v_P(u_n)} \ll_P  n^{O_P(1)}$.
 Finally if $P|n$ then by Proposition \ref{pr: coolio} applied to $x_n$ we have, for arbitrarily selected $y$,
 \[
  \prod_{P| n} P^{v_P} \ll_y n^{O_y(1)}\exp\bigg( O \bigg( n\, \sum_{p|n,   p>y} \frac 1p \bigg)\bigg) \ll_y n^{O_y(1)} e^{O(\epsilon n)}
 \]
 by   Lemma \ref{PolyGenDivByp}(b) for $p>y$, and Lemma \ref{PolyGenDivByp}(a) for $p\leq y$.
 
 We combine the above bounds to obtain
 \[
 w_n:=\prod_{\substack{P^{v_P}\| (u_n,\dots,u_{(k-1)n}) \\ P\nmid Rn}} P^{v_P} =
 \frac{(u_n,\dots,u_{(k-1)n}) } {  \prod_{P|R   n} P^{v_P} }  \gg e^{\frac \kappa 2 n}.
 \]
 
 Suppose that $P\nmid Rn$.  Since $P^e$ divides $ \mathcal G_n(S_ n(P^e))$ by definition, 
  Proposition \ref{prop: Primepowers} implies that   $P^{v_P}$ divides $ \mathcal G_n(S_ n(P^e))^K$ where $K=k!^2$
  for some $e\geq v_P/k!^2$ with $S_n(P^e)\in \mathcal I^\dag$, and so
\[
w_n \text{ divides } \bigg( \prod_{S \in \mathcal I^\dag }  \mathcal G_n(S) \bigg)^K.
\]
Combining the last two inequalities for $w_n$ we deduce  the result.
  \end{proof}

 \section{Corvaja-Zannier}

The key to the next part of our work is  the  beautiful Corollary 1 of Corvaja and Zannier \cite{CZ}, inspired by their earlier work with Bugeaud \cite{BCZ} (and see \cite{HL}), resting on the Subspace Theorem of Schmidt and Schlickewei, which can be written as follows:
If the numerators and denominators of algebraic numbers $u$ and $v$ have prime ideal factors only from some given finite set $S$ (so that $u$ and $v$ are \emph{$S$-units})  but are multiplicatively independent then
\begin{equation} \label{eq: CZ-ineq}
|N((u-1,v-1))| \ll_{S,\epsilon} \max\{ H(u), H(v)\}^\epsilon,
\end{equation}
where the \emph{height} $H(u):=\prod_\nu \max\{ 1,|u|_\nu\}$.\footnote{Here $\nu$ runs over the set of places of $K$, and 
$|\cdot |_\nu$ is the corresponding absolute value, normalized so that the product formula holds in $K$. The precise formulation is not so important as the difference between any sensible variants are made irrelevant by the $\epsilon$-power.}
This implies the following:

\begin{prop} \label{prop: CZ}
Let $u,v\in K$ be multiplicatively independent algebraic numbers.    For any given $\epsilon>0$, if $n$ is sufficiently large then
\[
|N((u^n-1,v^n-1))| < e^{\epsilon n}.
\]
Here ``sufficiently large'' depends on $u,v$ and $\epsilon$.
\end{prop}

\subsection{Multiplicative dependence}  \label{sec: MultDep}
 
\begin{lemma} \label{lem: premultdep} 
If $u_1,\dots,u_k$ are multiplicatively dependent algebraic numbers in $K$ then
there  exists an algebraic number $w\in K$,  and   roots of unity $\zeta_1=1,\dots,\zeta_k$ such that 
\[
u_i = \zeta_i w^{r_i} \text{ for each  }  i, 1\leq i\leq k \text{ where }  \gcd_i r_i=1.
\]
\end{lemma}
\begin{proof} By induction on $k$. For $k=2$ the numbers $u=u_1$ and $v=u_2$ are multiplicatively dependent so there exist non-zero integers $R,S$ with
 $u^S=v^R$. We select $R,S$ to be the minimal integers with this property.
 Let $g=(R,S)$ then $R=rg, S=sg$ so that $(r,s)=1$, and 
 select  integers $m,\ell$ so that $rm+s\ell=1$. Let $w=u^mv^\ell$ so that 
 \[
w^R =u^{mR}(v^{R})^\ell=u^{mR}(u^S)^\ell=u^g \text{ and similarly } w^S=v^g.
 \]
 Taking $g$th roots we deduce that  $u = \zeta_u w^r$ and $v = \zeta_v w^s$ where $\zeta_u^g=\zeta_v^g=1$.
 We can multiply $w$ by $\zeta_{ur}$  where $\zeta_{ur}^r= \zeta_u$ so that $u=w^r$.
 
Suppose the result is true for $k$, so  there exists $v\in K$ with $u_i=\zeta_i v^{R_i}$ for $1\leq i\leq k$ with $\gcd_i R_i=1$. Now $u=u_{k+1}$ and $v$ are multiplicatively dependent, by hypothesis, so the result for $k=2$ yields there exists $w\in K$ with 
 $u=\zeta_u w^r,\ v=\zeta_v w^s$ where $(r,s)=1$. Therefore
 $u_i=\zeta_i (\zeta_v w^s)^{R_i} = \zeta_i' w^{r_i}$ where $r_i=R_is$ for $1\leq i\leq k$, and so
 \[
 \gcd_{1\leq i\leq k+1}  r_i = \gcd( r, \gcd_{1\leq i\leq k}  R_is) = \gcd( r, s\gcd_{1\leq i\leq k}  R_i) =  \gcd( r, s)=1.\qedhere
 \]
 \end{proof}

\begin{corollary} \label{Cor: multfep} If $u_1,\dots,u_k$ are as in Lemma \ref{lem: premultdep}  with $N(\gcd_{1\leq i\leq k} ( u_i^n-1)^*)> 2^d$ then each $\zeta_i^n=1$ and 
\[
  \gcd_{1\leq i\leq k} ( u_i^n-1)^*  = (w^n-1)^*.
\]
\end{corollary}

 \begin{proof} By induction on $k$. For $k=2$, we define $R=rg, S=sg$ as in the proof of  Lemma \ref{lem: premultdep}  so that  $\zeta_u^g=\zeta_v^g=1$.
 Let $G:=(u^n-1,v^n-1)^*=(w^{rn}-\alpha, w^{sn}-\beta)^*$ where $\alpha=\zeta_u^{-n}, \beta=\zeta_v^{-n}$. Therefore $\alpha^s \equiv w^{rsn} \equiv \beta^r \pmod G$. If $\alpha^s\ne \beta^r$ then   $G^\sigma$ divides $(\alpha^s)^\sigma-(\beta^r)^\sigma\ne 0$ for each $\sigma\in \text{Gal}(K/\mathbb Q)$, so that 
 $|N(G)|\leq \prod_\sigma |(\alpha^s)^\sigma-(\beta^r)^\sigma|\leq 2^d$, contradicting the hypothesis. Therefore
  $\alpha^s=\beta^r$ and so 
   $u^{sn} = \alpha^{-s} w^{rsn}= \beta^{-r} w^{rsn} = v^{rn}$. This implies  $g|n$ by the minimality of $R$ and $S$, and so
 $\alpha=\beta=1$. Therefore $G =(w^{rn}-1, w^{sn}-1)^*=(w^n-1)^*$. 
 
 Also $w$ cannot be a root of unity else $N(w^n-1)\leq 2^d$ contradicting   hypothesis.
 
 Now suppose the result is true for $k$ (but replace $w$ by $v$ in the conclusion for convenience) and then let $u=u_{k+1}$, so that
 $u$ and $v$ are multiplicatively dependent and   
 \[
 \gcd_{1\leq i\leq k+1} ( u_i^n-1)^* = \gcd( u^n-1,  \gcd_{1\leq i\leq k} ( u_i^n-1)^*) =
 \gcd( u^n-1, v^n-1)^*
 \]
 and so we can (again) apply the result with $k=2$ to complete our proof.  
  \end{proof}

\begin{corollary} \label{Cor: multfep2} Suppose that $u_1,\dots,u_k$ are algebraic numbers which are not roots of unity or $0$. Fix $\epsilon >0$. If $n$ is sufficiently large then either $u_1^n,\dots,u_k^n$ are all integer powers of some $w^n$
with $gcd_{1\leq i\leq k} ( u_i^n-1)^*  = (w^n-1)^*$ or $gcd_{1\leq i\leq k} ( u_i^n-1)^* <e^{\epsilon n}$.
\end{corollary}

\subsection{Proof of Theorem \ref{thm: main4}}
We may assume that $u_n$ is not divisible by a non-trivial exponential LDS, else we are done.

  Let $y>2\epsilon^{-2}$ and then $s_y(n)=\sum_{p|n, p>y} 1/p$. If $p|n$ we write $n=pm$ so  that
 \begin{align*}
 \#\{ n\leq x: n\equiv a \text{ mod }M: s_y(n)\geq \epsilon\} & \leq \epsilon^{-1} \sum_{\substack{ n\leq x\\ n\equiv a \pmod M}}  \sum_{p|n, p>y} \frac 1p \\
 & \leq \epsilon^{-1}  \sum_{y<p\leq x}  \frac 1p \sum_{\substack{  pm\leq x \\  pm\equiv a \pmod M}}  1\\ 
 & \leq \epsilon^{-1}  \sum_{y<p\leq x}\bigg( \frac x{p^2M}+ \frac 1p \bigg) <\frac \epsilon 2 \frac xM
 \end{align*}
if $x\gg M\log M$. Therefore if $x$ is sufficiently large there are $\geq \epsilon \frac x{2M}$ integers $n\leq x$ for which $n\equiv a \pmod M$, with $s_y(n)< \epsilon$ and $\gcd(u_n,\dots,u_{(k-1)n})>e^{\epsilon n}$.

Taking $\kappa=\epsilon$ in Corollary \ref{cor: Whattodo} we deduce that for each such $n$ there exists   a partition $S \in \mathcal I^\dag$ for which $\mathcal G_n(S)  \geq e^{\delta n}$ where $\delta=\rho_k \epsilon$.

We claim that  $S \in \mathcal I^*$ for, if not, 
 there exist ratios $\alpha_i/\alpha_j$ and $\alpha_{i'}/\alpha_{j'}$
with $i\in I_j$ and $i'\in I_{j'}$ that are   multiplicatively independent. Therefore
$\mathcal G_n(S)$ divides $N( (\alpha_i/\alpha_j)^n-1,(\alpha_{i'}/\alpha_{j'})^n-1)$ which is $<e^{\delta n}$ by Proposition
\ref{prop: CZ}, a contradiction.

By Corollary \ref{Cor: multfep2} we deduce that there exists an algebraic integer $\gamma$ such that 
each $(\alpha_i/\alpha_j)^n$ with $i\in I_j$ equals $(\gamma^n)^{e_i}$ for some integer $e_i$. In particular we have
\[
u_{rn} = \sum_{j=1}^h \alpha_j^{rn} \sum_{i\in I_j} g_i(rn) (\gamma^n)^{re_i} \equiv \sum_{j=1}^h \alpha_j^{rn} g_{I_j}(rn)= 0 \pmod {\gamma^n-1}
\]
for all $r\geq 1$. Converting this to polynomial form (as in section \ref{sec: Period}) we deduce that $\gamma^n-1$ is a factor of $u_n$ for $n\equiv a \pmod q$.  

Now $u_n$ is an integer and so $(\gamma^n-1)^\sigma$ is a factor of $u_n$ for $n\equiv a \pmod q$ for every $\sigma\in \text{Gal}(K/\mathbb Q)$, and these factors combine to yield a polynomial generated LDS $(v_n)_{n\geq 0}$ which divides
$u_n$ whenever $n\equiv a \pmod M$. The result follows by definition.\hfill  \qed

 \subsection{GCDs of LDSs} \label{sec: GCDsLDSs}
 
 We   use the Corvaja-Zannier result to show that $\gcd(2^n-1,F_n)$ is not a LDS (though it is easily seen to be a division sequence).
 
 \begin{theorem} \label{thm: NotALinRec}
 $u_n=(2^n-1,F_n)$ does not satisfy a linear recurrence.
 \end{theorem}
 
 \begin{proof}   Assume that $u_n$ satisfies a linear recurrence and so is a LDS (as it is clearly a division sequence).
 Since  $2$ and $\frac{1+\sqrt{5}}2\big/ \frac{1-\sqrt{5}}2$ are  multiplicatively independent we see that 
 $u_n=e^{o(n)}$ by Proposition \ref{prop: CZ}.  Therefore each $|\alpha_i|\leq 1$, and as $u_n\in \mathbb Z$ we get all of the conjugates, and therefore they are roots of unity by Kronecker's Theorem.  This implies that there exists an integer $M$, such that if
 $(n,M)=d$ then $u_n=u_d (n/d)^{e_d}$ by Theorem \ref{thm: main}
 
 Let $p$ be a  prime $\equiv \pm 1 \pmod 5$ and $>u_M$. Now $p$ divides $(F_{p-1}, 2^{p-1}-1)=u_{p-1}$ which divides
 $u_d (p-1)^{e_d}$, and so $p$ divides $u_d$ which divides $u_M$, a contradiction as   $p>u_M$.
 \end{proof}
 
 This argument is easily generalized to many analogous situations.

 \section{Low order LDSs with $u_0=0$} \label{sec: Loworder}
 
 In this section we return to the quest of other researchers in this area to investigate the LDS's of low order.
 
 \subsection{Reducing the work} \label{sec: Reducing}
 
 \begin{lemma} \label{lem: Simplifying}
 If $(u_n)_{n\geq 0}$ is a LDS of order $k$ and period $M$ in the integers with $u_M\ne 0$ then
 $(v_n)_{n\geq 0}$ with $v_n=u_{Mn}/u_M$ is a LDS of order $\leq k$ and period $1$ (and in fact of period $k$ unless $u_n$ is degenerate).
 \end{lemma}
 
 \begin{proof}
$(v_n)_{n\geq 0}$  has period $1$ by the definition of ``period''. It has order $\leq k$ since if $\rho_1,\dots,\rho_\ell$ are the roots of an irreducible factor of the characteristic polynomial of $(u_n)_{n\geq 0}$, then $\rho_1^M,\dots,\rho_\ell^M$ are the roots of a polynomial in $\mathbb Z[x]$ by elementary Galois theory. The order of $(v_n)_{n\geq 0}$ is less than that of  $(u_n)_{n\geq 0}$ if and only if there are distinct roots of the characteristic polynomial of $(u_n)_{n\geq 0}$ that differ by an $M$th root of unity, that is, if $(u_n)_{n\geq 0}$ is degenerate. (This was all observed by Ward \cite{Ward1} in 1937.)
 \end{proof}

 Lemma \ref{lem: Simplifying} allows us to work with  $(v_n)_{n\geq 0}$, a LDS of period $1$. By Theorem \ref{thm: main}   $v_n=n^e \lambda^{n-1} w_n$ where $w_n$ is a polynomially generated LDS,  and integers $e\geq 0, \lambda\ne 0$. 
(We obtained this same decomposition for $u_n$ if it is non-degenerate in Corollary \ref{thm: main6}.)
 Then $\text{order}(v_n)=(e+1)\text{order}(w_n)$. Therefore we have reduced determining all  LDS's of order $\leq K$,  to finding all period 1, polynomially generated LDS's of order $\leq K$ (although finding all LDS's of order $\leq K$ from this information can be challenging). Moreover, for polynomially generated LDS's, each $g_i(t)=g_i\ne 0$ in \eqref{eq: Formula} is a constant, $\sum_{i=1}^k g_i=u_0=0$ (so that $k\geq 2$), and we may assume that $w_1=1$.

 In  section \ref{sec: FindFactors} we  determined $w_n$ from $\mathcal I^*$. To do so we must 
 \begin{itemize}
 \item Identify all partitions of $\mathcal I=I_1\cup \dots \cup \{ 1,\dots, k\}$ with each $\sum_{i\in I_j} g_i=0$.  
 \item Assume  there exists  $\eta_{\mathcal I}\in K$ such that if $i\in I_j$ then
 $\alpha_i = \alpha_j \eta_{\mathcal I}^{r_{\mathcal I,i}}$ for some $r_{\mathcal I,i}\in \mathbb Z$ with $\gcd_i r_{\mathcal I,i}=1$.
 \end{itemize}
 
 If $(u_n)_{n\geq 0}$ is a LDS of order $k$ and period $M$, then the torsion associated with the characteristic roots must generate the $M$th roots of unity. Moreover if any given primitive $r$th root of unity is generated then we can obtain the others by conjugation. The smallest set of roots of unity which do this involves the primitive $p^e$th roots of unity for all $p^e>2$ for all prime powers $p^e\|M$.  Therefore
 \begin{equation} \label{eq: orderandperiod}
 k \geq \sum_{p^e\| M, \\  p^e>2}  \phi(p^e).
 \end{equation}

 \subsection{Polynomially generated LDS's with $k=2$, period 1 and $u_0=0$} We have $w_n=a\alpha^n+b\beta^n$ where $a+b=0$ and $w_1=1$ so
 $w_n=L_n= \frac{\alpha^n-\beta^n}{\alpha-\beta}$, a Lucas sequence.

\subsection{Polynomially generated LDS's with $k=3$, period 1 and $u_0=0$} \label{sec: k=3,p=1}
Here we have  $w_n=a \alpha^n+b\beta^n + c\gamma^n$ where $a+b+c=0$ and  
 no proper subsum can equal $0$ (as $a,b,c$ are all non-zero). Hence $\mathcal I^*=\{ \{ 1,2,3\}\}$.
 Therefore $\beta/\alpha=\eta^r$ and $\gamma/\alpha=\eta^s$  with $(r,s)=1$ and so $( \beta^n- \alpha^n, \gamma^n- \alpha^n)=(\eta^n-1)^*$.  This is the only possible such divisor of $w_n$ so we must have
 $w_n=\kappa  ( (\eta^n-1)^*)^e$ for some $\kappa \in K$ and integer $e\geq 1$. Lemma \ref{lem: order} implies that $e\leq 2$ (and, in general, $e\leq k-1$) and so $e=2$ for $k$ to be $3$. Therefore
$w_n=L_n^2$, for some Lucas sequence $L_n$.  We require only that $L_n^2\in \mathbb Z$, but not necessarily
$L_n$, so $w_n$ might be a Lehmer sequence, as discussed in section \ref{sec: conjugates}.
 
\subsection{The number of monomials in a product}
 
\begin{lemma}\label{lem: order} If none of $\gamma_1,\dots,\gamma_t$ are  roots of unity or $0$, and
\[
\prod_{i=1}^t (\gamma_i^n-1)^{h_i} \text{ divides }  \sum_{j=1}^r \delta_j \alpha_j^n \text{ for all } n\geq 1
\]
where the $\alpha_j$ are distinct,  and the $\delta_j\ne 0$ then $r\geq H+1$ where $H=\sum_{i=1}^t h_i$.
\end{lemma}

This result is ``best possible'' since for $(\gamma-1)^h$ we have $r=h+1=H+1$.

\begin{proof} By the Hadamard quotient theorem there exists a linear recurrence sequence $v_n$ which is the quotient of the right-hand side divided by the left. The $\alpha_i, \gamma_j$ and characteristic roots for $v_n$ can all be written in terms of a multiplicative basis $\beta_1,\dots,\beta_b$ and torsion. If the torsion are all $M$th roots of unity then each
\[
\gamma_j^M=\beta_1^{c_{j,1}}\cdots \beta_b^{c_{j,b}} \text{ and } \alpha_j^M = \beta_1^{a_{j,1}}\cdots \beta_b^{a_{j,b}}
\]
 for integers $c_{j,i}$ and $a_{j,i}$, where the $c_j$-exponent vectors are distinct, and the $a_j$-exponent vectors are distinct.   
 As in section \ref{sec: Period} we can replace each $\beta_i^n$ by a variable $y_i$ and then the corresponding  
 left-hand side divides the right-hand side. 
 
 Next we replace each $y_i$ by $x^{N^{i-1}}$ where   integer $N>2\max_{i,j}  |a_{j,i}|, |c_{j,i}|$, so that 
  each $\alpha_j^n$ is replaced by $x^{a_j}$ where $a_j=\sum_{i=1}^b  a_{j,i}N^{i-1}$ for $1\leq j\leq r$, 
  and each $\gamma_j^n$ by $x^{c_j}$ where $c_j=\sum_{i=1}^b  c_{j,i}N^{i-1}$. 
  The $a_j$ are distinct for if $a_j=a_k$ then $a_{j,1}\equiv a_{k,1} \pmod N$ and therefore $a_{j,1}=a_{k,1} $. Looking at the next coefficient and repeating the argument we eventually deduce that $\alpha_j=\alpha_k$, a contradiction. 
From these substitutions we deduce that
\[
\prod_{i=1}^t (x^{c_i}-1)^{h_i} \text{ divides }  \sum_{j=1}^r \delta_j x^{a_j}.
\]
The left-hand side is divisible by $(x-1)^H$.
Applying $(x\frac d{dx})^\ell$ to the right-hand side for $0\leq \ell \leq H-1$ at $x=1$, we obtain
\[
\sum_{j=1}^r \delta_j a_j^\ell =0 \text{ for } 0\leq \ell \leq H-1.
\]
If $r\leq H$ we can construct the Vandermonde matrix, which has non-zero determinant as the $a_j$ are distinct, so that the $\delta_j$ are all $0$, contradicting the hypothesis. Therefore $r\geq H+1$ as claimed.
\end{proof}

Lamzouri points out an alternative proof of this last step:  Here $H$ is no more than the number of positive real roots of our polynomials which is bounded by the number of sign changes amongst the coefficients, by Descartes'  ``rule-of-signs'', which is $\leq r-1$.

\subsection{LDS's with $k=3, u_0=0$ and $u_1=1$} \label{sec: LDSk=3}
Lemma \ref{lem: Simplifying} implies that $u_{Mn}$ has period 1 and order $\leq 3$.

$\bullet$ \ If $u_{Mn}$ has order $3$ then $u_n$ is non-degenerate and we saw that $u_{Mn}=u_ML_{Mn}^2$  by section \ref{sec: k=3,p=1}, which we write as $u_{Mn}=c( (\alpha^{2})^{Mn} -2 (\alpha\beta)^{Mn} + (\beta^{2})^{Mn} )$
for some constant $c>0$ (where $\alpha$ and $\beta$ are only determined up to an $M$th root of unity).  Adjusting $\alpha$ and $\beta$ by a suitable $2M$th root of unity, we deduce that $u_{n}=c( (\alpha^{2})^{n} -2 (\zeta \alpha\beta)^{n} + (\beta^{2})^{n} )$ for some $M$th root of unity $\zeta$. Moreover $\alpha/\beta$ is not a root of unity, so any automorphism of $\mathbb Q(\zeta)$ fixes $\alpha$ and $\beta$, and so 
mapping $\zeta\to \zeta^{-1}$ we see that  $\zeta \alpha\beta= \zeta^{-1} \alpha\beta$, that is $\zeta=\pm 1$.
If $\zeta=-1$ we replace $\beta$ by $-\beta$ and therefore we have $u_n=L_n^2$. Again this is either the square of a Lucas sequence, or it is a Lehmer sequence.

$\bullet$ \ If $u_{Mn}$ has order $1$ and $u_n=\sum_i a_i \alpha_i^n$ then the $\alpha_i^M$ are all equal, and so
we can write $\alpha_i=\zeta_i a$ where each $\zeta_i$ is an $M$th root of unity. Therefore
$u_n= a^{n-1} v_n$ where $v_n$ is periodic of order 3. The only possible sets of three roots of unity involving only complete set of conjugates 
are $\{ \pm 1,\omega,\omega^2\}, \{ \pm 1,-\omega,-\omega^2\}$ and $\{ \pm 1,i,-i\}$, and the map $a\to -a$ means we can restrict attention to the sets with $1$ rather than $-1$.  In each case
$v_n$ starts $0,1,b\in \mathbb Z$ and then for the characteristic polynomial

$(x-a)(x^2+ax+a^2)$ we need $b$ divides $a^2$; 

$(x-a)(x^2+a^2)$\hskip .41in  we need $b-1$ divides $a^3$; 

$(x-a)(x^2-ax+a^2)$ we need $b$ divides $3a^2$, $2b-3$ divides $3a^4$ and $b-2$ divides $a^5$. 

\noindent For the characteristic polynomial to be in $\mathbb Z[x]$ we must have $A:=a^3\in \mathbb Z$ in the first case, and 
$a\in \mathbb Z$ in the other two.

$\bullet$ \ If $u_{Mn}$ has order $2$ then it is a Lucas sequence by section \ref{sec: k=3,p=1}, and so $u_n=a\alpha^n + b(\zeta\alpha)^n-(a+b) \beta^n$ where $\zeta^M=1$. Taking $n\equiv 1 \pmod M$ then $u_n|u_{2n}$ so
\[
u_n \text{ divides } (a+b)u_{2n}-u_n( (a+\zeta b) \alpha^n+(a+b)\beta^n) = ab (\zeta-1)^2 \alpha^{2n}
\]
and similarly $u_n$ divides $(a+b) ab (\zeta-1)^2 \beta^{2n}$ and so $u_n$ divides $(a+b) ab (\zeta-1)^2$ as $(\alpha,\beta)=1$. This is not possible as $u_n$ is not bounded (else it would be periodic).

The discussion in section \ref{sec: Reducing} reveals that the only other possibility is $n^2a^{n-1}$.

We have therefore proved that \emph{the only LDS's of order 3 with $u_0=0$ and $u_1=1$ are 
$n^2a^{n-1}, L_n^2a^{n-1}$ or $w_na^{n-1}$ (where $L_n$ is a Lucas sequence in $\mathbb Z$, and $w_n$ is a Lehmer sequence) and the three degenerate families of examples just above.}
We observe that the characteristic polynomials are $x^3-A$,
\[
(x-a)^3,\ (x-a\alpha^2)(x-a\beta^2)\cdot (x-a\alpha\beta), (x-a)(x^2+a^2) \text{ and } (x-a)(x^2-ax+a^2).
\]
This re-establishes the classification given by Ward \cite{Ward3}. All of the  characteristic polynomials other than $x^3-A$ are reducible, as remarked upon by Hall \cite{MaHa}.

\subsection{Polynomially generated LDS's with $k=4$, period 1 and $u_0=0$} Here we have  $w_n=a \alpha^n+b\beta^n + c\gamma^n+d\delta^n$ where $a+b+c+d=0$ and $a,b,c,d\ne 0$, so
 any proper subsums equalling $0$ must be a pair  like $a+b=0$ which implies $c+d=0$. If there is another pair then  $u_n=a(\alpha^n+\delta^n-\beta^n -\gamma^n)$.
 
 If $\{1,2,3,4\}\in \mathcal I^*$ then $\beta/\alpha=\eta^r, \gamma/\alpha=\eta^s, \delta/\alpha=\eta^t$ with $(r,s,t)=1$.
 If this is the only element of $\mathcal I^*$ then $w_n$ is essentially  $(\eta-1)^e$ for some integer $e\geq 1$ where 
 $(x-1)^e = dx^t+cx^s+bx^r+a$. To obtain exactly four monomials we must have $e=3$ and therefore
 $w_n=L_n^3$, for some Lucas sequence $L_n$. Here we need $L_n^3\in \mathbb Z$ so we could have $L_n\in \mathbb Z$, but also perhaps some analogy to the Lehmer sequence, but now for cubes.
 
 If we also have $\{1,2\}\cup \{3,4\} \in \mathcal I^*$  with $(r,s-t)=q$ then $\eta^q-1$ divides $w_n$. If these are the only elements of $\mathcal I^*$ then we have 
 $w_n =\kappa\lcm[ (t^q-1)^e, (t-1)^f]_{t=\eta}= \kappa   (t^q-1)^e(t-1)^{\max\{ 0,f-e\} }\big|_{t=\eta}$. Then 
 $\max\{ e,f\} \leq 3$ by Lemma \ref{lem: order}, and so we either get $w_n=L_{qn}L_n$ for some Lucas sequence $L_n$,
 or $f=3, e=1$ or $2$ to obtain $(t^q-1)^2(t-1)$ or $(t^q-1)(t-1)^2$. These have 6 terms so two must cancel, which only happens for $(t^2-1)(t-1)^2 = t^4-2t^3 +2t-1$. Therefore $w_n=L_{2n}L_n^2$ for some Lucas sequence $L_n$.
 This could be in $\mathbb Z$ but we only need to have $c L_{2n}L_n^2\in \mathbb Z$ for some constant $c$.
 Or if $L_n$ is as in the Lehmer sequence then $L_{2n}\in \sqrt{r}\mathbb Z$ and $L_n^2\in \mathbb Z$, and so 
 we could have $\sqrt{r}L_{2n}L_n^2\in \mathbb Z$.
 
 If $\mathcal I^*=\{ \{1,2\}\cup \{3,4\}, \{1,3\}\cup \{2,4\} \}$ then 
 $\alpha/\beta, \gamma/\delta$ are powers of $\eta$, and 
 $\alpha/\gamma, \beta/\delta$ are powers of $\theta$, where $\eta$ and are   multiplicatively independent.
 Therefore $(\eta-1)^e$ and $(\theta-1)^f$ divide $w_n$, with $e,f\geq 1$. Since they are independent, the $(e+1)(f+1)$ terms in the product are all distinct. If this is $\leq 4$ then $e=f=1$.
 Therefore $w_n$ is the product of two distinct Lucas sequences, or is a Lucas sequence of order two.
   
 Finally if $\mathcal I^*=\{ \{1,2\}\cup \{3,4\}, \{1,3\}\cup \{2,4\} , \{1,2,3,4\} \}$ then
 $\alpha/\beta, \gamma/\delta$ are powers of $\eta^r$, and 
 $\alpha/\gamma, \beta/\delta$ are powers of $\eta^s$ with $(r,s)=1$, so that for some $k\in \mathbb Z$ we have
 \[
 w_n=\kappa \alpha^n( 1-t^{br+krs}-t^{ds+krs}+t^{br+ds+krs})\big|_{t=\eta} \text{ where } (b,kr)=(d,ks)=(r,s)=1.
 \]
 This is not divisible by $(1-t^r)^2$ else, by differentiating $r|ds$ and $krs=0$ and so $r|d$ and $k=0$, so that $b, d=\pm 1$ and so $r=1$, which is impossible for an element of $\mathcal I^*$. Similarly it is not divisible by  $(1-t^s)^2$.
  If $(1-t)^2$ divides then $k=0$ and so $b=d=1$, and then $w_n=\kappa \alpha^n( 1-\eta^r)(1-\eta^s)$ so that $w_n=L_{rn}L_{sn}$ for a given Lucas sequence $L_n$.
 Finally $[ 1-t^r, 1-t^s]_{t=\eta}$ might be the $\eta$-part of $w_n$. If $s<r$ this equals $(1+t+\dots +t^{s-1})(1-t^r)$ so that $k=2s$. Hence 
 $s=2$ with $r$ odd and we obtain $w_n=[L_{2n}, L_{rn}]=(\alpha^n+\beta^n)L_{rn}$. This generalizes Bala's example (the case $r=3$).

 Therefore the possibilities for $k=4$ are $w_n=L_n^3, L_{n}^2L_{2n}$ where $L_n$ is a Lucas sequence in $\mathbb Z$ or $\sqrt{r} L_{n}^2L_{2n}$ where $L_n$ is the order two Lucas sequence that yields a Lehmer sequence, or  the product of two different Lucas sequences, or  a Lucas sequence of order two (like the Guy-Williams examples), or $\lcm [L_{2n}, L_{rn}]$ (like in Bala's example).
 One can find a different perspective to this classification in \cite{ABCM}. These examples were also identified by Oosterhaut \cite{Oost}.
 
 We leave the task of determining a full classification of all LDS's of order 4 (not just those of period 1) to the reader --- we already saw that it was complicated (but feasible) to deduce such classifications, even for order 3, in section \ref{sec: LDSk=3}.
 
\subsection{Polynomially generated LDS's with $k>4$, period 1 and $u_0=0$}

 One can continue in much the same way for larger $k$ and the examples mostly generalize what we have here. For example, for $k=6$, there are products like $(x-1)^2(y-1)$ and 
\[
(x-1)^5, (x-1)^3(x^2-1), (x-1)^4(x^2-1), (x-1)^2(x^2-1)(x^3-1),  (x-1)^2(x^2-1)(x^3-1)^2.
\]
There is also $\lcm [x^3-1,x^s-1]$ where $3\nmid s$, and $\lcm [x^2-1,(x^r-1)^2]$ where $2\nmid r$.

These lead to $L_n^2M_n$, where $M_n\ne L_n$ or $L_{2n}$, the product of two Lucas sequences, or
\[
L_n^5, L_n^3L_{2n}, L_n^4L_{2n}, L_n^2L_{2n}L_{3n}, L_n^2L_{2n}L_{3n}^2, [L_{2n}, L_{rn}^2], [L_{3n}, L_{sn}].
\]
A new type of product  appears at $k=6$ from the cancelation of the $xy$-term when forming the identity
 \[
 (1-x)(1-y)(1-xy) =  1-x-y+x^2+xy^2-x^2y^2.
 \]
 This  gives the product of three inter-related Lucas sequences:
 \[
 \frac{\alpha^n-\beta^n}{\alpha-\beta} \cdot \frac{\beta^n-\gamma^n }{\beta-\gamma} \cdot \frac{\gamma^n-\alpha^n}{\gamma-\alpha}
 \]
whose product is an integer, so $\alpha,\beta,\gamma$ could be integers, or the three roots of an irreducible cubic, or
$\alpha=uv, \beta=\overline{u}v, \gamma=u \overline{v}$ where $u,v$ are quadratics in the same field, and we end up dividing through by $(uv)^{n-1}$. We can rewrite this
\[
\frac{ \zeta_1^n+\zeta_2^n+\zeta_3^n-\xi_1^n-\xi_2^n-\xi_3^n}{\zeta_1 +\zeta_2 +\zeta_3 -\xi_1 -\xi_2 -\xi_3}
\]
where $\zeta_1=\alpha\beta^2, \xi_1=\alpha \gamma^2, \zeta_2= \beta\gamma^2, \xi_2=\gamma\beta^2, \zeta_3=\gamma\alpha^2, \xi_3=\alpha\gamma^2 $ so that each $\zeta_i\xi_i=(\alpha\beta\gamma)^2\in \mathbb Z$
and $\zeta_1\zeta_2\zeta_3=(\alpha\beta\gamma)^3$. (See \cite{GRW, GW2} for similar constructions.)

One can also use this $xy$-identity in different ways, for example obtaining $L_{an}L_{bn}L_{(a+b)n}$ for any integers $a>b\geq 1$, or 
 getting the two LDS's
\begin{align*}
\frac 1{10} (2^n-1)(3^n-1)(6^n-1)& =\frac 1{10} (36^n -18^n-12^n+3^n+2^n-1) \\
\text{ and } \frac 12(2^n-1)(3^n-1)(3^n-2^n)&=\frac 12(18^n-12^n-9^n+4^n+3^n-2^n).
\end{align*}

  .

 \subsection{Periodic LDS's of low order}  \label{sec: periodiclow}
 Since $(u_n)_{n\geq 0}\in \mathbb Z$, its characteristic roots contain whole sets of conjugate algebraic numbers; if $u_n$ is periodic these must be sets of primitive $m$th roots of unity for various $m$, and there can be no repetitions. Therefore if  the characteristic polynomial is $\prod_{m\in \mathcal M} \phi_{m}(x)$ for some set $\mathcal M$ of distinct integers then $\lcm_{m\in \mathcal M} m=M$ and order$(u_n)=k=\sum_{m\in \mathcal M} \phi(m)$. Running through the options
 we have, when $u_0=0$, and we multiply through by $(-1)^{n-1}$ to guarantee $u_2>0$:
    
\[ \begin{array}{ccll}
k & M & \mathcal M &  u_0,\dots,u_{M-1} \\
2 & 2 &   \{ 1,2\}    &  0,1\\
2 & 3 &  \{ 3 \} &   0,1,-1 \\
2 & 4 & \{  4 \} &  0,1,0,-1 \\
2 & 6 &  \{  6 \} & 0,1,1,0,-1,-1  \\
3 & 3&  \{ 1,3 \}  & 0,1,1\\ 
3 &4 &  \{ 1,4 \} &   0,1,2,1 \\
3 & 6 & \{ 1,6 \} &  0, 1, 3, 4, 3,1 \\
4 & 4  & \{ 1,2,4  \} &   0,1,b,\pm 1        \\ 
4 &  5 & \{  5   \} &   0,1,1,-1,-1 \text{ or }    0,1,-1,1,-1   \text{ or }    0,1,-1,-1,1  \\ 
4 & 6  & \{ 1,2,3  \} &   0,1,b,1-b,b,1        \\ 
4 & 6  & \{ 1,2,6   \} &    0,1,b,b+1,b,1        \\ 
4 & 6  & \{  3, 6   \} &  0,1,b,0,-b,-1 \text{ or } 0,1,b,-2,-b,1         \\ 
4 & 8 & \{  8   \} &  0,1,b,1,0,-1,-b,-1   \text{ or }      0,1,b,-1,0,-1,-b,1   \\ 
4 & 10   & \{  10   \} &         0,1,b,-1,-b,0,-1,-b,1,b \text{ or }  0, 1, 1, 1, 1, 0, -1, -1, -1, -1     \\ 
4 & 12  & \{  3, 4   \} &    0, 1, 0, -1, 0, 1, 0, -1, 0, 1, 0, -1      \\ 
4 & 12  & \{  4, 6   \} &    0, 1, 1, -2, -3, 1, 4, 1, -3, -2, 1, 1   \\ 
4 & 12  & \{  12    \} &     0, 1, b, 0, b, -1, 0, -1, -b, 0, -b, 1    \text{ or }   0, 1, b, 2, b, 1, 0, -1, -b, -2, -b, -1     \\ 
  \end{array}\] 
  
 Evidently one could go on further but there is an explosion of possibilities!  One can also work through possibilities with $u_0\ne 0$, with again many possibilities.

\subsection{Reducible characteristic polynomials} 

 In section \ref{sec: LDSk=3} we saw that the characteristic polynomial of any non-degenerate LDS  of  order $3$ is reducible. We guess that the following holds (though with limited evidence): \smallskip

 \noindent \textbf{(Hyp 1)}\, The characteristic polynomial of any non-degenerate LDS    of odd order $k>1$ is reducible.
 
 \smallskip
 
 If this is true we can also partly classify the degenerate cases:
 
 \begin{lemma} Suppose that (Hyp 1) is true. If $(u_n)_{n\geq 0}$ is a degenerate LDS    of odd order $k>1$ with irreducible  characteristic polynomial $f(x)\in \mathbb Z[x]$ then $u_{\ell n}=0$ for all $n\geq 0$ for some integer $\ell>1$, and $f(x)=g(x^\ell)$ for some irreducible $g(x)\in  \mathbb Z[x]$.
 \end{lemma}
 
 This is exactly what we saw in section \ref{sec: LDSk=3} in the case $k=3$.
 
\begin{proof} Let $G$ be  the Galois group of the splitting field extension of $f$. There exist distinct roots $\alpha$ and $\beta$ of $f(x)$ for which  $\alpha/\beta$ is a root of unity, say $\alpha=\zeta \beta$.  
Now $\alpha$ is Galois conjugate to every root of $f$, so if $\zeta_1\alpha,\dots,\zeta_m\alpha$ are the 
roots of $f$ that differ from $\alpha$ by a root of unity then there exists $\sigma_i\in G$ with $\alpha^{\sigma_i}=\zeta_i\alpha$. 
For any other root $\gamma$ there exists $\tau\in G$ with $\gamma=\alpha^\tau$ 
and then $\gamma^{\tau^{-1}\sigma_i \tau} = \alpha^{\sigma_i \tau}=(\zeta_i\alpha)^\tau=\zeta_i^\tau\gamma$, which is a  root of $f$ that differs from $\gamma$ by a root of unity. This argument implies that   we can partition the roots of $f$ into   sets of size $m$, such that the ratio of any two roots in the same set is a root of unity.   
 
 Now since $f$ has odd degree at least one root is real, call it $r$, and so there are   $\ell\geq m$ roots of $f(x)$ on the circle $|z|=|r|$.  Then Ferguson  \cite{Ferg} proved that 
$f(x)=g(x^\ell)$ for some $g(x)\in  \mathbb Z[x]$.
Moreover $g$ must be irreducible and of odd degree as $f$ is, and no two distinct roots of $g$ can differ by roots of unity.

Now $(u_{\ell n})_{n\geq 0}$ is a LDS and so either always equals $0$ or has  characteristic polynomial $g(x)$. 
In the latter case, (Hyp 1) implies that   $(u_{\ell n})_{n\geq 0}$ must have order 1, and so   $g(x)$ is linear, say $g(x)=x-a$ and so $f(x)=x^\ell-a$. Therefore $f$ has order $\ell$ and so $\ell=k$. But then the recurrence relation is
$u_{n+\ell}=au_n$ and so $u_{\ell n}=0$ for all $n\geq 0$ as $u_0=0$. This proves the claimed result.
 \end{proof}

For $k=2$   most LDS's have  irreducible characteristic polynomials since most Lucas sequences, like the Fibonacci numbers, come from an irreducible quadratic.
This is also true when $k=2^r$ since if $u_n$ is a product of $r$ Lucas sequences with roots say $\alpha_{i,1}, \alpha_{i,2}$  where each $(\alpha_{i,1}- \alpha_{i,2})^2=p_i$   a distinct prime, then the roots of  the characteristic polynomial of $u_n$ are $\alpha_{1,a_1}\cdots \alpha_{r,a_r}$ and these can all be mapped to each other by   Galois elements from each separate quadratic extension.

 \section{Large gcds among linear recurrence sequences} \label{sec: Future}

 
 In Theorem \ref{thm: main4} we've shown that if 
 $\gcd(u_n,u_{2n},\dots,u_{(k-1)n})>e^{\epsilon n}$ then there is an ``algebraic reason'' for it; that is, $u_n$ is divisible by an  exponentially-fast growing LDS.  One might guess this is true even if $(u_n,u_{2n})>e^{\epsilon n}$ but our techniques do not seem to be able to work with this weaker hypothesis (though see Proposition \ref{prop: Key result}).
 
 We might also guess that if $\max_{m<n} (u_m,u_n)>e^{\epsilon n}$  then  $u_n$ is divisible by an  exponentially-fast growing LDS.  Moreover if $(u_n)_{n\geq 0}$ and $(v_n)_{n\geq 0}$ are distinct linear recurrence sequences in $\mathbb Z$, and $\max_{m\leq n} (u_m,v_n)>e^{\epsilon n}$  then  $u_n$ and $v_n$ are both divisible by the same   exponentially-fast growing LDS.  Corvaja and Zannier's \eqref{eq: CZ-ineq} implies this holds if $(u_n)_{n\geq 0}$ and $(v_n)_{n\geq 0}$ are both LDS's:
 \medskip

\noindent \textbf{Theorem \ref{thm: main5}}
\emph{ If $(u_n)_{n\geq 0}$ and $(v_n)_{n\geq 0}$ are linear division sequences in the integers, and there are arbitrarily large integers $m,n$ for which
 $\gcd(u_m,v_n)>e^{\epsilon (m+n)}$ then there exists an exponential or polynomially generated LDS $(w_n)_{n\geq 0}$ in the integers and an integer $q\geq 1$ such that $w_n$ divides $(u_{qn},v_{qn})$.}
  \medskip
 
 \begin{proof}
 By Theorem \ref{thm: main} we know that $u_n$ and $v_n$ are each the products of   an exponential  LDS  and a finite number of polynomially generated LDS's up to a multiplicative factor of $n^{O(1)}$. If they have an exponential LDS in common then we are done.
 If not then we can disregard the exponential LDS since an exponential LDS and a polynomially generated LDS cannot have a fast growing gcd by Lemma \ref{PolyGenDivByp}. Since $u_n$ and $v_n$ are each the product of finite number of polynomially generated LDS there must be polynomially generated LDS's $U_n$ and $V_n$ dividing $u_n$ and 
 $v_n$ respectively, for which $\gcd(U_m,V_n)>e^{\epsilon' (m+n)}$  for some $\epsilon'$ depending only on $\epsilon$ and the number of LDS factors in $u_nv_n$.  Now any polynomially generated LDS is the lcm of factors of the form $\gamma^n-1$. Hence there exist $\gamma,\eta\in K$ for which $(\gamma^m-1)^*$ divides $u_m$, and 
 $(\eta^n-1)^*$ divides $v_n$ such that 
 \[
 \max_{m\leq n} N(\gamma^m-1,\eta^n-1)>e^{\epsilon'' n}.
 \]
 Taking $u=\gamma^m, v=\eta^n$ and $S$ to include all the of prime ideal factors of the numerator and denominator of $\eta$ and $\gamma$, we deduce that $\gamma^m$ and $\eta^n$ must be multiplicatively dependent by    \eqref{eq: CZ-ineq}.    Therefore there exists an integer $q$ such that $\gamma^q$ and $\eta^q$ are both powers of some 
 $\tau$, and $(\tau^n-1)^*$ divides both $u_{qn}$ and $v_{qn}$. By multiplying together a suitable set of conjugates of $(\tau^n-1)^*$ we deduce (as described in the introduction) that some Lucas sequence $(w_n)_{n\geq 0}$ (perhaps of order $>1$) divides both $u_{qn}$ and $v_{qn}$.
 \end{proof}

\subsection{Levin's generalization of Corvaja-Zannier's \eqref{eq: CZ-ineq} }
The great thing about \eqref{eq: CZ-ineq} is that it allows one to attack all the binomial factors of the linear recurrence sequence (and monomial factors are easy), which is how we obtained Theorem \ref{thm: main5} which focuses on products of monomials and binomials. 
  Aaron Levin \cite{Lev} recently provided a way to attack multinomial factors and this has had several interesting consequences:
  
Let $f(x_1,\dots,x_m)$,   $g(x_1,\dots,x_m)\in \overline{\mathbb Q}[x_1,\dots,x_m]$ be non-constant coprime polynomials, which do not both vanish at $(0,\dots,0)$. If $u_1,\dots,u_m$ are $S$-units for which
\begin{equation} \label{eq: CZ-ineq2}
|N((f(u_1,\dots,u_m),g(u_1,\dots,u_m)))| \gg_{S,\epsilon} \big( \max_i H(u_i) \big)^\epsilon
\end{equation}
(where $H(\cdot)$ is the height defined after the statement of \eqref{eq: CZ-ineq})
then $(u_1,\dots,u_m)$ must satisfy one of a finite number of sets of multiplicative equations of the form
\[
u_1^{a_{i,1}}\cdots u_m^{a_{i,m}} = A_i \text{ for } 1\leq i\leq \ell,
\]
 where each $a_{i,j}$ is an integer and each $A_i$ is an $S$-unit.\footnote{Technically, ``$(u_1,\dots,u_m)$
belongs to a finite union of translates of proper algebraic subgroups''.} (This was proved for $m=2$ by Corvaja and Zannier in  \eqref{eq: CZ-ineq}.)

 \subsection{Consequences of Levin's generalization  for linear recurrence sequences}
  
 Grieve and Wang \cite[Theorem 1.8(a)]{GrW} showed that  if $(u_n)_{n\geq 0}$ and $(v_n)_{n\geq 0}$ are any two   linear recurrences in $\mathbb Z$  for which $(u_n,v_n)>e^{\epsilon n}$ for  infinitely many $n$, then there exists 
a linear recurrence sequence $(w_n)_{n\geq 0}$ such that   $w_n$ divides $(u_n,v_n)$ for all $n$ in some arithmetic progression. And if $(u_n,v_n)>e^{\epsilon n}$ for all sufficiently large $n$ then $w_n$ divides $(u_n,v_n)$ for all $n\geq 1$.  (See also the inspiring works of Fuchs \cite{Fuc}.) Taking $v_n=u_{2n}$ we see that if
$(u_n,u_{2n})>e^{\epsilon n}$ for  all sufficiently large $n$ then there is  a linear recurrence $(w_n)_{n\geq 0}$ in $\mathbb Z$ which divides both $u_n$ and $u_{2n}$. Analogous to Theorem \ref{thm: main5} we guess that one can force $w_n$ to always divide $w_{2n}$ (but have not done so).

Levin \cite[Theorem 1.11]{Lev} show that if $(u_n)_{n\geq 0}$ and $(v_n)_{n\geq 0}$ are any two simple  linear recurrences in $\mathbb Z$, for which there are arbitrarily large $m\leq n$ with   $(u_m,v_n)>e^{\epsilon n}$ then there exists a linear recurrence sequence $(w_n)_{n\geq 0}$ in the integers, 
 and arithmetic progressions $an+b, cn+d$ such that $w_n$ divides $u_{an+b}$ and $v_{cn+d}$ for all integers $n\geq 1$.
 This does not extend to all linear recurrence sequences in $\mathbb Z$ (that is, those that are not simple) as one can see from the example:
\begin{equation} \label{eq: closemiss}
u_n=2^n-1, v_n=n2^n-1 \text{ so that }  u_{2^r+r} = v_{2^r};
\end{equation}
moreover, if $(a,c)=1$ then $(u_{at+b},v_{ct+d}) $ divides $(ct+d)^a -2^{bc-ad}$ which does not grow exponentially.
On the other hand, Luca \cite{FlLu} showed that if $u_n$ and $v_n$ have characteristic polynomials $(x-1)^q(x-a)^r$ and 
  $(x-1)^s(x-b)^t$ where $q,r,s,t\geq 1$ and $a$ and $b$ are multiplicatively independent integers then $(u_m,v_n)<e^{\epsilon n}$ whenever $m\leq n$.

Recently Xiao \cite{Xiao} showed that the only  examples with large gcds are like \eqref{eq: closemiss}: 
Given any two linear recurrences $(u_n)_{n\geq 0}$ and $(v_n)_{n\geq 0}$ there is a finite set of coprime positive integer pairs $(a_i,b_i)_{i=1,\dots,r}$ such that all sufficiently large solutions $m,n$ to $(u_m,v_n)>e^{\epsilon n}$ satisfy
$0<m-a_ik, n-b_ik \ll \log k$ for some integer $k$. It may well be that the only such pairs take the form
$a \ell^r +b r+c, A\ell^r +B r+C$ for constants $a,b,c,A,B,C$ and where $\ell$ is a characteristic root of both $(u_n)_{n\geq 0}$ and $(v_n)_{n\geq 0}$.


 \bibliographystyle{plain}

\appendix

\section{Applying the \v Cebotarev density theorem} \label{App A}

\centerline{by \textsc{Carlo Pagano}}
\medskip

We wish to show the following result:

\begin{prop} \label{prop; cebapple} Suppose that  $\alpha$ and $\beta$ are non-zero,   multiplicatively independent algebraic numbers in $K$. There exist primes $q$ and $r$ such that for a positive proportion of prime ideals $P$ of $K$ with $NP\equiv 1 \pmod {qr}$, we have  $\alpha$ is a $q$th power mod $P$, but not $\beta$, and $\beta$ is an $r$th power mod $P$, but not $\alpha$,
\end{prop}

Let $L$ be the splitting field extension over $K$ of the polynomial 
\[
f(x):=(x^q-\alpha)(x^q-\beta)(x^{r}-\alpha)(x^{r}-\beta)
\]
with Galois group $G$.  If there is an element of $G$ which acts trivially on the roots of the first and fourth polynomial, and as a full cycle on the roots of the second and third polynomial then the theorem follows from  \v Cebotarev's density theorem.

\subsection{Finding the primes $q$ and $r$}
 
\begin{lemma}  If  $1,\alpha$ and $\beta$ are non-zero multiplicatively independent algebraic numbers in $K$ then  there are only finitely many primes $p$ for which  $1,\alpha, \beta$  are 
multiplicatively dependent in $(K^*)/(K^*)^p$.
\end{lemma}

\begin{proof} Suppose that $\alpha^r\beta^s=\lambda^p$ for some prime $p$ with $r$ and $s$ not both $0 \pmod p$.
We factor the fractional ideals $(\alpha)=P_1^{a_1}\cdots P_k^{a_k}, (\beta)=P_1^{b_1}\cdots P_k^{b_k}$ into prime ideals with each $a_j,b_j\in \mathbb Z$ not both $0$,  so that each $ra_i+sb_i\equiv 0 \pmod p$.   
We suppose prime $p$ does not divide any non-zero $a_ib_j-a_jb_i$.

Taking any two  congruences $ra_i+sb_i\equiv 0 \pmod p$ we deduce that every $a_ib_j-a_jb_i\equiv 0 \pmod p$, since $(r,s)^T\not\equiv 0 \pmod p$. But then $a_ib_j-a_jb_i=0$ and so if $ma_1=nb_1$ with $(m,n)=1$ then
$ma_i=nb_i$ for all $i$.  Therefore  
$\gamma=\alpha^m\beta^{-n}$ is a unit, and if $r=\ell m$ then $p\nmid \ell$ so that $\gamma$ is a $p$th power.

If $u_1,\cdots,u_m$ is a basis for the units of $K$ of infinite order, then we can write
$\gamma=\zeta u_1^{c_1}\cdots u_m^{c_m}$ where each $c_i\in \mathbb Z$ and $\zeta$ is a root of unity. We suppose prime $p$ does not divide any non-zero $c_j$, but then $\gamma$ cannot be a $p$th power unless each $c_j=0$.
Therefore $\gamma=\alpha^m\beta^{-n}$ is a root of unity, say an $\ell$th root of unity, and so
$\alpha^{\ell m}=\beta^{\ell n}$, contradicting the hypothesis.
\end{proof}

 
Let $L_p=K(\alpha^{1/p},\beta^{1/p})$.
If $1,\alpha, \beta$  are multiplicatively independent in   $(K^*)/(K^*)^p$  then $[L_p:K]=p^2$ by Kummer theory.
 Let $\zeta_n=e^{2i\pi/n}$ be a primitive $n$-th root of unity; if $(n,\text{Disc}(K))=1$ then $\mathbb Q(\zeta_n)\cap K=\mathbb Q$ and  so $[K(\zeta_n):K]=\phi(n)$.   
 
For any primes $q\ne r$ with   $qr\nmid (q-1)(r-1)\text{Disc}(K)$, we see that 
\[ L=K(\alpha^{1/q},\beta^{1/q}, \zeta_q, \alpha^{1/r},  \beta^{1/r},\zeta_r)\]
 is generated by $L_q, L_r$ and $K(\zeta_{qr})$; since their extension degrees over $K$ are pairwise coprime 
 we deduce that $[L:K]=[L_q:K]\cdot [L_r:K] \cdot [K(\zeta_{qr}):K]= q^2(q-1)r^2(r-1)$.

\subsection{The key group structure}  
 Fix primes $q$ and $r$ as above. Let $H$ be the  finite group
\[
H:=(\mathbb{Z}/q\mathbb{Z})^2 \times (\mathbb{Z}/r\mathbb{Z})^2 \rtimes ((\mathbb{Z}/q\mathbb{Z})^* \times (\mathbb{Z}/r\mathbb{Z})^*),
\]
defined by the set of  quadruples $(a,A,b,B)\in H$ where the group law is given by  
\[
(a,A,b,B)\ast (a',A',b',B'):=(ba'+a,BA'+A,bb',BB').
\]

We now define a map $\phi: G\to H$; Let $u^q=U^r=\alpha$ and $v^q=V^r=\beta$. For each  $\sigma\in G$ we define $(a(\sigma),A(\sigma),b(\sigma),B(\sigma))\in H$, by  first obtaining 
$a(\sigma)$ and $b(\sigma)$ from
\[
\sigma(u)=\zeta_q^{a(\sigma)_1} u, \sigma(v)=\zeta_q^{a(\sigma)_2} v \text{ and } \sigma(\zeta_q)=\zeta_q^{b(\sigma)};
\]
and then defining $A(\sigma)$ and $B(\sigma)$ analogously  using the prime $r$.

\begin{prop} \label{injective}
The map $\phi$ is a  group isomorphism.
\end{prop}
\begin{proof} $\phi$ is evidently a group homomorphism. It is injective since 
if $\sigma\in \ker \phi$ then $\sigma(u)= u, \sigma(v)= v $ and $ \sigma(\zeta_q)=\zeta_q$, and similarly
$\sigma(U)= U, \sigma(V)= V$ and $ \sigma(\zeta_r)=\zeta_r$, so that $\sigma$ must be the identity map. 
We deduce that it is an isomorphism since $|G|= [L:K]=q^2(q-1)r^2(r-1)=|H|$.
\end{proof}

\subsection{Completing the proof}

Let $P$ be a prime ideal in $\mathcal{O}_K$ (the ring of integers of $K$) not dividing the discriminant. Fix a prime ideal $\widetilde{P}$ in $\mathcal{O}_L$, lying above $P$. There exists a unique  $\sigma:=\text{Art}(\widetilde{P},L/K)\in G$, the \emph{Artin element}, for which
\[
\sigma(\widetilde{P})=\widetilde{P} \text{ and } \sigma(x) \equiv x^{NP} \pmod {\widetilde{P}} \text{ for all } x\in \mathcal{O}_L.
\]
If $\widetilde{P}'$ is another prime ideal of $\mathcal{O}_L$ lying above $P$, then $\text{Art}(\widetilde{P}',L/K)$ 
is a conjugate  of $\text{Art}(\widetilde{P},L/K)$ in $G$. The  \v Cebotarev density theorem \cite[Th\'eor\`eme B.32]{Hara} implies that for any $\sigma\in G$, a positive proportion of   prime ideals $P$ of $\mathcal{O}_K$  are unramified in $L$, and have a prime ideal 
$\widetilde{P}$ of $L$, lying above $P$, for which $\text{Art}(\widetilde{P},L/K)=\sigma$.


\begin{proof} [Proof of Proposition \ref{prop; cebapple}]
Since  $((0,1),(1,0),1,1)\in H$ which is isomorphic to $G$, the \v Cebotarev density theorem implies that 
for a positive proportion of   prime ideals $P$ of $\mathcal{O}_K$ (unramified in $L$) we have 
$\sigma:=\text{Art}(\widetilde{P},L/K)=((0,1),(1,0),1,1)$  for some  prime ideal 
$\widetilde{P}$ of $L$, lying above $P$.

For the prime $q$ this implies that $\sigma(\zeta_q)=\zeta_q,  \sigma(u)=u$ and $\sigma(v)=\zeta_qv$.
Now $\sigma(\zeta_q)=\zeta_q$ implies that  $NP \equiv 1 \pmod {q}$ and so all the $q$th roots of unity in $\overline{\mathcal{O}_K/P}$  lie in  the finite field $\mathcal{O}_K/P$.

We have $u=\sigma(u)\equiv u^{NP}  \pmod {\widetilde{P}}$ and we get the same congurence for any $\widetilde{P}$ lying above $P$ after conjugating, so that $u^{NP} \equiv u \pmod P$. Therefore $u\in \mathcal{O}_K/P$, and so $\alpha$ is a $q$th power in $\mathcal{O}_K/P$; that is, mod $P$.

On the other hand  $\zeta_qv=\sigma(v)\equiv v^{NP}  \pmod {\widetilde{P}}$ and so $v^{NP}\not\equiv v \pmod P$, which means that $v\not \in \mathcal{O}_K/P$ and also $\zeta_q^i v\not \in \mathcal{O}_K/P$ since each $\zeta_1^i\in  \mathcal{O}_K/P$. Therefore none of the $q$th roots of $\beta$ are  in $\mathcal{O}_K/P$ so that $\beta$ is not a $q$th power in $\mathcal{O}_K/P$, and so mod $P$.

The analogous proof works for the prime $r$.
 \end{proof}

\end{document}